%% file: CZ_REV_160915.tex
\documentclass{amsart}

\theoremstyle{definition}

\theoremstyle{remark}

\numberwithin{equation}{section}

\usepackage{amsmath,amsfonts,amssymb}
\usepackage{xcolor}
\usepackage[extra,safe]{tipa}
\usepackage[applemac]{inputenc}

\def \H{\mathbf{H}}

\def \T{\mathbb{T}}
\def\my_c{c_\infty}

\def \gM{{\mathbf{M}}}

\def \z{{\mathbf{z}}}
\def \k{{\mathbf{k}}}

 \def\bxi{{\boldsymbol{\xi}}}
 \def\btheta{{\boldsymbol{\theta}}}

\def \cR{{\mathcal{R}}}
\def \cG{{\mathcal{G}}}

\def\leftB{[\![}
\def\rightB{]\!]}
\include{macro}

\def\diag{{{\rm diag}}}
\def\Tr{{{\rm Tr}}}

\def\0{{\mathbf{0}}}

\def\X{{\mathbf{X}}}
\def\bx{{\mathbf{x}}}
\def\x{{\mathbf{x}}}
\def\by{{\mathbf{y}}}
\def\y{{\mathbf{y}}}
\def\m{{\mathbf{m}}}

\def\gF{{\mathbf{F}}}

\def\K{{\mathbf{K}}}
\def\gR{{\mathbf{R}}}

\def\gF{{\mathbf{F}}}

\def\gD{{\mathbf{D}}}

 \def\bphi{{\boldsymbol{\phi}}}

\begin{document}

\title{
Martingale problems for some degenerate Kolmogorov equations}

\author{St\'ephane Menozzi}
\address{Laboratoire de Mod\'elisation Math\'ematique d'Evry (LaMME), Universit\'e d'Evry Val d'Essonne, 23 Boulevard de France 91037 Evry, France and Laboratory of Stochastic Analysis, HSE, Moscow.}
\email{stephane.menozzi@univ-evry.fr}


\date{\today}


\begin{abstract}
We obtain Calder\'on-Zygmund estimates for some  degenerate equations of Kolmogorov type with inhomogeneous nonlinear coefficients. We then derive the well-posedness of the martingale problem associated with related degenerate operators, and therefore uniqueness in law for the corresponding stochastic differential equations. Some density estimates are established as well. 
\end{abstract}

\maketitle

\mysection{Introduction}
\label{INTRO}
\subsection{Statement of the problem}

 
Consider the following system of Stochastic Differential Equations (SDEs in short)
\begin{equation}
\label{SYST}
\begin{array}{l}
\displaystyle dX_t^1 = F_1(t,X_t^1,\dots,X_t^n) dt + \sigma(t,X_t^1,\dots,X_t^n) dW_t,
\\
\displaystyle dX_t^2 = F_2(t,X_t^1,\dots,X_t^n) dt,
\\
\displaystyle dX_t^3 = F_3(t,X_t^2,\dots,X_t^n) dt,
\\
\displaystyle \cdots
\\
\displaystyle dX_t^n = F_n(t,X_t^{n-1},X_t^n) dt,
\end{array}
\quad t \geq 0,
\end{equation}
$(W_t)_{t \geq 0}$ standing for a $d$-dimensional Brownian motion, and each $(X_t^i)_{t \geq 0}$, $ i \in\leftB 1, n\rightB$, being $\R^d$-valued as well. 

From the applicative viewpoint, systems of type \eqref{SYST} appear in many fields. Let us for instance mention for $n=2$ stochastic Hamiltonian systems (see e.g. Soize \cite{soiz:94} for a general overview or Talay \cite{tala:02} and H\'erau and Nier \cite{hera:nier:04} for convergence to equilibrium). 
Again for $n=2$,  the above dynamics is used in mathematical finance to price Asian options (see for example 
\cite{baru:poli:vesp:01}). For $n \geq 2$, it appears in heat conduction models (see e.g. Eckmann et al. 
\cite{eckm:99} and Rey-Bellet and Thomas \cite{reyb:thom:00} when the chain of differential equations is forced by two heat baths).

Assume first that  the coefficients $(F_i)_{i\in \leftB 1,n\rightB}$ are Lipschitz continuous in space and that the diffusion matrix $a(t,.):=\sigma\sigma^*(t,.) $ is bounded.
If we additionally suppose that 
$a(t,.) $ and $(D_{x_{i-1}}F_i(t,.))_{i\in \leftB 2,n\rightB} $ are non-degenerate (weak H\"ormander condition) and H\"older continuous in space, with respective H\"older exponents in $(1/2,1] $ and $(0,1] $, some multi-scale Gaussian Aronson like estimates have been proved in \cite{dela:meno:10} for the density of \eqref{SYST}  uniformly on the time set $(0,T]$, for fixed $T>0$ (see Example 2 and Theorem 1.1 of that reference). Those results extend to the case of an arbitrary  H\"older exponent in $(0,1] $ for  $a(t,.)$  thanks to uniqueness in law arguments that have been investigated  in \cite{meno:10} through the well posedness of the martingale problem. This result is established exploiting specifically some regularizing effect of an underlying \textit{parametrix} kernel.

Anyhow, when studying the martingale problem, the natural framework is to consider non-degenerate continuous coefficients\footnote{This assumption yields even in the non-degenerate case, estimates in $L^q$ spaces for the density (see e.g. Chapter 9 in \cite{stro:vara:79}), whereas the H\"older continuity gives, still in a weak solution framework, pointwise controls (see e.g. Sheu \cite{sheu:91}).}.
In the special case $n=1$ if $a(t,.)$ is bounded and uniformly elliptic, i.e. \eqref{SYST} corresponds to a non-degenerate SDE, it is well known that the martingale problem associated with the generator $(L_t)_{t\ge 0} $ of \eqref{SYST}  is well posed as soon as the coefficient $F_1$ is bounded measurable 
and that $a(t,.) $ is continuous in space, see e.g. 
Stroock and Varadhan \cite{stro:vara:79}. 

The key ingredient consists in proving Calder\'on-Zygmund estimates.
In that framework, those estimates write as controls in $L^p$ norms of suitable singular integrals related to the Gaussian density of \eqref{SYST} when the diffusion coefficient is constant in space and $F_1=0$. These controls then allow, through an operator inversion, to derive the well posedness of the martingale problem, or from a PDE viewpoint of the Cauchy problem with $L^p$ source term, when the diffusion coefficient does not vary much. The case of a bounded drift can then be handled through a Girsanov transform. Eventually, the well posedness of the martingale problem is established under the sole continuity assumption on $a$ and boundedness on $b$ thanks to a localization procedure.


We refer to the monographs of Stein \cite{stei:70} or Gilbarg and Trudinger \cite{gilb:trud:83} for a presentation of the Calder\'on-Zygmund theory for non-degenerate elliptic equations. In that framework a more probabilistic approach is proposed in Bass \cite{bass:95}. We also mention the monograph of Coifman and Weiss \cite{coif:weis:71}, from which some Calder\'on-Zygmund estimates can be derived in some degenerate frameworks when there is an underlying \textit{homogeneous space}.


In this work, for $n>1$,  under the previous assumptions of non-degeneracy and continuity on $a$ and a weak H\"ormander condition of the $(D_{x_{i-1}}F_i(t,.)_{i\in \leftB 2,n\rightB}) $, we are interested in proving the well-posedness of the martingale problem for the generator $(L_t)_{t\ge 0} $ of \eqref{SYST}.

To achieve this goal, we will establish Calder\'on-Zygmund estimates for a singular Gaussian kernel derived from a suitable linearization of the degenerate system \eqref{SYST}, already used in \cite{dela:meno:10}, \cite{meno:10}. The linearization is here crucial since it provides the \textit{proxy} density on which some good controls can be established. Observe indeed that when  $\sigma $ is constant in space and the coefficients $F$ are linear and  s.t. $(D_{x_{i-1}}F_i)_{i\in\leftB 2,n \rightB} $ satisfy a weak H\"ormander condition, then the SDE has a multi-scale Gaussian density (see the seminal paper of Kolmogorov \cite{kolm:33}, Di Francesco and Polidoro \cite{difr:poli:06}, \cite{dela:meno:10}). Roughly speaking, the non degeneracy of $a$ and the H\"ormander assumption on the drift allow to say that the $i^{{\rm th}}$ component of the SDE \textit{feels} a noise whose typical scale corresponds to the one of the $(i-1)^{{\rm th}} $ iterated integral of the Brownian motion which is $t^{1/2+(i-1)}=t^{(2i-1)/2} $ at time $t$. On the other hand, the Gaussian density exhibits deviations w.r.t. the transport of the initial condition by the deterministic  differential system\footnote{corresponding to \eqref{SYST} when $\sigma=0 $} having unbounded coefficients. The multi-scale Gaussian densities of Kolmogorov type will play here the same role as the standard Gaussian one in the non-degenerate setting of \cite{stro:vara:79}.
%
%
%
%
%
%
%
%
%
%
%
%
%
%

Let us now mention that in the linear case, Calder\'on-Zygmund estimates have been obtained by Bramanti \textit{et al.} in \cite{bram:cupi:lanc:prio:09}, \cite{bram:cupi:lanc:prio:13}. Precisely, they consider an operator 
$${\mathcal A}:=\sum_{i,j=1}^{p_0}a_{ij}\partial_{x_i,x_j}+\bsum{i,j=1}^{N} b_{ij}x_i\partial_{x_j}, $$
where the  matrices $(a_{ij})_{(i,j)\in \leftB 1,p_0\rightB^2} $ are symmetric positive definite, constant in \cite{bram:cupi:lanc:prio:09}, with continuous variable homogeneous coefficients which do not vary much in \cite{bram:cupi:lanc:prio:13}, and the $(b_{ij})_{(i,j)\in \leftB 1,N\rightB^2}$ are s.t. ${\mathcal A} $ is hypoelliptic.
The authors then establish global $L^p$ estimates $p\in (1,+\infty) $ of the following type: $\exists c:=c(a,b,p_0,N,p), \ \forall u\in C_0^2(\R^N) $, 
$$\|\partial_{x_ix_j} u\|_{L^p(\R^N)}\le c\{\|{\mathcal A} u\|_{L^p(\R^N)}+\|u\|_{L^p(\R^N)} \}, \ (i,j)\in \leftB 1,p_0\rightB^2.$$
Weak (1-1) estimates are also obtained. 
It has been proved in \cite{lanc:poli:94} that for some suitable basis, the matrices $b$ have the same form as in \eqref{SYST}, i.e. in our setting the coefficient $F$ would write $F(x)=Bx $ where $B_{i,j}=0_{d\times d}$ for $j<i-1,\ (i,j)\in \leftB 1,n\rightB$.  
Hence, the operator ${\mathcal A}$ can be seen as a particular case of generator associated with \eqref{SYST}. 

The strategy in those works still consists in estimating suitable singular integrals\footnote{precisely second order derivatives w.r.t. to the non degenerate components} related to the Gaussian fundamental solution of $L={\mathcal A}-\partial_t$ that enjoys the previously described properties (multi-scale and unbounded transport). Even in the linear framework, when the matrix $B$ has strictly upper diagonal entries it is not possible to enter the Coifman and Weiss \cite{coif:weis:71} framework of homogeneous spaces. Here, the underlying homogeneous norm would be the one derived from the various time scales of the components, corresponding once again to those of the iterated integrals of the Brownian motion. Namely,
\begin{equation}
\label{NAT_METR}
\forall (s,x_1,\cdots,x_n)\in \R\times \R^{nd},\ \bar \rho(s,x_1,\cdots,x_n):=|s|^{1/2}+\sum_{i=1}^n |x_i|^{1/(2i-1)}.
\end{equation}
Now, for a given $i\in \leftB 1,n-1\rightB$, the entries $B_{i,j}\in \R^{d}\otimes \R^d,\ j\in \leftB i+1, n\rightB $ are associated with components that have negligible time scale, namely $t^{(2j-1)/2} $, w.r.t. to the current one of order $t^{(2i-1)/2}$ in small time. This property has been exploited thoroughly in \cite{dela:meno:10}, \cite{meno:10} and Section \ref{CZ_FORMAL} to derive pointwise estimates, but in the current framework it breaks the global homogeneity in 
\eqref{NAT_METR} when considering, for $(s,x),(t,y)\in \R\times \R^{nd} $,
$\bar {\mathbf d}((s,x),(t,y))=\bar \rho(t-s, x-\exp(-B(t-s)) y) $ which appears as a natural candidate to be a quasi-distance taking into account the transport. It can be shown that $\bar {\mathbf d}((s,x),(t,y)) $ and $\bar {\mathbf d}((t,y),(s,x)) $ are equivalent on quasi metric balls only (whereas they are actually globally equivalent when there are no strictly upper-diagonal contributions in $B$, see also Section \ref{CZ_FORMAL}).
This observation could lead to consider the associated metric balls as homogeneous spaces to be in the Coifman-Weiss setting. The problem with this choice is that it is not clear anymore that the (sub)-balls enjoy the doubling  property.
This is why the authors in \cite{bram:cupi:lanc:prio:09}, \cite{bram:cupi:lanc:prio:13} rely on  some specific estimates established by Bramanti \cite{bram:10} on possibly non doubling spaces (see  Section \ref{CZ_FORMAL} for details). Also, their analysis strongly relies on some underlying Lie group structure.  

Let us mention that for variable homogeneous coefficients (namely in $VMO_{loc}$ (resp. $C^\alpha $) w.r.t. the distance induced by the vector fields), local $L^p$ (resp. Schauder) estimates have been obtained by Bramanti and Zhu \cite{bram:zhu:13} following the same lines. 
Concerning the link between the $L^p$ estimates of \cite{bram:cupi:lanc:prio:13} and the well-posedness of the martingale problem, we can refer to the recent work of Priola \cite{prio:15} who introduces a rather general localization procedure that allows to extend the well posedness of the martingale problem from the case of \textit{almost constant} coefficients to the natural one of continuous coefficients. The contribution of \cite{prio:15}, w.r.t. to the classic localization results of Stroock and Varadhan (see e.g. Chapter 6.6 in \cite{stro:vara:79}), being the handling of rather general unbounded coefficients.

The main novelty in our approach consists in considering inhomogeneous coefficients and rather general non-linear drifts for the degenerate part of equation \eqref{SYST}. To this end we introduce a suitable kernel, for which we establish Calder\'on-Zygmund estimates, and do not exploit some underlying Lie group properties appearing in the quoted works and which fail in our setting. The key idea consists in viewing \eqref{SYST} as an ODE \textit{perturbed} by a noise. This naturally yields to consider balls 
that are build around the characteristic lines of the ODE and reflect the multi-scale behavior of the process, where the various scales are once again those of the Brownian motion and its iterated integrals. This approach also allows, through a suitable localization procedure, to establish density estimates in $L^q$ (see equation \eqref{krylov}), which is the natural framework for diffusion coefficients that are just continuous. This is to our best knowledge the first result of this kind in the weak H\"ormander setting, even for a linear drift.

Let us also mention that, as a byproduct of our Cald\'eron-Zygmund estimates, it should be possible to get the well posedness of a decoupled degenerate BSDE having a H\"older continuous in space driver (or from the analytical viewpoint to develop a strong theory for semi-linear degenerate PDEs with H\"older in space source term) following the lines of Delarue and Guatteri \cite{dela:guat:06}. 

A challenging open problem would consist in extending the density estimates of equation \eqref{krylov} to degenerate It\^o processes of the form \eqref{SYST} where the diffusion coefficient would simply be measurable, bounded from above and from below. This would indeed give degenerate Krylov like estimates (see Sections 2 and 3 of Chapter 2 in \cite{kryl:87} in the uniformly elliptic setting) which would be the crux to get existence and uniqueness results for fully coupled degenerate Backward SDEs (or again to get a strong theory of quasilinear degenerate PDEs) with the previous type of drift (see Delarue \cite{dela:02}, \cite{dela:03} for an exposition of the strategy in the non-degenerate case). 


The article is organized as follows.  We state our assumptions and main results in Section \ref{ASS_AND_RES}. We then introduce in Section \ref{CZ_FORMAL} the degenerate Gaussian kernel for which we establish Calder\'on-Zygmund estimates, recalling formally how uniqueness can be derived from these controls when the coefficients do not vary much.
 In Section \ref{THE_CZ_SEC} we specify the various steps that lead to the Calder\'on-Zygmund estimates of Theorem  \ref{EST_CZ}. 
We then perform in Section \ref{APP_MT} a localization procedure  and give some local and global controls on the density from the previous estimates. This requires some careful extensions of the arguments of the non-degenerate framework,  see e.g. Sections 7.1, 7.2 and 9.1 in \cite{stro:vara:79}), exploiting again the characteristic lines of the underlying ODE. Section \ref{SEC_TEC} is the technical core of the paper and is devoted to the proof of the technical results of Section \ref{THE_CZ_SEC}.

\mysection{Assumptions and Main Results}
\label{ASS_AND_RES}

\subsection{Notations and Assumptions}

In what follows, we denote a quantity in $\R^{nd}$ by a bold letter: i.e. $\0$, stands for zero in $\R^{nd}$ and the solution $(X_t^1,\dots,X_t^n)_{t \geq 0}$ 
to \eqref{SYST} is denoted by $({\mathbf X}_t)_{t \geq 0}$. Introducing the embedding matrix $B$ from $\R^d$ into $\R^{nd} $, i.e. $B= (I_d , 0, \dots, 0)^*$, where ``$*$'' stands for the transpose, we rewrite \eqref{SYST} in the shortened form 
\begin{equation*}
d{\mathbf X}_t = {\mathbf F}(t,{\mathbf X}_t) dt+ B \sigma(t,{\mathbf X}_t) dW_t,
\end{equation*}
where ${\mathbf F}=(F_1,\dots,F_n)$ is an $\R^{nd}$-valued function.

With these notations the generator of \eqref{SYST} writes for all $ t\ge 0$:
\begin{equation}
\label{GEN}
\forall \varphi \in C_0^2(\R^{nd}),\ \forall \x\in \R^{nd},\ L_t \varphi(\x)= \langle \gF(t,\x) , \gD_\x\varphi(\x)\rangle +\frac 12 \tr(a(t,\x) D_{\x_1}^2 \varphi(\x)).
\end{equation}
Also, for a point $\x:=(\x_1,\cdots,\x_n)\in \R^{nd} $, we will often denote for all $i\in \leftB 1,n-1\rightB,\ \x^{i,n}:=(\x_i,\cdots,\x_n) $. The notation $|\cdot | $ stands for the (Euclidean) norm on $\R^m $ or $\R^m \otimes \R^m,\ m\in \{d,nd\}$.

Let us now introduce some assumptions concerning the coefficients of \eqref{SYST}. 
\begin{trivlist}
\item[\A{${\mathbf C}$}]   The diffusion coefficient $(a(t,.))_{t\ge 0} $  is bounded measurable and continuous in space, i.e.
$$\lim_{\y\rightarrow \x}\sup_{0\le s\le T}|a(s,\y)-a(s,\x)|=0 $$
for all $T>0$ and $\x\in \R^{nd}$.
\item[\A{UE}] There exists  $\Lambda \ge 1,\ \forall t\ge 0, \x\in \R^{nd},\ \xi \in \R^d, \ \Lambda^{-1}|\xi|^2\le \langle a(t,\x) \xi,\xi\rangle \le \Lambda |\xi|^2 $. 
\begin{trivlist}
\item[\A{S}] The $(F_i)_{i\in \leftB 1,n\rightB}$ are bounded measurable in time, globally Lipschitz continuous 
in space. 
Also the $(D_{\x_{i-1}} F_i)_{i\in \leftB 2,n\rightB} $ are $\eta $-H\"older continuous in space.
\end{trivlist}

\item[\A{ND}]  
There exists a closed convex subset ${\mathcal E}_{i-1} \subset GL_{\mathbf d}(\R)$
(set of invertible $d \times d$  matrices)
 s.t., for all
$t \geq 0$ and $(\x_{i-1},\dots,\x_n) \in \R^{(n-i+2)d}$, $D_{\x_{i-1}} F_i(t,\x_{i-1},\dots,\x_n)
\in {\mathcal E}_{i-1}$.
For example, ${\mathcal E}_i$, $ i \in \leftB 1, n-1\rightB$, may be a closed ball
included in $GL_{\mathbf d}(\R)$, which is an open set.
 \end{trivlist}
Assumptions \A{UE}, \A{ND} can be seen as a kind of (weak) H\"ormander condition. They allow to transmit the non degenerate noise of the first component to the other ones. Let us also recall that the last part of Assumption \A{ND} and the particular structure of $\gF(t,.)=(F_1(t,.),\cdots, F_n(t,.))$ yield that the $i^{{\rm th}} $ component of the system \eqref{SYST} has intrinsic time scale $(2i-1)/2, i\in \leftB 1, n\rightB $. This fact will be thoroughly used in our analysis (see Section \ref{CZ_FORMAL} for details). 
We notice that the coefficients may be irregular in time, see \A{S}. 
We say that assumption \A{A} is in force when \A{C}, \A{UE}, \A{S}, \A{ND} hold.


\subsection{Main Results}
\label{main_results}


Our main result is the following theorem.
\begin{THM}
\label{THM_M}
Under \A{A} the martingale problem associated with $(L_t)_{t\ge 0}$ in \eqref{GEN} is well-posed. That is, for every $\x\in \R^{nd}$, there exists a unique probability measure $\P $ on $C(\R^+, \R^{nd}) $ s.t. denoting by $(\X_t)_{t\ge 0} $ the canonical process, $\P[\X_0=\x]=1 $ and for all $\varphi\in C_0^{1,2}(\R^+\times \R^{nd},\R),\ \varphi(t,\X_t)-\varphi(0,\x)-\int_0^t (\partial_s+L_s)\varphi(s,\X_s)ds  $ is a $\P $-martingale. In particular, weak uniqueness in law holds for the SDE \eqref{SYST}.

Also, if the diffusion coefficient $a$ is uniformly continuous, the unique weak solution of \eqref{SYST} admits a density in the following sense. Letting  $P(s,t,\x,.)$ be the transition probability determined by $(L_t)_{t\ge 0}$, then for a given $T>0$,  almost all $t\in (s,T] $ and all $\Gamma \in {\mathcal B}(\R^{nd})$,   $P(s,t,\x,\Gamma) =\int_\Gamma p(s,t,\x,\y) d\y$.

More specifically, for any $f\in L^p([0,T]\times \R^{nd}),\ p>\frac{(n^2d+2)}{2} $, there exists $C_{\ref{krylov}}:=C_{\ref{krylov}}(T,p,\A{A}) $ s.t. for all $(s,\x)\in [0,T)\times \R^{nd} $:
\begin{equation}
\label{krylov}
|\E^{\P_{s,\x}}[\int_s^T f(t,\X_t)dt]|\le  C_{\ref{krylov}}(1+|\x|)\|f \|_{L^p([0,T]\times \R^{nd})},
\end{equation}
where $\E^{\P_{s,\x}}$ denotes the expectation w.r.t. $\P_{s,\x}[\cdot]:=\P[\cdot |\X_s=\x] $.
\end{THM}

\begin{REM}
Let us first emphasize that by duality, the previous control gives a bound for the density in $L^q([0,T]\times \R^{nd})$ where $q^{-1}+p^{-1}=1 $. Also,
the contribution in $\x $ in the r.h.s. of  \eqref{krylov} is specifically linked to the unboundedness of the drift term in \eqref{SYST}.
It derives from the localization procedure needed for the analysis. Namely, we are led to consider a suitable partition of $[0,T]\times \R^{nd} $ on which the coefficients of \eqref{SYST} satisfy a same given continuity constraint. Consider for instance a given point $(s,\x)\in [0,T]\times \R^{nd} $, and a given threshold $\varepsilon>0 $. It is then clear that, 
for the deterministic differential system deriving from \eqref{SYST} with dynamics
\begin{equation}
\label{FORWARD_FLOW}
\dot \btheta_{t,s}(\x)=\gF(t,\btheta_{t,s}(\x) ), t\ge s, \btheta_{s,s}(\x)=\x,
\end{equation}
which can somehow be seen as the \textit{mean} of the system, one has
$|\btheta_{t,s}(\x)-\x|\le  \kappa \int_s^t |\btheta_{u,s}(\x)|du\le \kappa (t-s)|\x|\exp(\kappa(t-s)),\ t\ge s $, where $ \kappa$ stands for the Lipschitz constant of $\gF $. Hence, one has $ |\btheta_{t,s}(\x)-\x|\le \varepsilon $ for $|t-s|\le C/|\x|, C:=C(\kappa,\varepsilon),$ which corresponds to the time step of the partition for a given $|\x|$. For a fixed $T>0$, $\lceil |\x|C^{-1}T \rceil$ is then an upper bound for the total number of time-steps.


\end{REM}

\mysection{``Frozen" Kernel and Formal derivation of uniqueness from Calder\'on-Zygmund estimates}
\label{CZ_FORMAL}

Assume \A{A} is in force. One of the main differences between the uniform H\"older continuity assumed in \cite{bass:perk:09} in the non degenerate case or in \cite{meno:10} for the current framework and the continuity statement of \A{C} is that in the first two cases no localization is needed. Indeed, the global H\"older continuity allows to remove globally the time singularities coming from the second order spatial derivatives of suitable Gaussian kernels arising in a parametrix like expansion of the density. In the current framework we first focus on the ``local case". As in the non-degenerate case, we assume the diffusion coefficient $a(t,.):=\sigma\sigma^*(t,.) $ of  \eqref{SYST} ``does not vary much" (see e.g. Chapter 7 of \cite{stro:vara:79}). 
Precisely, we first assume that  there exists  a measurable function $\varsigma:[0,T]\rightarrow {\mathcal S}_d $ (symmetric matrices of dimension $d$)
satisfying \A{UE} and such that
\begin{eqnarray}
\label{LOC_COND}
\varepsilon_a&:=&\sup_{0\le t \le T}\sup_{\x\in \R^{nd} } |a(t,\x)-\varsigma(t)|,
\end{eqnarray}
is small. In particular, we do not assume any \textit{a priori} continuity of $a$. The continuity assumption \A{C}
will actually allow, through a suitable localization procedure described in Section \ref{APP_MT}, to have  \eqref{LOC_COND} for all $\x_0\in \R^{nd} $ with  $\varsigma(t)=a(t,\x_0)$ on some neighborhood of $\x_0$.

To define the Gaussian kernel needed for the analysis we first introduce the backward deterministic differential system associated with \eqref{SYST}. For fixed $T>0,\ \y\in \R^{nd}$ and $t\in [0,T]$, we define:
\begin{equation}
\label{DET_SYST}
\overset{.}{\btheta}_{t,T}(\y)=\gF(t,\btheta_{t,T}(\y)),\ \btheta_{T,T}(\y)=\y .
\end{equation} 
\begin{REM}
Observe from \eqref{DET_SYST}, \eqref{FORWARD_FLOW} that $\btheta_{s,t}(\x) $ is well defined for all $s,t\in [0,T], \x\in \R^{nd} $. The associated differential dynamics in $s$ runs forward in time if $s\ge t $ and backward otherwise. 
\end{REM}

Consider now the deterministic ODE  
\begin{equation}
\label{eq:F:240409:3} 
\frac{d}{dt} \tilde{\bphi}_t = {\mathbf F}(t,\btheta_{t,T}(\by))
+ D {\mathbf F}(t,\btheta_{t,T}(\by))[\tilde{\bphi}_t
- \btheta_{t,T}(\by)], \quad t \geq 0,
\end{equation}
where for all $\x\in \R^{nd}$,\\ 
$$D\gF(t,\x)=\left (\begin{array}{ccccc}0 & \cdots & \cdots &\cdots  & 0\\
D_{\x_1}\gF_2(t,\x) & 0 &\cdots &\cdots &0\\
0 & D_{\x_2} \gF_3(t,\x)& 0& 0 &\vdots\\
\vdots &  0                      & \ddots & \vdots\\
0 &\cdots &     0      & D_{\x_{n-1}}\gF_n(t,\x) & 0
\end{array}\right) 
 $$ 
denotes the subdiagonal of the Jacobian matrix ${\mathbf D_\x \gF} $ at point $\x$.

Introduce now for a given $(T,\y)\in \R^{+*}\times \R^{nd}$, the resolvent $(\tilde \gR^{T,\y}(t,s))_{s,t\ge 0} $ associated with the partial gradients $(D {\mathbf F}(t,\btheta_{t,T}(\by)))_{t \ge0} $ which satisfies for $(s,t)\in (\R^{+})^2 $:
\begin{equation}
\begin{split}
\partial_{t}\tilde \gR^{T,\y}(t,s)&={ D \gF}(t,\btheta_{t,T}(\y))\tilde \gR^{T,\y}(t,s) , 
\ \tilde \gR^{T,\y}(s,s)=I_{nd\times nd},\\
\partial_{s}\tilde \gR^{T,\y}(t,s)&=-\tilde \gR^{T,\y}(t,s){ D \gF}(s,\btheta_{s,T}(\y)), 
\ \tilde \gR^{T,\y}(t,t)=I_{nd\times nd}.
\end{split}
\label{DYN_RES}
\end{equation}
Note in particular that since the partial gradients are subdiagonal $ {\rm det}(\tilde \gR^{T,\y}(t,s))=1$.

Setting as well for $0\le s, t\le T $,
\begin{eqnarray*}
\tilde \m^{T,\y}(s,t)&:=&\int_{s}^t \tilde \gR^{T,\y}(t,u)\bigl({\mathbf F}(u,\btheta_{u,T}(\by))
- D {\mathbf F}(u,\btheta_{u,T}(\by)) \btheta_{u,T}(\by)\bigr) du,
\end{eqnarray*}
and denoting by $(\tilde{\btheta}_{t,s}^{T,\by})_{t, s  \geq 0}$
the flow associated with \eqref{eq:F:240409:3}, i.e. $\tilde{\btheta}_{t,s}^{T,\by}(\bx)$ is the value of $\tilde{\bphi}_t$ when $\tilde{\bphi}_s = \bx$, we thus derive:
\begin{equation}
\label{AFFINE}
\begin{split}
\tilde{\btheta}_{t,s}^{T,\by}(\bx)&=\tilde \gR^{T,\y}(t,s)\x
 \\
 &\hspace{5pt}
+\int_{s}^t \tilde \gR^{T,\y}(t,u)\bigl({\mathbf F}(u,\btheta_{u,T}(\by))
- D {\mathbf F}(u,\btheta_{u,T}(\by)) \btheta_{u,T}(\by)\bigr) du\\
 &
=\tilde \gR^{T,\y}(t,s)\x+\tilde \m^{T,\y}(s,t).
\end{split}
\end{equation}
Note that the flow is affine.

We now introduce for all $0\le s<t,\ (\x,\y)\in (\R^{nd})^2 $ the kernel: 
\begin{eqnarray}
\label{DEF_KERN}
\tilde q(s,t,\x,\y):=\frac{1}{(2\pi)^{nd/2} \det(\tilde \K^\y(s,t))^{1/2}}\nonumber\\
\times\exp\left(-\frac 12\langle  \tilde \K^\y(s,t)^{-1}(\tilde \btheta^{t,\y}_{t,s}(\x)-\y),\tilde \btheta^{t,\y}_{t,s}(\x)-\y \rangle \right),
\end{eqnarray}
where $\tilde \K^\y(s,t):=\int_s^t \tilde \gR^{t,\y}(t,u)B \varsigma(u)B^* \tilde \gR^{t,\y}(t,u)^* du$. In other words, under \A{A}, denoting by $ \varsigma(u)^{1/2}$ the only subdiagonal matrix s.t. $\varsigma(u)^{1/2}[\varsigma(u)^{1/2}]^*=\varsigma(u)$,
$\tilde q(s,t,\x,\y) $ is the density at time $t$ and point $\y$ of the diffusion $(\tilde \X_u^{t,\y})_{u\in [s,t]} $ with dynamics:
\begin{eqnarray}
&&\hspace*{-.5cm}d\tilde \X_u^{t,\y}=[\gF(u,\btheta_{u,t}(\y))+ D\gF(u,\btheta_{u,t}(\y))(\tilde \X_u^{t,\y}-\btheta_{u,t}(\y))]du +B\varsigma(u)^{1/2} dW_u,\nonumber\\
&& \hspace*{.75cm}\forall u\in  [s,t],\
 \tilde \X_s^{t,\y}=\x. \label{FROZ}
 \end{eqnarray} 

Assumption \A{A} also guarantees that the covariance matrix $(\tilde \K^\y(s,t))_{0\le s<t} $ satisfies \textit{uniformly} in $\y\in \R^{nd} $ a \textit{good scaling property} in the sense of Definition 3.2 in \cite{dela:meno:10} (see also Proposition 3.4 of that reference). That is: for all fixed $T>0$, there exists $C_{\ref{GSP}}:=C_{\ref{GSP}}(T,\A{A})\ge 1$ s.t. for all $0\le s<t\le T $, for all $\y \in \R^{nd} $:
\begin{equation}
\label{GSP}
\forall \bxi \in \R^{nd},\   C_{\ref{GSP}}^{-1} (t-s)^{-1}|\T_{t-s} \bxi|^2\le \langle  \tilde \K^\y(s,t)\bxi,\bxi\rangle \le C_{\ref{GSP}} (t-s)^{-1}|\T_{t-s} \bxi|^2,
\end{equation}
where for all $t>0,\ \T_t={\rm diag}((t^i I_d)_{i\in \leftB 1,n\rightB} ) $ is a scale matrix. As pointed out in the introduction, equation \eqref{GSP} indicates that the $i^{{\rm th}} $  component of \eqref{FROZ} has characteristic time scale of order $(2i-1)/2$.

From \eqref{DEF_KERN} and \eqref{GSP}, we directly derive that for all $T>0$ there exists $C_{\ref{equiv_dens}}:= C_{\ref{equiv_dens}}(T,\A{A})\ge 1$ s.t. for all $0\le s<t\le T, \ (\x,\y) \in (\R^{nd})^2 $:
\begin{eqnarray}
\label{equiv_dens}
C_{\ref{equiv_dens}}^{-1}(t-s)^{-n^2d/2}\exp(-C_{\ref{equiv_dens}} (t-s)|\T_{t-s}^{-1}(\tilde \btheta^{t,\y}_{t,s}(\x)-\y)| ^2) \le  \tilde q(s,t,\x,\y)\nonumber\\
\le C_{\ref{equiv_dens}}(t-s)^{-n^2d/2}\exp(-C_{\ref{equiv_dens}}^{-1} (t-s)|\T_{t-s}^{-1}(\tilde \btheta^{t,\y}_{t,s}(\x)-\y)| ^2).
\end{eqnarray}


Now, Lemma 5.3 and Equation (5.11) from the proof of Lemma 5.5   in \cite{dela:meno:10} (see also Sections \ref{PROOF_CZ} and \ref{SEC_FLOW} below for details) give that there exists $C_{\ref{EQUIV_FL}}:=C_{\ref{EQUIV_FL}}(T,\A{A})\ge 1$ s.t.:
\begin{eqnarray}
C_{\ref{EQUIV_FL}}^{-1} |\T_{t-s}^{-1}(\x-\btheta_{s,t}(\y))|\le |\T_{t-s}^{-1}(\btheta_{t,s}(\x)-\y)| \le C_{\ref{EQUIV_FL}} |\T_{t-s}^{-1}(\x-\btheta_{s,t}(\y))|,\nonumber\\
C_{\ref{EQUIV_FL}}^{-1} |\T_{t-s}^{-1}(\x-\btheta_{s,t}(\y))|\le |\T_{t-s}^{-1}(\tilde \btheta_{t,s}^{t,\y}(\x)-\y)| \le C_{\ref{EQUIV_FL}} |\T_{t-s}^{-1}(\x-\btheta_{s,t}(\y))|,\nonumber\\
|D_{\x_j}\tilde q(s,t,\x,\y)|\le C_{\ref{EQUIV_FL}}(t-s)^{-j+1} |\T_{t-s}^{-1}(\tilde \btheta_{t,s}^{t,\y}(\x)-\y)|\tilde q(s,t,\x,\y),\ j\in \leftB 1,n \rightB. \nonumber \\
\label{EQUIV_FL}
\end{eqnarray}

On the other hand it is crucial to observe that $\tilde q(s,t,\x,\y)$ satisfies the following Backward Kolmogorov equation for all $(t,\y)\in \R^{+*}\times \R^{nd}$ :
\begin{equation}
\left(\partial_s +\tilde L_s^{t,\y}\right)\tilde q(s,t,\x,\y)=0 ,\ (s,\x) \in [0,t)\times \R^{nd},\ \tilde q(s,t,.,\y)\underset{s\uparrow t}{\longrightarrow }\delta_\y(.). \label{KOLM_BK}
\end{equation}
In the above equation we wrote:
\begin{eqnarray*}
\tilde L_s^{t,\y}\tilde q(s,t,\x,\y)&:=&\langle \gF(s,\btheta_{s,t}(\y))+ D \gF(s,\btheta_{s,t}(\y)) (\x-\btheta_{s,t}(\y)), \gD_\x \tilde q(s,t,\x,\y)\rangle\nonumber \\
&&+ \frac 12 \tr(\varsigma(s) D_{\x_1}^2 \tilde q(s,t,\x,\y) ).
\end{eqnarray*}
For the rest of the section we assume w.l.o.g. that $T\le 1$.
For $0\le s<T $ and a function $f\in C_0^\infty ([0,T)\times \R^{nd})$ we now define for all $\x\in \R^{nd} $:
\begin{equation}
\label{green}
\tilde G f(s,\x):=\int_s^T dt \int_{\R^{nd}}^{} \tilde q(s,t,\x,\y) f(t,\y) d\y .
\end{equation}

From \eqref{KOLM_BK} one easily gets that
\begin{eqnarray*}
\partial_s \tilde Gf(s,\x)+\tilde Mf(s,\x)=-f(s,\x), \ (s,\x)\in [0,T)\times \R^{nd},\ \tilde Gf(s,\cdot)\underset{s\uparrow T}{\longrightarrow } 0,
\end{eqnarray*}
with $\tilde Mf(s,\x):=\int_s^T dt \int_{\R^{nd}} d\y \tilde L_s^{t,\y} \tilde q(s,t,\x,\y) f(t,\y) $.

Hence,
\begin{eqnarray*}
\partial_s \tilde Gf(s,\x)+L_s \tilde Gf(s,\x)=(-f+Rf)(s,\x),\ (s,\x)\in  [0,T)\times \R^{nd},
\end{eqnarray*}
where $Rf(s,\x):=(L_s \tilde Gf-\tilde Mf)(s,\x)=\int_{s}^T dt \int_{\R^{nd}}d\y(L_s -\tilde L_s^{t,\y})\tilde q(s,t,\x,\y) f(t,\y)$.

Now, the local condition \eqref{LOC_COND} yields:
\begin{eqnarray}
|Rf(s,\x)|&\le & |N f(s,\x)|+\bsum{i=2}^n|D_{\x_i}R_i  f(s,\x) |+\frac {\varepsilon_a} 2 |D_{\x_1}^2 \tilde Gf(s,\x)|,\label{CTR_PREAL_INV}
\end{eqnarray}
where setting
\begin{eqnarray}
\gF^{t,\y}(s,\x)&:=&\bigl( \gF_1(s,\btheta_{s,t}(\y)),\gF_2(s,\x_1,(\btheta_{s,t}(\y))^{2,n}),\gF_3(s,\x_2,(\btheta_{s,t}(\y))^{3,n}),\cdots,\nonumber \\
 && \gF_n(s,\x_{n-1},(\btheta_{s,t}(\y))_{n})\bigr),\nonumber\\
Nf(s,\x)&:=&\bsum{i=1}^n \int_s^T dt \int_{\R^{nd}}d\y  \langle \bigl((\gF-\gF^{t,\y})(s,\x)\bigr)_i , \gD_{\x_i} \tilde q(s,t,\x,\y) \rangle  f(t,\y), \nonumber\\ 
\label{DEF_N}
 \end{eqnarray}
and for all $i\in \leftB 2,n \rightB$,\
\begin{eqnarray}
 R_i f(s,\x) &:=&\int_s^T dt \int_{\R^{nd}} d\y \tilde q(s,t,\x,\y) f(t,\y) \bigl\{\gF_i^{t,\y}(s,\x)\nonumber\\
 &&-\left[\gF_i(s,\btheta_{s,t}(\y))+D_{\x_{i-1}}\gF_i(s,\btheta_{s,t}(\y))(\x -\btheta_{s,t}(\y))_{i-1}\right] \bigr\}.
 \label{DEF_R} 
 \end{eqnarray}
\begin{REM}
\label{REM_LIN} Observe from the above equation that if, for all $i\in \leftB 2,n\rightB $, the function $\gF_i$ is linear w.r.t. to the $(i-1)^{{\rm th}} $ variable (component that transmits the noise), then for all $(s,\x)\in [0,T)\times \R^{nd},\ R_i f(s,\x)=0$.
\end{REM}

The terms $Nf$ in \eqref{DEF_N} and $(D_{\x_i}R_i f)_{i\in \leftB 2,n\rightB} $ in \eqref{DEF_R} do not have time singularities. Let us justify this point.

Using \eqref{equiv_dens}, \eqref{EQUIV_FL}, we derive from \eqref{DEF_N} that: 
\begin{eqnarray*}
|Nf(s,\x)|\le C \bsum{i=1}^n \int_s^T dt \int_{\R^{nd}}^{}d\y |(\x-\btheta_{s,t}(\y))^{i,n}| |D_{\x_i} \tilde q(s,t,\x,\y)||f(t,\y)|\nonumber \\
\le C\bsum{i=1}^n \int_s^T dt \int_{\R^{nd}}^{}d\y \left\{ \frac{|(\x-\btheta_{s,t}(\y))^{i,n}|}{(t-s)^{(2i-1)/2}}\right\} (t-s)^{1/2}|\T_{t-s}^{-1}(\tilde \btheta_{t,s}^{t,\y}(\x)-\y)|\tilde q(s,t,\x,\y)|f(t,\y)|\\
             \le C\int_s^Tdt \int_{\R^{nd}} \frac{d\y}{(t-s)^{n^2d/2}}\exp(-C^{-1}(t-s)|\T_{t-s}^{-1}(\tilde \btheta_{t,s}^{t,\y}(\x)-\y)|^2)|f(t,\y)|,\ C:=C(\A{A}).\nonumber\\
\end{eqnarray*}
 Now,  as a consequence of H\"older's inequality we derive that for all $p>1,\ p^{-1}+q^{-1}=1$:
\begin{eqnarray*}
|Nf(s,\x)|^p\le C(p,\A{A})T^{p/q} \int_s^T dt \int_{\R^{nd}} \frac{d\y}{(t-s)^{n^2d/2}}\nonumber\\
\exp(-C^{-1}(t-s)|\T_{t-s}^{-1}(\tilde \btheta_{t,s}^{t,\y}(\x)-\y)|^2)|f(t,\y)|^p.
\end{eqnarray*}
 The Fubini Theorem and \eqref{EQUIV_FL} then yields:
\begin{equation}
\|Nf\|_{L^p([0,T)\times \R^{nd})}\le C(p,\A{A}) T \|f\|_{L^p([0,T)\times \R^{nd})}.
\label{CTR_NF}
\end{equation}

From \eqref{equiv_dens}, \eqref{EQUIV_FL} we also derive from \eqref{DEF_R} that for all $i\in \leftB 2,n\rightB $:
\begin{eqnarray*}
|D_{\x_i}R_if(s,\x)|&\le& C \int_s^T dt \bint{\R^{nd}}^{} \frac{d\y }{(t-s)^{n^2d/2}} (t-s)^{-i+1/2}|(\x-\btheta_{s,t}(\y))_{i-1}|^{1+\eta}\nonumber\\
&&\times \exp\left(-C^{-1}(t-s)|\T_{t-s}^{-1}(\tilde \btheta_{t,s}^{t,\y}(\x)-\y)|^2\right) |f(t,\y)| \nonumber\\
&\le & C\int_s^T dt (t-s)^{-1+(i-\frac32 )\eta }\bint{\R^{nd}}^{}\frac{d\y }{(t-s)^{n^2d/2}} |f(t,\y)|\nonumber\\
&&\times \exp\left(-C^{-1}(t-s)|\T_{t-s}^{-1}(\tilde \btheta_{t,s}^{t,\y}(\x)-\y)|^2\right).
\end{eqnarray*}
Hence, by H\"older's inequality and for $p>2,\ p^{-1}+q^{-1}=1 $:
\begin{eqnarray}
|D_{\x_i}R_if(s,\x)|^p
\le C(p,\A{A})\left( \int_s^T dt(t-s)^{-1+\frac \eta 2}\right)^{p/q}\nonumber \\
\times \left( \int_s^T dt (t-s)^{-1+\frac \eta 2}\int_{\R^{nd}} \frac{d\y}{(t-s)^{n^2 d/2}}\exp\left(-C^{-1}(t-s)|\T_{t-s}^{-1}(\tilde \btheta_{t,s}^{t,\y}(\x)-\y)|^2\right) |f(t,\y)|^p\right),\nonumber \\
\|D_{\x_i}R_if\|_{L^p([0,T)\times \R^{nd})}\le C(p,\A{A}) T^{\eta/2}  \|f\|_{L^p([0,T)\times \R^{nd})},
\label{CTR_DR}
\end{eqnarray}
where the last control again follows from Fubini's theorem.

Now, the key tool to prove uniqueness for the martingale problem derives from the following Calder\'on and Zygmund type estimate for the Green function $\tilde G f$. Namely, we have the following theorem which is proved in Section \ref{SUBS_EST_CZ}.
\begin{THM}
\label{EST_CZ}
Assume that Assumption \A{A} is in force. Suppose also that 
that $T\in (0,T_0],\ T_0:=T_0(\A{A})\le 1 $. Then,  for all $p\in (1,+\infty)$ there exists $C_{\ref{EQ_EST_CZ}}:=C_{\ref{EQ_EST_CZ}}(T_0,p,\A{A})$ s.t. for all $
f\in L^p([0,T)\times \R^{nd}) $,
\begin{eqnarray}
\label{EQ_EST_CZ}
\|D_{\x_1}^2\tilde Gf\|_{L^p([0,T)\times \R^{nd})}&\le &C_{\ref{EQ_EST_CZ}} \|f\|_{L^p([0,T)\times \R^{nd})},  
\end{eqnarray}
where the Green function $\tilde Gf  $ is defined in \eqref{green} with the kernel $\tilde  q$ introduced in \eqref{DEF_KERN}. 
\end{THM}

\begin{REM}
Let us specify that the \textit{small time} condition appearing here is due to the fact that we are led to compare the flow $\btheta$ and its \textit{linearization}. It is clear that this procedure can be a good approximation in small time only.
\end{REM}

Hence, plugging \eqref{EQ_EST_CZ}, \eqref{CTR_DR} and \eqref{CTR_NF}  into \eqref{CTR_PREAL_INV} we derive that under \A{A}
, with the notations of \eqref{LOC_COND}, for $p>2$, 
\begin{equation}
\label{CTR_R_OPINV}
\|Rf\|_{L^p([0,T)\times \R^{nd})}\le  (\frac {\varepsilon_a}  2 C_{\ref{EQ_EST_CZ}}+C(p,\A{A}) T^{\frac \eta 2})\|f\|_{L^p([0,T)\times \R^{nd})}. 
\end{equation}
Thus, for $\varepsilon_a<C_{\ref{EQ_EST_CZ}}^{-1} $ and $T<(4C(p,\A{A}))^{-2/\eta}$, the operator $I-R $ admits a bounded inverse on $L^p([0,T)\times \R^{nd}) $, and formally $Gf(s,\x):=\tilde G\circ (I-R)^{-1} f (s,\x), \ (s,\x)\in [0,T]\times \R^{nd} $ solves the Cauchy problem:
\begin{equation*}
\begin{cases}
(\partial_t +L_t)u(t,\x)=-f(t,\x), (t,\x)\in [0,T)\times \R^{nd},\\
u(T,\x)=0,
\end{cases}
\end{equation*}
for $f\in L^p([0,T)\times \R^{nd}),\ p>(n^2d+2)/2 $. This last condition on $p$ is needed to give a pointwise sense to $Gf $. Observe indeed from H\"older's inequality and the upper-bound in \eqref{equiv_dens} that, for all $p>(n^2d+2)/2 $, there exists $C_{\ref{POINT_W_CTR}}:=C_{\ref{POINT_W_CTR}}(p,\A{A}) $, s.t. for all  $f\in L^p([0,T)\times \R^{nd})$, $(s,\x)\in [0,T)\times \R^{nd} $,
\begin{equation} 
\label{POINT_W_CTR}
|\tilde Gf(s,\x)|\le C_{\ref{POINT_W_CTR}}T^{1-(2+n^2d)/(2p)}\|f\|_{L^p([s,T)\times \R^{nd})}.
\end{equation}
The above control is an extension of Lemma 7.1.1. in \cite{stro:vara:79} in the non-degenerate case.
From the probabilistic viewpoint we will prove that there is only one probability $\P$ on $C([0,T],\R^{nd}) $ solving the martingale problem and therefore derive $Gf(s,\x)=\E^{\P_{s,\x}}[\int_s^T f(t,\X_t)dt] $, $(\X_t)_{t\in [0,T]} $ being the canonical process. 
A localization argument similar to the one 
in Priola \cite{prio:15}
then allows to extend the well posedness of the martingale problem under the sole continuity assumption \A{C} in \A{A}, i.e. without the local condition \eqref{LOC_COND}, see Section \ref{APP_MT}. 

\begin{REM}[Some Points about the Drift]
One can wonder if the assumptions on the drift $\gF$ can be weakened in order to conserve the global well-posedness of the martingale problem.
To answer this query one needs to consider separately $\gF_1$, associated with the non degenerate component, and the $(\gF_i)_{i\in \leftB 2,N\rightB}$, associated with the degenerate ones. 
\begin{trivlist}
\item[-] For $\gF_1$ the assumptions in \A{S} can be considerably weakened. Indeed if $\gF_1\in L^p([0,T]\times \R^{nd}), p>n^2d+2 $ then the martingale problem is still well posed provided the $(\gF_i)_{i\in \leftB 2,n\rightB}$ satisfy \A{S}. To see this, instead of \eqref{DET_SYST}, \eqref{eq:F:240409:3},
we consider the following dynamics to define the underlying Gaussian kernel:
\begin{equation*}
\overset{.}{\bar \btheta}_{t,T}(\y)=\bar \gF(t,\bar \btheta_{t,T}(\y)),\ \bar \btheta_{T,T}(\y)=\y,\ 
\end{equation*} 
where for all $\x\in  \R^{nd},\ \bar \gF(t,\x):=({\mathbf 0}, \gF_2(t,\x),\cdots,\gF_n(t,\x))$ and 
\begin{equation*}
\frac{d}{dt} \tilde{\bar \bphi}_t = \bar {\mathbf F}(t,\bar \btheta_{t,T}(\by))
+ D {\mathbf F}(t,\bar \btheta_{t,T}(\by))[\tilde{\bar \bphi}_t
- \bar \btheta_{t,T}(\by)], \quad t \geq 0,
\end{equation*}
i.e. we put the non degenerate drift to 0 in our proxy model. This would yield  in \eqref{CTR_PREAL_INV}, \eqref{DEF_R} the additional contribution
\begin{equation}
\begin{split}
|D_{\x_1}R_1f(s,\x)|:=|\int_{s}^T dt \int_{\R^{nd}}  d\y \langle \gF_1(s,\x),D_{\x_1}\tilde q(s,t,\x,\y) \rangle  f(t,\y)|\\
\le C |\gF_1(s,\x)| \int_s^T dt \int_{\R^{nd}} \frac{d\y}{(t-s)^{(n^2d+1)/2}}\exp\left(-C^{-1}(t-s)|\T_{t-s}^{-1}(\tilde \btheta_{t,s}^{t,\y}(\x)-\y)|^2\right) |f(t,\y)|,\label{THE_DRIFT}
\end{split}
\end{equation}
exploiting \eqref{EQUIV_FL} for the last inequality. We get the same time-singularity as in the non-degenerate case.
Thus, if $f\in L^p([0,T)\times \R^{nd}) $:
\begin{eqnarray*}
\|D_{\x_1}R_1f\|_{L^p([0,T)\times \R^{nd})}^p\\
\le \int_0^T ds\int_{\R^{nd}} d\x |\gF_1(s,\x)|^p \|f\|_{L^p([0,T)\times \R^{nd})}^p (\int_s^T dt (t-s)^{-(n^2d/2(q-1)+q/2)})^{p/q},\\
\|D_{\x_1}R_1f\|_{L^p([0,T)\times \R^{nd})} \le \|\gF_1\|_{L^p([0,T)\times \R^{nd})}\|f\|_{L^p([0,T)\times \R^{nd})}T^{\beta},
\end{eqnarray*}
where $\beta:=\frac12 (1-\frac{2+n^2d}{p})>0$ under the previous condition on $p$. Adding this contribution in \eqref{CTR_R_OPINV}, we derive that the operator inversion can still be performed provided $T $ is small enough. In the non-degenerate case, we refer to the work of Krylov and R\"ockner \cite{kryl:rock:05}, or Fedrizzi and Flandoli \cite{fedr:flan:11} for an alternative proof, for additional results concerning strong solvability for drifts in $L^p $-spaces.
Let us also mention that, as in the non-degenerate case, a bounded measurable drift $\gF_1$ does not alter the well posedness of the martingale problem. This can be seen from \eqref{THE_DRIFT} similarly to the previous computations for the terms $(D_{\x_{i-1}}R_i f)_{i\in \leftB 2,n \rightB} $.
\item[-] For $(\gF_i)_{i\in \leftB 2,n \rightB}$ the situation is more complicated. To have as ``proxy" a Gaussian process that satisfies in the whole space the \textit{good scaling property}  \eqref{GSP}, it seems rather natural to impose that $(D_{\x_{i-1}}\gF_i)_{i\in \leftB 2,n\rightB} $ are pointwise defined and non degenerate (H\"ormander like assumption).
In the current framework, a natural question consists in relaxing the H\"older continuity of the $(D_{\x_{i-1}}\gF_i)_{i\in \leftB 2,n\rightB} $. Assuming simply continuity on those functions would again lead to consider a singular integral operator. Observe indeed from the previous computations that for all  $i\in \leftB 2,n \rightB$,
\begin{eqnarray*}
|\gF_i^{t,\y}(s,\x)-\{\gF_i(s,\btheta_{s,t}(\y))+D_{\x_{i-1}}\gF_i(s,\x)(\x-\btheta_{s,t}(\y))_{i-1}\}||D_{\x_i}\tilde q(s,t,\x,\y)|\\
\le \frac{C}{(t-s)^{n^2d/2+1}}\exp\left(C^{-1}(t-s)|\T_{t-s}^{-1}(\tilde \btheta_{t,s}^{t,\y}(\x)-\y)|^2\right),
\end{eqnarray*}
which is the expected time singularity for the convolution kernel of a singular integral operator. Anyhow, it would be in this case rather delicate to establish the \textit{cancellation} property needed to complete the analysis, see also Proposition \ref{CZ_KER_SING} for the properties required on a Calder\'on-Zygmund kernel. The $\eta $-H\"older continuity of the $(\gF_i)_{i\in \leftB 2,n \rightB}$ is a sufficient condition to \textit{globally} get rid of the time singularity (see again \eqref{CTR_DR}).


Let us eventually mention that for $n=2$, under \A{UE}, \A{ND}, when $\sigma$ is Lipschitz continuous, $\gF$ is Lipschitz in $\x_1 $ and $D_{\x_1}\gF_2 $ is H\"older continuous, strong uniqueness has been established for \eqref{SYST} by Chaudru de Raynal \cite{chau:14} provided $ \gF_1, \gF_2$ are $\eta $-H\"older continuous in $\x_2 $ with $\eta>2/3 $.

\end{trivlist}
\end{REM}


\mysection{Derivation of the Calder\'on-Zygmund estimates}
\label{THE_CZ_SEC}
We assume \A{A} is in force and that $T\le T_0(\A{A})\le 1$. 

\subsection{Quasi Metric Structure and Covering}
To derive Theorem \ref{EST_CZ}, a crucial step  consists in considering a ``good" parabolic metric and in taking into account 
the unbounded transport term in \eqref{SYST}.
 In order to take into consideration our various time-scales, associated with the propagation of the noise into the system, we introduce the following metric:
\begin{equation}
\label{metric}
\forall (t,\x)\in \R\times \R^{nd},\ \rho(t,\x):=|t|^{1/2}+\bsum{i=1}^n |\x_i|^{1/(2i-1)}.
\end{equation}
\begin{REM}
\label{Homogeneity}
Recalling the definition of the scale matrix $\T_t:=\diag( (t^i I_d)_{i\in \leftB 1,n \rightB }) $, $t\ge 0 $, we can now observe that 
$\x\in \R^{nd} \mapsto \rho(t,t^{-1/2}\T_t \x)$ is $1/2$ homogeneous in the time variable, i.e. $\rho(t,t^{-1/2}\T_t \x)=t^{1/2}\rho (1,\x)$.
\end{REM}

The metric introduced in \eqref{metric} is similar to the one appearing 
 in \cite{bram:cupi:lanc:prio:09}, \cite{bram:cupi:lanc:prio:13} for $L^p $ regularity.  
 
 Introducing now the strip $S:=[-T,T]\times \R^{nd} $, we then define for $\big( (s,\x),(t,\y)\big)\in S^2$ the quasi-distances:
 \begin{eqnarray}
\label{distance}
  {\mathbf d}((s,\x),(t,\y)):=\rho(t-s,\btheta_{t,s}(\x)-\y), \nonumber\\
  {\mathbf d}^*((s,\x),(t,\y)):={\mathbf d}((t,\y),(s,\x))=\rho(t-s,\x-\btheta_{s,t}(\y)),
 \end{eqnarray}
  with $\rho$ as in \eqref{metric}.
 We now define the ``balls" associated with the quasi-metric ${\mathbf d}$ (${\mathbf d}$-balls) in the following way:
 \begin{eqnarray}
\label{DEF_MAIN_BALLS}
 \forall (s,\x)\in S,\ \forall \delta>0,\nonumber\\
  B((s,\x),\delta):= \{(t,\y)\in S:  {\mathbf d}((s,\x),(t,\y)) 
 \le \delta \}.
 \end{eqnarray}
We mention that the natural extension of the balls considered in \cite{bram:cupi:lanc:prio:09}, \cite{bram:cupi:lanc:prio:13} would have been to consider ${\mathbf d}^*$ in the above definition. For this choice, in the linear, homogeneous case  $\btheta_{s,t}(\y):=\gR_{s-t}\y$, $\gR $ standing for the resolvent of the linear differential system deriving from \eqref{SYST}, which can indeed be seen as a group action. We choose here to follow the characteristic associated with the center of the ball, considering a \textit{metric tube} around it. Anyhow those choices are very close and locally equivalent, see Proposition \ref{PROP_QM}.

There is now, as in the previously mentioned works, a double difficulty, first the quasi-distance used to define the balls satisfies the quasi-triangle inequality only locally. A natural choice would then consist in considering singular integrals for the ``homogeneous space" associated to the balls of the above form, but in such case it is not clear that such balls, seen as homogeneous spaces, enjoy the doubling property, which is however satisfied on the whole strip $S=[-T,T]\times \R^{nd} $.



%

The first key-point is the following result.
\begin{PROP}
\label{PROP_QM}
Let $S:=[-T,T]\times \R^{nd}$. The space $(S,d,dtd\x) $ is a locally invariant quasi-metric space in the following sense: for a given $\Lambda \in (0,1] $
there exists a constant $C_{\ref{PROP_QM}}:=C_{\ref{PROP_QM}}(\A{A},T,\Lambda)>0$ s.t.
\begin{trivlist}
\item[a)] For all $(s,\x), (t,\y) \in S$,  if ${\mathbf d}((s,\x),(t,\y))\le \Lambda$ then 
$${\mathbf d}((t,\y),(s,\x))\le C_{\ref{PROP_QM}} {\mathbf d}((s,\x),(t,\y))=\rho(|t-s|,\btheta_{t,s}(\x)-\y)),$$ 
and for  $(\sigma,\bxi) \in S$ s.t. ${\mathbf d}((s,\x),(\sigma,\bxi))\le \Lambda $ and ${\mathbf d}((t,\y),(\sigma,\bxi))\le \Lambda $ then
$${\mathbf d}((s,\x),(t,\y))\le C_{\ref{PROP_QM}}({\mathbf d}((s,\x),(\sigma,\bxi))+{\mathbf d}((\sigma,\bxi),(t,\y))).$$
\item[b)] Every ${\mathbf d}$-ball in the sense of \eqref{DEF_MAIN_BALLS}
has positive and finite measure and every non-empty intersection of two balls has positive measure. 
\item[c)] There exists $R>0$ s.t. for $0<R_1<R_2\le R $ there exists $C:=C(R_1,R_2)$ s.t. for all $(s,\x)\in S $,
$$|B((s,\x),R_2)|\le C|B((s,\x),R_1)|, $$
where $|.| $ stands here for the Lebesgue measure of the balls. 
 \end{trivlist}
\end{PROP}
 
 \begin{REM}[General and Subdiagonal structure]
 \label{SUB_D_STRUCT}
Let us stress, as it will appear from the proof of Proposition \ref{PROP_QM} in Section \ref{APP_IS}, that for the general form of $\gF$ in the dynamics of $\btheta $, the constant $C_{\ref{PROP_QM}} $ appearing here depends on the specific radius, here $\Lambda \le 1$, chosen for the balls. In the following, we assume $\Lambda $ ``small enough" and refer to Section \ref{SEC_GROS_CALCULS} for a specific discussion on the choice of $\Lambda :=\Lambda(\A{A})$.

However, the proof also emphasizes that when the function $\gF $  has the following structure, $\gF_1(t,\x)=\gF_1(t,\x_1), \ \forall i\in \leftB 2,n\rightB,\  \gF_i(t,\x^{i-1,n})=\gF_i(t,\x_{i-1},\x_i)$ (subdiagonal case), then the constant $C_{\ref{PROP_QM}} $ does not depend on the radius (see Remark \ref{REM_SUBD}).
Hence, in this latter case, point a) of the proposition gives that the quasi-distances ${\mathbf d}$ and ${\mathbf d}^*$ involving respectively the forward and backward transport are actually equivalent. In such a case ${\mathbf d}$ is a usual quasi-distance in the sense of Coifman and Weiss \cite{coif:weis:71} and the strip $S$ can be seen as a homogeneous space.
\end{REM}

From Proposition \ref{PROP_QM} we can use Theorem 25 in \cite{bram:cupi:lanc:prio:09} that we now state in our specific case.
\begin{THM}[Covering Theorem]
\label{THM_COV}
For every $\delta_0>0$ and $K>1$ there exists $\delta\in (0,\delta_0)$, a positive integer $M$ and a countable set $\bigl( (s_i,\x_i)\bigr)_{i\in A}\subset S $ s.t.
\begin{trivlist}
\item[1.] $S=\bigcup_{i\in A} B((s_i,\x_i),\delta)$.
\item[2.] $\sum_{i\in A}\I_{B((s_i,\x_i),K\delta)}\le M^2$.
\end{trivlist}
\end{THM}

  \subsection{Singular kernel and associated estimates}   
\label{PROOF_CZ}

Let us first define for $s>0$, $\varsigma(-s):=\varsigma(s) $, i.e. we symmetrize the diffusion coefficient.
Fix now $T>0$ and introduce: $$\forall 
(s,t,\x,\y) \in \R^2 \times (\R^{nd})^2, k\bigl(s,t,\x,\y  \bigr) :=
\I_{t>s} D_{\x_1}^2 \tilde q(s,t,\x,\y).$$ 
From \eqref{AFFINE} and \eqref{DEF_KERN} a direct computation yields (see also the proof of Lemma 5.5 in \cite{dela:meno:10}): 
\begin{eqnarray}
\label{reg_kernel}
k(s,t,\x,\y)&=&
\I_{t>s} \left(- [\tilde \gR^{t,\y}(t,s)^*\tilde \K^\y(s,t)^{-1}\tilde \gR^{t,\y}(t,s)]_{11}\right. \nonumber\\
&&\left.+[\tilde \gR^{t,\y}(t,s)^*\tilde \K^\y(s,t)^{-1}(\tilde \btheta_{t,s}^{t,\y}(\x)-\y)  ]_1^{\otimes 2}\right) \tilde q(s,t,\x,\y).
\end{eqnarray}
In the above equation, for a matrix $\gM \in \R^{nd}\otimes \R^{nd}$ (resp. a vector $\z\in \R^{nd}$), the notation $[\gM]_{11}$ stands for the $d\times d $ submatrix $(\gM_{ij})_{(i,j)\in \leftB 1,d\rightB} $ (resp. $[\z]_1$ stands for the subvector of $\R^d, (\z_i)_{i\in\leftB 1,d \rightB} $). 

From \eqref{DEF_KERN}, \eqref{GSP} and the scaling Lemma \ref{scaling_lemma} (see also equations (5.10), (5.11) in \cite{dela:meno:10}), we  have that there exists $C:=C(T,\A{A})$ s.t.:
\begin{eqnarray*}
|[\tilde \gR^{t,\y}(t,s)^*\tilde \K^\y(s,t)^{-1}(\tilde \btheta_{t,s}^{t,\y}(\x)-\y)  ]_i|\le C\left(  (t-s)^{-i+1}|\T_{t-s}^{-1}(\tilde \btheta_{t,s}^{t,\y}(\x)-\y)|  \right),\\
|[\tilde \gR^{t,\y}(t,s)^*\tilde \K^\y(s,t)^{-1}\tilde \gR^{t,\y}(t,s)]_{11}+[\tilde \gR^{t,\y}(t,s)^*\tilde \K^\y(s,t)^{-1}(\tilde \btheta_{t,s}^{t,\y}(\x)-\y)  ]_1^{\otimes 2}|\\
\le C((t-s)^{-1}|I_d|+|\T_{t-s}^{-1}(\tilde \btheta_{t,s}^{t,\y}(\x)-\y)|^2),
\end{eqnarray*}
so that \eqref{GSP}, \eqref{reg_kernel} yield that  $\exists (c_{\ref{CTR_SING}},C_{\ref{CTR_SING}}):=(c_{\ref{CTR_SING}},C_{\ref{CTR_SING}})(T,\A{A}) $ s.t.
\begin{eqnarray}
\label{CTR_SING}
 |k(s,t,\x,\y)|&\le& C_{\ref{CTR_SING}}
 \I_{t>s} (t-s)^{-1} q_{c_{\ref{CTR_SING}}}(s,t,\x,\y),
 \end{eqnarray}
where for all $c>0$,
$$q_{c}(s,t,\x,\y):=\frac{c^{nd/2}}{(2\pi)^{nd/2}(t-s)^{n^2d/2}} \exp\left(-\frac{c}2 (t-s)|\T_{t-s}^{-1}(\tilde \btheta^{t,\y}_{t,s}(\x)-\y)|^2 \right).$$

Observe that this is the same order of singularity than in the non-degenerate case. This is anyhow expectable since we are considering the derivatives w.r.t. the non-degenerate variables. 
From equation \eqref{CTR_SING} we get that for all $\epsilon \in (0,1),\ \forall (i,j)\in \leftB 1,d\rightB^2,\ f\in L^\infty(\R\times \R^{nd})$, $(s,\x)\in \R^{1+nd} $,
\begin{equation}
\label{DEF_KIJ_EPS}
K_{ij}^\epsilon f(s,\x):=\int_{S\cap {\mathbf d}((s,\x),(t,\y))>\epsilon} dt d\y k_{ij}(s,t,\x,\y) f(t,\y),\ \forall (i,j)\in \leftB 1,d\rightB^2,
\end{equation}
is well defined.

\subsection{Proof of Theorem \ref{EST_CZ}}
\label{SUBS_EST_CZ}

The first thing to do consists in splitting the kernel into a singular and a non singular part observing that the singularity is \textit{diagonal}. Specifically, for a given fixed $\delta>0$ there exist $(c,C):=(c,C)(T,\A{A},\delta)$  s.t. if for $\big( (s,\x),(t,\y)\big)\in S^2, \ {\mathbf d}((s,\x),(t,\y))=\rho(t-s,\btheta_{t,s}(\x)-\y)\ge \delta  $, then from \eqref{CTR_SING} we have:
\begin{equation}
\label{CTR_HD}
|k(s,t,\x,\y)|\I_{\rho(t-s,\btheta_{t,s}(\x)-\y)\ge \delta}\le C q_c(s,t,\x,\y).
\end{equation}
Indeed, from the definitions in \eqref{distance}, we have either $|t-s|^{1/2}\ge \delta/ (n+1) $ or that there exists 
$i\in \leftB 1,n\rightB$ s.t. $|(\btheta_{t,s}(\x)-\y)_i|^{1/(2i-1)}\ge \delta/(n+1) $. Thus:
\begin{trivlist}
\item[-] If $|t-s|^{1/2}\ge \delta/ (n+1)  $, there is no singularity in \eqref{CTR_SING} and \eqref{CTR_HD} holds.
\item[-] If there exists 
$i\in \leftB 1,n\rightB, \ |(\btheta_{t,s}(\x)-\y)_i|^{1/(2i-1)}\ge \delta/(n+1) $ we have from \eqref{CTR_SING}:
\begin{eqnarray*}
|k(s,t,\x,\y)|\I_{\rho(t-s,\btheta_{t,s}(\x)-\y)\ge \delta}\\
\le \frac{C_{\ref{CTR_SING}}(n+1)^2}{\delta^2} \frac{|(\btheta_{t,s}(\x)-\y)_i|^{2/(2i-1)}}{t-s}q_{c_{\ref{CTR_SING}}}(s,t,\x,\y)\\
\le\frac{C_{\ref{CTR_SING}}(n+1)^2}{\delta^2}((t-s)|\T_{t-s}^{-1}(\btheta_{t,s}(\x)-\y)|^2)^{1/(2i-1)}q_{c_{\ref{CTR_SING}}}(s,t,\x,\y)\\
\overset{\eqref{EQUIV_FL}}{\le} \frac{C_{\ref{CTR_SING}} C_{\ref{EQUIV_FL}} (n+1)^2}{\delta^2}((t-s)|\T_{t-s}^{-1}(\x-\btheta_{s,t}(\y))|^2)^{1/(2i-1)}q_{c_{\ref{CTR_SING}}}(s,t,\x,\y)\\
\overset{\eqref{EQUIV_FL}}{\le} \frac{C_{\ref{CTR_SING}} C_{\ref{EQUIV_FL}}^2(n+1)^2}{\delta^2}((t-s)|\T_{t-s}^{-1}(\tilde \btheta_{t,s}^{t,\y}(\x)-\y)|^2)^{1/(2i-1)}q_{c_{\ref{CTR_SING}}}(s,t,\x,\y),
\end{eqnarray*}
which again yields \eqref{CTR_HD} from the definition of $q_{c_{\ref{CTR_SING}}} $ after \eqref{CTR_SING}.
\end{trivlist}

Let us now write:
\begin{eqnarray}
K_{ij}^\epsilon f(s,\x)=\int_{S\cap {\mathbf d}((s,\x),(t,\y))>\epsilon} dt d\y k_{ij}(s,t,\x,\y)f(t,\y)\eta_{\delta}(t-s,\btheta_{t,s}(\x)-\y)\notag\\
+\int_{S\cap {\mathbf d}((s,\x),(t,\y))>\epsilon} dt d\y k_{ij}(s,t,\x,\y)f(t,\y)(1-\eta_{\delta})(t-s,\btheta_{t,s}(\x)-\y)\notag \\
:=K_{ij}^{\epsilon,d} f(s,\x)+K_{ij}^{\epsilon,\infty} f(s,\x),\label{THE_DEC_PRES_LOIN}
 \end{eqnarray}
where $\eta_{\delta} $ is a smooth non-negative cut-off function s.t. for all $(u,z)\in \R\times \R^{nd},\ \eta_\delta(u,z)=1 $ if $\rho(u,z)\le \delta $ and $\eta_\delta(u,z)=0 $ if $\rho(u,z)\ge 2\delta$.  
It is then easily seen from \eqref{CTR_HD}, 
that for all $f\in L^p(S),\ p\in [1,+\infty]$, 
\begin{equation}
\label{CTR_KER_HD_LP}
\|K_{ij}^{\epsilon,\infty}f\|_{L^p(S)}\le C_{\ref{CTR_KER_HD_LP}} \|f\|_{L^p(S)},\ C_{\ref{CTR_KER_HD_LP}}:=C_{\ref{CTR_KER_HD_LP}}(T,\A{A},\delta,p). 
\end{equation}
The singular part of the kernel requires a much more subtle handling. Setting $k_{ij}^{d}(s,t,\x,\y)=\eta_\delta(t-s,\btheta_{t,s}(\x)-\y)k_{ij}(s,t,\x,\y)$ we will prove the following proposition.
\begin{PROP}[Calder\'on-Zygmund Kernel]
\hspace*{.2cm}

\vspace*{.2cm}
\label{CZ_KER_SING}
\begin{trivlist}
\item[\textit{i)}] $\exists C_{\ref{CZ_KER_SING}}:=C_{\ref{CZ_KER_SING}}(T,\A{A},\delta), \ \forall \big( (s,\x),(t,\y) \big)\in S^2,\  |k_{ij}^{d}(s,t,\x,\y)|\le \frac{C_{\ref{CZ_KER_SING}}}{{\mathbf d}((s,\x),(t,\y))^{n^2 d+2}}.$
\item[\textit{ii)}] There exists a constant $c_{\ref{CZ_KER_SING}} $ s.t.  
\begin{eqnarray*}
|k_{ij}^{d}(s,t,\x,\y)-k_{ij}^{d}(\sigma,t,\bxi,\y)|\le C_{\ref{CZ_KER_SING}}
\{\frac{{\mathbf d}((s,\x),(\sigma,\bxi))^\eta}{{\mathbf d}((s,\x),(t,\y))^{n^2 d+2+\eta}}+\frac{1}{{\mathbf d}((s,\x),(t,\y))^{n^2d+2-\eta}}\}
, 
\end{eqnarray*}
$\forall  \big( (s,\x),(\sigma,\bxi) \big)\in S^2,\ c_{\ref{CZ_KER_SING}}{\mathbf d}((s,\x),(\sigma,\bxi))\le {\mathbf d}((s,\x),(t,\y))\le \Lambda$ for some $\Lambda \le 1 $ that will be specified later on.
\item[\textit{iii)}] The two previous ``standard estimates" hold for the adjoint kernel $$k_{ij}^{d,*}(s,t,\x,\y):=k_{ij}^{d}(t,s,\y,\x). $$
\item[\textit{iv)}] Cancellation Property:
\begin{eqnarray*}
\sup_{\epsilon>0}|\int_{{\mathbf d}^*( (s,\x),(t,\y))> \epsilon}k_{ij}^{d}(s,t,\x,\y)dtd\y|\\
+\sup_{\epsilon>0}|\int_{ {\mathbf d}^*( (s,\x),(t,\y))> \epsilon}k_{ij}^{d,*}(s,t,\x,\y)dtd\y|<+\infty.
\end{eqnarray*}
Also the limits: 
$$\lim_{\epsilon\rightarrow 0}\int_{ {\mathbf d}^*( (s,\x),(t,\y))> \epsilon}k_{ij}^{d}(s,t,\x,\y)dtd\y,\ \lim_{\epsilon\rightarrow 0}\int_{ {\mathbf d}^*( (s,\x),(t,\y))> \epsilon}k_{ij}^{d,*}(s,t,\x,\y)dtd\y $$
exist and are finite for almost every $(s,\x)\in S$.
\end{trivlist}
\end{PROP}
\begin{REM}
\label{REM_LAMBDA}
Let us emphasize that the constant $\Lambda $ in point \textit{ii)} is rather arbitrary since it is only needed to split the singular and non-singular part of the kernel. In practice, we will choose $\Lambda=2\delta$ small enough to suitably control the linearization \eqref{AFFINE} of the initial deterministic differential system \eqref{FORWARD_FLOW}. 
Let us as well mention that it is precisely this linearization error that also yields the second term in the r.h.s. of point \textit{ii)}, which is not \textit{usual} in the estimates for a singular kernel, but gives  an integrable singularity, w.r.t. the singular control of point \textit{i)}. This term could have been avoided by modifying the definition of the singular kernel at hand which would anyhow have seemed more complicated and less \textit{natural}.
We refer to Lemma \ref{LEMME_STAB} and the proof of Proposition \ref{CZ_KER_SING} in Section \ref{SEC_GROS_CALCULS} for details.
\end{REM}

The strategy is now to exploit those estimates to derive $L^p$ controls on the covering of $S$ with the ${\mathbf d}$-balls introduced in the Theorem \ref{THM_COV}. But to do so, we have to carefully check that the cancellation property appearing in Proposition \ref{CZ_KER_SING} for the whole space still holds on the metric balls. This property can be conserved thanks to a H\"older continuous cut-off as in Proposition 18 from \cite{bram:cupi:lanc:prio:09}. Namely, from Proposition \ref{CZ_KER_SING} it can be derived similarly to the previous reference that:
\begin{PROP}[Localized Cancellation]
\label{LOC_CANC} There exists a constant $R_0>0$ s.t. for $(s_0,\x_0)\in S,\ R\le R_0 $, if $a,b$ stand for two cut-off functions belonging to $C^\alpha(\R^{n+1},\R),\ \alpha>0$ and with support in $B((s_0,\x_0),R) $, then defining 
\begin{eqnarray*}
 k_{ij}^{d,{\rm loc}}(s,t,\x,\y)&:=&a(s,\x)k_{ij}^{d}(s,t,\x,\y)b(t,\y),\\
 k_{ij}^{d,*,{\rm loc}}(s,t,\x,\y)&:=&a(s,\x)k_{ij}^{d,*}(s,t,\x,\y)b(t,\y), 
 \end{eqnarray*}
we have that: 
\begin{trivlist} 
\item[-] $k_{ij}^{d,{\rm loc}},\ k_{ij}^{d,*,{\rm loc}}$ satisfy the first three points of Proposition \eqref{CZ_KER_SING} and for all $(s,\x)\in B((s_0,\x_0),R) $:
\begin{eqnarray*}
\sup_{\epsilon>0}|\int_{(t,\y)\in B((s_0,\x_0),R),\ {\mathbf d}^*( (s,\x),(t,\y))> \epsilon}k_{ij}^{d,{\rm loc}}(s,t,\x,\y)dtd\y|\\
+\sup_{\epsilon>0}|\int_{(t,\y)\in B((s_0,\x_0),R),\ {\mathbf d}^*( (s,\x),(t,\y))> \epsilon}k_{ij}^{d,*,{\rm loc}}(s,t,\x,\y)dtd\y|<+\infty.
\end{eqnarray*} 
\item[-] For almost all $(s,\x) \in B((s_0,\x_0),R) $ the limits
\begin{eqnarray*}
\lim_{\epsilon\rightarrow 0}\int_{(t,\y)\in B((s_0,\x_0),R),\ {\mathbf d}^*( (s,\x),(t,\y))> \epsilon}k_{ij}^{d,{\rm loc}}(s,t,\x,\y)dtd\y,\\ \lim_{\epsilon\rightarrow 0}\int_{(t,\y)\in B((s_0,\x_0),R),\ {\mathbf d}^*( (s,\x),(t,\y))> \epsilon}k_{ij}^{d,*,{\rm loc}}(s,t,\x,\y)dtd\y
\end{eqnarray*}
exist and are finite.
\end{trivlist}
\end{PROP}

 Now from Propositions \ref{CZ_KER_SING} and \ref{LOC_CANC} we derive from Theorem 3 in \cite{bram:10},  recalling from Proposition \ref{PROP_QM} that ${\mathbf d}$ and ${\mathbf d}^*$ are equivalent on ``quasi" metric balls, that for every $(s_0,\x_0)\in S,\ R\le R_0,\ p\in (1,+\infty)$, there exists a constant $C_{p,T,\A{A}}$ independent of $(s_0,\x_0) $ s.t. setting 
 $$T_{ij}f(s,\x)
 :=\lim_{\epsilon \rightarrow 0}\int_{(t,\y)\in B((s_0,\x_0),R),\ {\mathbf d}^*( (s,\x),(t,\y))> \epsilon}k_{ij}^{d,{\rm loc}}(s,t,\x,\y)f(t,\y)dtd\y, $$
$$\|T_{ij} f\|_{L^p(B((s_0,\x_0),R))}\le C_{p,T,\A{A}}\|f\| _{L^p(B((s_0,\x_0),R))}.$$

The covering Theorem \ref{THM_COV} then gives, similarly to the proof of Theorem 22 in \cite{bram:cupi:lanc:prio:09}, that setting $K_{ij}^d f(s,\x):=\lim_{\epsilon\rightarrow 0}K_{ij}^{\epsilon,d} f(s,\x) $, for every $p\in (1,+\infty)$, there exists a constant $C_{p,T,\A{A}}$ s.t. $\|K_{ij}^df\|_{L^p(S)}\le C_{p,T,\A{A}}\|f\| _{L^p(S)}$.
Combining this control with  equations \eqref{THE_DEC_PRES_LOIN} and \eqref{CTR_KER_HD_LP} eventually yields that
  $$\|D_{\x_1^i,\x_1^j}^2\tilde G f\|_{L^p(S)}\le C_{p,T,\A{A}}\|f\| _{L^p(S)},$$
 up to a modification of $C_{p,T,\A{A}}$.
This concludes the proof of Theorem \ref{EST_CZ} under \A{A}.

\mysection{Derivation of Theorem \ref{THM_M} from the Calder\'on-Zygmund estimates}
\label{APP_MT}
\subsection{Well posedness of the martingale problem}
Existence can be obtained by usual compactness arguments, see e.g. Theorem 6.1.7 in \cite{stro:vara:79} that can be adapted to the current framework. We will therefore focus on uniqueness.

The strategy is the following. We first prove the well-posedness of the martingale problem under the local condition \eqref{LOC_COND} for $T>0$ small enough. Still under \eqref{LOC_COND}, we then get rid of 
the small time constraint thanks to a chaining/gluing argument (see e.g. Chapter 6 in \cite{stro:vara:79}, and Chapter 4.6 in Ethier Kurtz \cite{ethi:kurz:97} or Chapter 4.11 in Kolokoltsov \cite{kolo:11} in the more general framework of c\`adl\`ag processes). We eventually derive the well-posedness on the whole space thanks to the localization results in Priola \cite{prio:15}.


\subsubsection{Well Posedness with Local Condition}
In this section we must adapt carefully the arguments in Chapter 7 of \cite{stro:vara:79}, who consider a zero or bounded drift term. In our model the drift is simply crucial. In the linear case, we can refer to the work of Priola \cite{prio:15}, who derived through resolvents the well posedness under \eqref{LOC_COND} for an arbitrary time. The nonlinear drift $\gF$ yields, for the linearization to be efficient, additional small time-constraints.

There are two key steps to derive uniqueness. 
The first one is the following Lemma.
\begin{LEMME}
\label{EST_UNIF_MP}
For $T>0$ small enough and under \eqref{LOC_COND}, whenever  $\P$ solves the martingale problem associated with $(L_t)_{t\in [0,T]} $, for every $p>(n^2d+2)/2 $, there exists $C_{\ref{EST_UNIF_MP}}:=C_{\ref{EST_UNIF_MP}}(p,\A{A})$  s.t. for all
 $(s,\x)\in [0,T]\times \R^{nd},\ f\in C_0^\infty([s,T]\times \R^{nd}) 
 $:
\begin{equation}
\bigg|\E^{\P_{s,\x}}\bigg[\int_s^T f(t,\X_t) dt\bigg]\bigg|\le C_{\ref{EST_UNIF_MP}}
\|f\|_{L^p([s,T)\times \R^{nd})}.\label{BD_LP_LOC}
\end{equation}
\end{LEMME}
Observe that this means that, for every solution of the martingale problem, 
the associated canonical process has a density whose $L^{q} $ norm, where $q^{-1}+p^{-1}=1, p,q>1 $, is uniformly controlled. 

\begin{PROP}
\label{PROP_FIN_EX_LOC}
Under \A{A}, if $T>0$ is small enough and the local condition \eqref{LOC_COND} is fulfilled, then the martingale problem is well posed on $[0,T]$ and if $\P_{s,\x} $ stands for the associated family of solutions, using the notations of Section \ref{CZ_FORMAL}, we have for all $(s,\x)\in [0,T)\times \R^{nd}, f\in C_0^\infty([s,T]\times \R^{nd})$:
$$\E^{\P_{s,\x}}\bigg[\int_s^T f(t,\X_t)dt\bigg]=\tilde G\circ (I-R)^{-1} f(s,\x).$$
Also, if there exists $\bar \x\in \R^{nd}$ s.t. 
$$\varepsilon_{a,\infty}:=\sup_{t\ge 0}\sup_{\x\in \R^{nd}} |a(t,\x)-a(t,\bar \x)|, $$
is small, then the martingale problem is well posed on $\R^+$.
\end{PROP}
The proof of Lemma \ref{EST_UNIF_MP} and Proposition \ref{PROP_FIN_EX_LOC} are postponed to Section \ref{TEC_SEC_MP_WP} to better emphasize the various steps required for the proof.

\subsubsection{Derivation of the Global Well-Posedness}

Now, without the local condition \eqref{LOC_COND}, the continuity assumed in \A{A} allows to localize. Precisely, 
it is  possible to consider a countable covering of $E:=[0,+\infty)\times \R^{nd}=\cup_{i\in \N} \cG_i^\delta$
where 
$$\cG_i^\delta:= [(s_i-\delta)\vee 0, s_i+\delta] \times B(\x_i,\delta) ,$$
for $(s_i,\x_i)\in \R^+\times \R^{nd}  $ and 
 $\delta>0$ s.t. $2\delta\le T $, for $T$ as in the previous paragraph, and  
\begin{eqnarray*}
\sup_{(t,\y)\in \cG_i^\delta}|a(t,\y)-a(t,\x_i)|\le \varepsilon_a,
\end{eqnarray*}
for a sufficiently small $\varepsilon_a $ (condition \eqref{LOC_COND}). This statement can be proved by compactness arguments. In the homogeneous case, we refer to the proof of Theorem 6 p. 263 in \cite{prio:15}.  Under \A{C}, those arguments extend to the current framework. 
Define then $\forall  (t,\y)\in \R^+\times \R^{nd},\ \tilde a_i(t,\y)=a(t,\y)\I_{(t,\y)\in \cG_i^\delta}+ (1-\I_{(t,\y)\in \cG_i^\delta})a(s_i,\x_i)$.
Denoting by $\tilde L_i $ the generator of \eqref{SYST} associated with diffusion coefficient $\tilde a_i$, since $\varepsilon_{\tilde a_i,\infty}\le \varepsilon_a $ which can be chosen small enough thanks to the continuity assumption, we have from Proposition \ref{PROP_FIN_EX_LOC} that the martingale problem is well posed for $ \tilde L_i$. Also, $\tilde L_i=L $ on $\cG_i^\delta $. The same holds if we add the time derivative in the operators (in order to take into account the inhomogeneity) or equivalently if we consider the time space processes. We can then conclude to global uniqueness from Theorem 26 in \cite{prio:15} taking as state space $E $. 

\subsubsection{Proof of Lemma \ref{EST_UNIF_MP} and Proposition \ref{PROP_FIN_EX_LOC}}
\label{TEC_SEC_MP_WP}

\textbf{Proof of Lemma \ref{EST_UNIF_MP}}
\begin{trivlist}
\item[-] \textbf{Step 1.} Let us introduce for given $(s,\x)\in [0,T]\times \R^{nd}$ 
and a measurable function $\varsigma:[0,T]\rightarrow {\mathcal S}^d $ satisfying \A{UE}
the process,
$$\bar \X_t^{s,\x,\varsigma}:=\x+\int_s^t \gF(u,\bar \X_u^{s,\x,\varsigma})du+\int_s^t B\varsigma^{1/2}(u) d\beta_u,$$
defined on some filtered probability space $({\mathcal E},\F,(\F_t)_{t\ge 0},\mu) $ on which $(\beta_u)_{u\ge 0} $ is a Brownian motion. 

The previous dynamics corresponds to a modification of \eqref{SYST} where we consider a deterministic non-degenerate non-homogeneous diffusion coefficient $\varsigma^{1/2}$. Observe that under \A{A}, we can derive from Theorem 1.1 in \cite{dela:meno:10} that for all $t>s$, $\bar \X_t^{s,\x,\varsigma} $ has a multiscale Gaussian density $\bar p^\varsigma(s,t,\x,\cdot) $. Precisely, $\bar p^\varsigma(s,t,\x,\y) $ satisfies \eqref{equiv_dens} with $\tilde \btheta_{t,s}^{t,\y}(\x) $ replaced by $\btheta_{t,s}(\x) $ solving \eqref{FORWARD_FLOW}.
\item[-] \textbf{Step 2}: Let $\pi^N([0,T]) $ be the partition of the interval $[0,T] $ with time-step $h:=T/N,\ N\in \N^* $. Define as well for $u\in [0,T] $, $\phi(u):=\{t_i:=ih, t_i\le u<t_{i+1}\} $ (i.e. $\phi(u) $ is the largest discretization time lower or equal than $u$). Let $(\Omega,\F,(\F_t)_{t\ge 0},\P) $ be a filtered probability space on which $\beta $ is a $d$-dimensional Brownian motion.
For given $(s,\x)\in [0,T]\times \R^{nd}$, let $(\bxi_t)_{t\in [s,T]} $ solve:
\begin{eqnarray*}
\bxi_t=\x+\int_s^t \gF(u,\bxi_u)du+\int_{s}^t B\Sigma_ud\beta_u,
\end{eqnarray*}
where the coefficient $\Sigma_u $ is $\F_{\phi(u)} $- measurable.
Conditioning iteratively w.r.t. the $(\F_{t_i})_{i\in \leftB h^{-1}\phi(s), h^{-1}T\rightB}$, it can be easily deduced that $(\bxi_t)_{t\in (s,T]} $ has a density. It can indeed be written as a convolution of the densities introduced in Step 1.
If we additionally assume that, setting $A_u=\Sigma_u\Sigma_u^*$, the following local condition holds: there exists a measurable function $\varsigma:[0,T]\rightarrow {\mathcal S}^d$  satisfying \A{UE} s.t.
\begin{equation}
\label{LOC_COND_AL}
\sup_{(u,\omega)\in [0,T]\times \Omega}\|\varsigma(u)-A_u(\omega)\|\le \varepsilon_a,
\end{equation}
for a small enough $\varepsilon_a $, the point is now to establish that  for all $(s,\x)\in [0,T]\times \R^{nd},f \in C_0^\infty([0,T)\times \R^{nd}) 
,\ p>(n^2 d+2)/2,
$ there exists  $C_p:=C_p(\A{A})$ s.t:
\begin{eqnarray*}
|\E[\int_s^T f(t,\bxi_t)dt]|\le C_p\|f\|_{L^p([s,T)\times \R^{nd})}.
\end{eqnarray*}
Recall from \eqref{green} 
that for $\tilde G f(s,\x)=\int_s^T dt\int_{\R^{nd}}d\y f(t,\y)\tilde q(s,t,\x,\y)$, where $\tilde q(s,t,\x,\y) $ stands for the density of $(\tilde \X_u^{t,\y})_{u\ge s} $ introduced in \eqref{FROZ},
we have:
\begin{eqnarray*}
\partial_s \tilde Gf(s,\x)+\tilde Mf(s,x)=-f(s,\x),\ (s,\x)\in [0,T)\times \R^{nd},
\end{eqnarray*}
for $\tilde Mf(s,\x)=\int_s^Tdt\int_{\R^{nd}}d\y \tilde L_s^{t,\y}\tilde q(s,t,\x,\y)f(t,\y)$.
Denoting by $L_t^\bxi $ the generator of $\bxi $ at time $t$, let us then write:
\begin{eqnarray}
\E[\tilde G f(s,\bxi_s)]
&=&-\E[\int_s^T \{\partial_t \tilde G f(t,\bxi_t)+\tilde M f(t,\bxi_t)\}dt]\nonumber\\
&&-\E^\P[\int_s^T  \{L_t^\bxi \tilde G f(t,\bxi_t)-\tilde Mf(t,\bxi_t)   \}dt]\nonumber\\
&=:&\E^\P[\int_s^T f(t,\bxi_t)dt]-\E^\P[\int_s^T  \bar R_\Sigma f(t,\bxi_{t})dt]\label{CTR_SCHEMA_EUL}.
\end{eqnarray}
For $(t,\z)\in [0,T)\times\R^{nd}$, the term $\bar R_\Sigma f(t,\z)=L_t^\bxi \tilde Gf(t,\z)-\tilde Mf(t,\z) $ can be controlled, thanks to \eqref{LOC_COND_AL}, exactly as the contribution $Rf(t,\z) $ in Section \ref{CZ_FORMAL} (see equation \eqref{CTR_R_OPINV}). 

From equation \eqref{CTR_SCHEMA_EUL}, we derive from \eqref{CTR_R_OPINV} and \eqref{POINT_W_CTR} that:
\begin{eqnarray*}
|\E^\P[\int_s^T f(t,\bxi_t)dt]|\\
\le [ C_{\ref{POINT_W_CTR}}T^{1-(n^2d+2)/(2p)}+\frac{\varepsilon_a}{2}C_{\ref{EQ_EST_CZ}}+C(p,\A{A})T^{\eta/2} ]\|f\|_{L^p([s,T)\times \R^{nd})}.
\end{eqnarray*}
We have thus proved the estimate \eqref{BD_LP_LOC} of Lemma \ref{EST_UNIF_MP} for processes of the form $(\bxi_t)_{t\ge 0} $.
\item[-]
\textbf{Step 3.} It now remains to extend the previous control to an arbitrary solution $\P $ of the martingale problem on $[0,T] $ under the local condition \eqref{LOC_COND}. To this end, for a	 progressively measurable $A:[0,T]\times \Omega\rightarrow {\mathcal S^d} $ satisfying \eqref{LOC_COND_AL},
we set for a family of non-negative mollifiers $(\zeta_\varepsilon)_{\varepsilon>0} $ with compact support on $\R$,  i.e. $\int_{\R}\zeta_\varepsilon(t)dt=1,\ \zeta_\varepsilon \in C_0^\infty(\R)$:
\begin{eqnarray*}
A_\varepsilon(t):=\int_0^t \zeta_\varepsilon(t-s)A(s)ds,\ 
\varsigma_\varepsilon(t):=\int_0^t \zeta_\varepsilon(t-s)\varsigma(s)ds,\ t\ge 0.
\end{eqnarray*}
For $N\ge 1$, define now $ A_{\varepsilon,N}(t):=A_\varepsilon(\phi(t)),\varsigma_{\varepsilon,N}(t):=\varsigma_\varepsilon(\phi(t)) $ where $\phi(t) $ denotes as in \textbf{Step 2} the largest discretization time lower or equal to $T$ for the time-step $h=T/N $. It is clear that the pair $(A_{\varepsilon,N}, \varsigma_{\varepsilon,N}) $ satisfies \eqref{LOC_COND_AL}.
Also,
$\lim_{\varepsilon\downarrow 0 }\lim_{N}\E^\P[\int_s^T\|A_{\varepsilon,N}(t)-A(t)\|^2 dt]=0.$
Let now $\X$ be the canonical process associated with $\P$ and set $A_u:=a(u,\X_u) $. Introduce then for $t\ge s$, $\beta_t:=\int_s^t A^{-1/2}(u)d\X_u^1-\int_{s}^t A^{-1/2}(u)\gF_1(u,\X_u)du $.  Then $\beta $ is a Brownian motion after time $s$ (see e.g. Theorem 4.5.1 in \cite{stro:vara:79} that extends to the current framework) and:
$$\X_t=\X_s+\int_s^t \gF(u,\X_u)du+\int_{s}^t B A^{1/2}(u)d\beta_u,\ t\ge s.$$
Set 
$$\bxi_t^{\varepsilon,N}=\X_s+\int_s^t \gF(u,\bxi_u^{\varepsilon,N})du+\int_s^tBA_{\varepsilon,N}^{1/2}(u)d\beta_u,\ t\ge s. $$
From \textbf{Step 2}, we derive that for all $f\in C_0^\infty([0,T)\times \R^{nd}) $:
$$|\E^\P[\int_s^T f(t,\bxi_t^{\varepsilon,N})dt]|\le C_p\|f\|_{L^p([s,T)\times \R^{nd})}.$$
On the other hand, since $\gF $ is globally Lipschitz, we get from Gronwall's Lemma and Doob's inequality
\begin{eqnarray*}
\E^\P[\sup_{t\in [s,T]}|\bxi_t^{\varepsilon,N}-\X_t|^2]\le C_p \E^\P[\int_s^T |A_{\varepsilon,N}(t)-A(t)|dt].
\end{eqnarray*}
Thus,
\begin{eqnarray*}
\big|\E^\P[\int_s^Tf(t,\X_t)dt]\big|=\lim_{\varepsilon\downarrow 0}\lim_N\big|\E[\int_s^T f(t,\bxi_t^{\varepsilon,N})dt]\big|\le C_p\|f\|_{L^p([s,T)\times \R^{nd})},
\end{eqnarray*}
which completes the proof.
\end{trivlist}

\textbf{Proof of Proposition \ref{PROP_FIN_EX_LOC}.}

Let now $\P$ be an arbitrary solution to the martingale problem. From Lemma \ref{EST_UNIF_MP} we get that  for all $ (s,\x) \in [0,T)\times \R^{nd}$, there exists $\varphi\in L^{q}([s,T)\times \R^{nd}), q\in (1,1+\frac{2}{n^2d}) $ s.t.:
$$\E^{\P_{s,\x}}[\int_s^Tf(t,\X_t)dt]=\int_s^T dt \int_{\R^{nd}} f(t,\y) \varphi(t,\y)d\y,$$
for all $f\in C_0^{\infty}([0,T]\times \R^{nd})$. On the other hand: 
\begin{eqnarray*}
\tilde Gf(s,\x)&=&\E^{\P_{s,\x}}[\int_s^T f(t,\X_t)dt]-\E^{\P_{s,\x}}[\int_s^T Rf(t,\X_t)dt]\\
&=&\int_s^T dt\int_{\R^{nd}} (I-R)f(t,\y)\varphi(t,\y)d\y.
\end{eqnarray*}
Since both sides are continuous w.r.t. the $L^p$ convergence and that $I-R$ is invertible on $L^p$ (see eq. \eqref{CTR_R_OPINV}) we conclude that:
\begin{eqnarray*}
\E^{\P_{s,\x}}[\int_s^T f(t,\X_t)dt]=\int_s^T dt\int_{\R^{nd}} f(t,\y)\varphi(t,\y)d\y=\tilde G\circ(I-R)^{-1}f(s,\x).
\end{eqnarray*}
This gives uniqueness on $[0,T] $.
Set now $\P^1=\P $, the unique solution of the martingale problem on $[0,T] $. If $\varepsilon_{a,\infty} $ is small, we have as well that, for all $i\ge 2 $, there is a unique solution $\P^i $ to the martingale problem on $[(i-1)T,iT]=:{\mathcal I}_i $, i.e. given $(s,\x)\in {\mathcal I}_{i}
\times \R^{nd}$ there is a unique $\P_{s,\x}^i $ on $C([0,\infty),\R^{nd}),{\mathcal B}(C([0,\infty),\R^{nd}))) $ s.t.
$\P_{s,\x}^i [\X_t=\x, 0\le t\le s]$ and $ f(\X_{t\wedge iT})-\int_s^{t\wedge iT} L_u f(\X_u)du$ is a $\P_{s,\x}^i $ martingale after time $s$ for all $f\in C_0^\infty(\R^{nd}) $. 
For fixed $(s,\x) \in [0,T]\times \R^{nd}$, define inductively $\Q^1_{s,\x}:=\P_{s,\x}^1 $, $\Q_{s,\x}^N:=\Q^{N-1}_{s,\x}\otimes_{(N-1)T} \P_{(N-1)T,\X_{(N-1)T}}^N,\ N\ge 2 $.
It can be shown from the well posedness of the martingale problem on each time interval $(\mathcal{I}_i)_{i\in \N^*}$ and Theorem 6.1.2 in \cite{stro:vara:79}, considering the $(iT)_{i\ge 1}$ as stopping times,  that the martingale problem is, under the local condition \eqref{LOC_COND} well posed on $[0,+\infty) $. Roughly speaking, the $\P^i$ can be ``glued" together yielding global well posedness.


\subsection{Existence of the Density and Associated Estimates}
The goal of this section is to prove the statement \eqref{krylov} of Theorem \ref{THM_M}.
To this end, we will need the following result which extends to our current degenerate setting Theorem 9.1.9 in \cite{stro:vara:79} in \textit{small} time.

\begin{THM}[Local existence of the density and associated estimates]
Assume $T\le T_0(\A{A})\le 1 $ as in Theorem \ref{EST_CZ}. 

\label{THM_LOC}
\begin{trivlist}
\item[-] If the diffusion coefficient $a$ is uniformly continuous, then, for $0\le s< T$,  
 the unique weak solution of \eqref{SYST} admits a density in the following sense. Letting  $P(s,t,\x,.)$ be the transition probability determined by $(L_t)_{t\ge 0}$, then for a given $T\in (0,T_0]$,  almost all $t\in (s,T] $ and all $\Gamma \in {\mathcal B}(\R^{nd})$,   $P(s,t,\x,\Gamma) =\int_\Gamma p(s,t,\x,\y) d\y$.
Also, for $q\in [1,2)$, the density $p$ satisfies:
\begin{eqnarray*}
\biggl(\int_s^T dt(t-s)^\alpha\int_{\R^{nd}} d\y|p(s,t,\x,\y)|^q \biggr)^{1/q}\le C_{\ref{THM_LOC}}^1(1+|\x|),
\end{eqnarray*}
where $\alpha=\left(\frac{(n^2d+2)}{2} \right)(q-1) $ and $ C_{\ref{THM_LOC}}^1:=C_{\ref{THM_LOC}}^1(T,q,\A{A},\delta_T)$,  denoting for all $\varepsilon>0 $, $\delta_T(\varepsilon):={\rm argmax}_{\zeta\in \R^+}\{ \sup_{s\in [0,T], |\x-\y|\le \zeta}|a(s,\x)-a(s,\y)|<\varepsilon\} $ the modulus of continuity of $a$. 
\item[-] For all $0\le s< T$  and $q\in [1,2)$, $\delta>0 $,
\begin{eqnarray*}
\biggl(\int_s^T dt\int_{\R^{nd}\backslash B^E(\btheta_{t,s}(\x),\delta)} d\y|p(s,t,\x,\y)|^q \biggr)^{1/q}\le C_{\ref{THM_LOC}}^2(1+|\x|),
\end{eqnarray*}
with $C_{\ref{THM_LOC}}^2:=C_{\ref{THM_LOC}}^2(T,q,\A{A},\delta_T,\delta)$, $B^E(\btheta_{t,s}(\x),\delta) $ standing for the Euclidean ball of $\R^{nd} $ with radius $\delta $ and center $\btheta_{t,s}(\x) $, recalling $\btheta_{t,s}(\x)=\x+\int_s^t \gF(u,\btheta_{u,s}(\x))du $ (i.e. $\btheta_{t,s}(\x) $ is the solution at time $t$ of the deterministic differential system associated with \eqref{SYST} starting from $\x$ at time $s$).
\end{trivlist}
\end{THM}

\begin{REM}
There are three differences w.r.t. to  Theorem 9.1.9 in \cite{stro:vara:79}. First, the norm of the initial point in the r.h.s. of the above controls is due to the transport by unbounded coefficients. Second, the small time constraint follows from our linearization strategy employed to derive the Calder\'on-Zygmund estimates of Theorem \ref{EST_CZ}. At last, the upper bound on $q$ comes from the control on the remainders in \eqref{CTR_DR}. When the system is linear w.r.t. the components that transmit the noise (see Remark \ref{REM_LIN}) this constraint disappears and the result of Theorem \ref{THM_LOC} hold for $q\in [1,+\infty)$.
\end{REM}

We provide below the principal lines needed to adapt the proof of Theorem 9.1.9 in \cite{stro:vara:79}, stressing which specific modifications  
are needed in the degenerate case and mainly concern the \textit{localization} arguments. Once again the key idea is to localize along the characteristic lines associated to the deterministic differential system instead of using spatial balls only as in \cite{stro:vara:79}. Observe anyhow that, when the drift is bounded, the product of the time interval and the spatial ball can be seen as a tube along a \textit{characteristic line}. Indeed,  if the drift is $0$ then the deterministic differential system does not leave its initial condition;  if it is bounded, the image of a spatial ball by the deterministic system
will stay uniformly in time in a ball whose radius only depend on the bound of the drift, the final time and the initial radius, \textbf{but not} on the points of the initial ball. 

\subsubsection{Controls for Slowly Varying Coefficients} 
We use here freely the notations of Section \ref{CZ_FORMAL} for the operators $\tilde G, R$ (see equations \eqref{green}-\eqref{DEF_R}).  
Also, in order to keep notations close to those in \cite{stro:vara:79}, we introduce for $r>2$, the class ${\mathcal A}(r,T) $ of measurable coefficients $a:\R^+\times \R^{nd}$ satisfying \A{UE} and $\gF $ satisfying \A{ND}, \A{S}, for which  there exists $\x_0\in\R^{nd} $ s.t. with the notations of \eqref{LOC_COND}, 
for all $\rho\in [r,\frac{n^2d+4}{2}\vee r] $, 
$ (\frac {\varepsilon_a}  2 C_{\ref{EQ_EST_CZ}}(T,\rho,\A{A})+C(\rho,\A{A}) T^{\frac \eta 2})<3/4$ so that from \eqref{CTR_R_OPINV}, we have that $(I-R)^{-1} $ is consistent as   bounded operator from $L^\rho([0,T]\times \R^{nd}) $ into itself. In particular this imposes that $T\le 1$ is sufficiently small.
Setting then $K :=\tilde G\circ(I- R)^{-1} $, we thus derive that it is consistent as bounded operator from $L^\rho([0,T]\times \R^{nd}) $ into $L^\sigma([0,T]\times \R^{nd}) $ provided that
$$0\le \frac 1\rho-\frac 1 \sigma<\frac{2}{n^2d+2}.$$
Here comes the first Lemma emphasizing some regularizing effects of $K$ which can be derived similarly to Lemma 9.1.2 in \cite{stro:vara:79}.
\begin{LEMME}
\label{SOOTH_INCREASE_MAP}
Let $a,\gF \in {\mathcal A}(r,T)$. Then, for $N=\lceil \frac{n^2d+2}{2}\frac{1}r\rceil$, $K^{N+1}$ maps $L^r([0,T]\times \R^{nd})$ into  $C_b([0,T]\times \R^{nd}) $ (space of real valued bounded continuous functions), i.e. it is $ L^r$-strong Feller. Precisely, for all $(s,\x)\in [0,T]\times \R^{nd} $:
$$|K^{N+1}f(s,\x)|\le C_{\ref{SOOTH_INCREASE_MAP}}\|f\|_{L^r([0,T]\times \R^{nd})}, $$
where $C_{\ref{SOOTH_INCREASE_MAP}}:=C_{\ref{SOOTH_INCREASE_MAP}}(T,r,\A{A})$.
\end{LEMME}
From Lemma \ref{SOOTH_INCREASE_MAP}, Lemma 9.1.3 in \cite{stro:vara:79} and the well posedness of the martingale problem, denoting by $P(s,t,\x,.)$ the associated transition function, one then gets:
$Pf(s,\x):=\int_s^Tdt \int_{\R^{nd}}P(s,t,\x,d\y)f(t,\y)=Kf(s,\x) $ if $f\in L^\rho([0,T]\times \R^{nd})\cap L^\infty([0,T]\times \R^{nd})$. Therefore, for $N=\lceil (n^2d+2)/2r\rceil,\ f\in C_0([0,T)\times \R^{nd}) $ (functions with compact support), 
\begin{equation}
\label{EQ_NP1}
|P^{N+1}f(s,\x)|\le C_{\ref{SOOTH_INCREASE_MAP}}\|f\|_{L^r([0,T]\times \R^{nd})}.
\end{equation}
This observation then yields the following result.
\begin{LEMME}
\label{REG_FIRST_LOC_COND}
If $a,\gF \in {\mathcal A}(r,T)$ denoting by $P $ the transition function associated with $(L_t)_{t\in [0,T]}$ then for $r\le \rho\le +\infty $,
$$\int_s^T dt(t-s)^N\int_{\R^{nd}}P(s,t,\x,d\y) f(t,\y)\le C_{\ref{REG_FIRST_LOC_COND}}^1 \|f\|_{L^\rho([0,T]\times \R^{nd})},$$
with $C_{\ref{REG_FIRST_LOC_COND}}^1:=C_{\ref{REG_FIRST_LOC_COND}}^1(T,r,\A{A}) $.
Also, for each $\delta>0$, $r<\rho\le \infty $,
\begin{eqnarray*}
\int_s^T dt \int_{\R^{nd}\backslash B^E(\btheta_{t,s}(\x),\delta)}P(s,t,\x,d\y) f(t,\y)\le C_{\ref{REG_FIRST_LOC_COND}}^2 \|f\|_{L^\rho([0,T]\times \R^{nd})},
\end{eqnarray*}
where $C_{\ref{REG_FIRST_LOC_COND}}^2:=C_{\ref{REG_FIRST_LOC_COND}}^2(T,r,\rho,\A{A},\delta) $.
\end{LEMME}
\begin{REM}
This is the first statement that differs from \cite{stro:vara:79}. Indeed, the unbounded transport contribution appears here for the first time.
To fully justify this aspect we give below the full proof of this result. 
\end{REM}
\begin{proof}
The first statement of the Lemma still follows from Lemma 9.1.3 in \cite{stro:vara:79} and \eqref{EQ_NP1} from an interpolation argument. For the second one, we can assume w.l.o.g. that $T\ge s+c,\ c>0 $. In that case:
\begin{eqnarray*}
|\int_s^T dt \int_{\R^{nd}\backslash B^E(\btheta_{t,s}(\x),\delta)}P(s,t,\x,d\y) f(t,\y)|\\
\le \sum_{k\ge 1} \int_{s+c/(k+1)}^{s+c/k} dt\int_{\R^{nd}\backslash B^E(\btheta_{t,s}(\x),\delta)}P(s,t,\x,d\y) |f(t,\y)|\\
+c^{-N}\int_{s+c}^T(t-s)^N\int_{\R^{nd}}P(s,t,\x,d\y) |f(t,\y)|.
\end{eqnarray*}
The last contribution can be bounded directly by the first statement of the Lemma. To control the sum, we see that introducing 
$$\Lambda_{s,\x}^k\varphi:=\int_{s+c/(k+1)}^{s+c/k}dt \int_{\R^{nd}\backslash B^E(\btheta_{t,s}(\x),\delta)}P(s,t,\x,d\y) \varphi(t,\y), $$
we indeed get, from Lemma 9.1.3 in \cite{stro:vara:79} and the first part of the lemma, that as a linear operator on $L^r([0,T]\times \R^{nd}),\ \Lambda_{s,\x}^k $ is bounded by $N!((k+1)/c)^N $. Now for $\varphi\in L^\infty([0,T]\times \R^d) $,
\begin{eqnarray*}
|\Lambda_{s,\x}^k \varphi|\le \frac ck |\varphi|_\infty \P_{s,\x}[\sup_{t\in [s,s+c/k]} |\X_t-\btheta_{t,s}(\x)|\ge \delta].
\end{eqnarray*}
Let us emphasize that it is precisely because we consider the deviations of the process from the deterministic differential system, that we can control the previous term with Bernstein like inequalities. Precisely, from Gronwall's lemma:
\begin{eqnarray*}
|\X_t-\btheta_{t,s}(\x)|\le \exp(C T)|\int_s^t \sigma(u,\X_u) dW_u|\le \exp(C)|\int_s^t \sigma(u,\X_u) dW_u|,
\end{eqnarray*}
with $C:=C(\A{A})$, so that, from Bernstein's inequality:
\begin{eqnarray*}
\P_{s,\x}[\sup_{t\in [s,s+c/k]} |\X_t-\btheta_{t,s}(\x)|\ge \delta]\le C\exp(-C^{-1}k\delta^2/c),\ C:=C(\A{A}),
\end{eqnarray*}
up to a modification of $C$.
The result then once again follows from standard interpolation.
\end{proof}

\subsubsection{Localization arguments}
Now we adapt more significantly the arguments in \cite{stro:vara:79} to get our results. The leading idea is the same as in the proof of Lemma \ref{REG_FIRST_LOC_COND}: to exploit the Bernstein-like deviations of the process from the deterministic system. We now want to localize carefully to get rid off the \textit{quasi-constant} diffusion coefficient of the previous section.
We have the following \textit{tubular} localization.
\begin{LEMME}[Tubular estimate]
\label{LEMME_TUB}
For $s_0\in [0,T),\ \x_1\in \R^{nd}$ let $\P_{s_0,\x_1} $ denote the solution to the martingale problem associated with $(L_{t})_{t\in [s_0,T]}$.
For $0<R_1<R_2 $, $\x_0\in \R^{nd} $ defining $\tau_{-1}=s_0$ and for all $k\in \N $,
\begin{eqnarray*}
\tau_{2k}:=\inf\{t\ge \tau_{2k-1}: |\X_t-\btheta_{t,s_0}(\x_0)|= R_2 \},\\
\tau_{2k+1}:=\inf\{t\ge \tau_{2k}: |\X_t-\btheta_{t,s_0}(\x_0)|= R_1 \},   
\end{eqnarray*}
then
$$\E^{\P_{s_0,\x_1}}[\sum_{k\ge 0}\I_{\tau_{2k}\in [0,T]}] \le C_{\ref{LEMME_TUB}}:=C_{\ref{LEMME_TUB}}(T,\A{A},R_2-R_1).$$
\end{LEMME}
The proof can be performed as in Lemma 9.1.6 in \cite{stro:vara:79}. The previous definitions of the stopping times allows to apply the required Bernstein like arguments similarly to the proof of Lemma \ref{REG_FIRST_LOC_COND}. 

The following result differs once again in the localization argument from Lemma 9.1.7 in \cite{stro:vara:79}, even though it can be proved rather similarly from Lemma \ref{LEMME_TUB}. We emphasize here that the localization has to be performed in time and space. Roughly speaking this is needed in order to \textit{partition} in time the characteristic tubes in subtubes for which the local condition \eqref{LOC_COND} is valid. This is the key of the proof.

\begin{LEMME}[First Localization Lemma]
\label{FIRST_LOC_LEMME}
Let $\P_{s_0,\x_1}$ solve the martingale problem for $(L_t)_{t\in [0,T]}$ starting from $(s_0,\x_1)\in [0,T]\times \R^{nd} $. Suppose now that the martingale problem associated with the operator 
$$\tilde L_t= \gF(t,\cdot)\cdot {\mathbf D}_\x +\frac12 \Tr(\tilde a(t,\cdot) D_{\x_1}^2) $$
is well posed and that $ \tilde a=a $ on 
\begin{equation}
\label{DEF_CROWN}
{\mathcal C}_{\underline t,\overline t ,R}(s_0,\x_0) :=\{(t,\y)\in [0,T]\times \R^{nd}: t\in [\underline t, \overline t], \btheta_{s_0,t}(\y) \in B^E(\x_0,R)  \},
\end{equation}
for some $s_0\le \underline t<\overline t\le T,\ R>0 $. 
Let us denote by $\tilde \P_{s_0,\x_1} $ the solution to the martingale problem for $(\tilde L_t)_{t\in [0,T]} $ starting from $(s_0,\x_1)\in [0,T]\times \R^{nd}  $.
Then for each $\delta \in (0,R) $:
\begin{eqnarray*}
|\E^{\P_{s_0,\x_1}}[\int_{s_0}^T f(t,\X_t)dt]|\le \E^{\tilde \P_{s_0,\x_1}}[\int_{s_0}^T |f(t,\X_t)|dt]\\
+C_{\ref{FIRST_LOC_LEMME}}\sup_{(s,\x)\in \partial {\mathcal C}_{\underline t,\overline t ,R-\delta}(s_0,\x_0)}\E^{\tilde \P_{s,\x}}[\int_{s}^T |f(t,\X_t)|dt],
\end{eqnarray*}
for all $f\in C_0({\mathcal C}_{\underline t,\overline t ,R-\delta}(s_0,\x_0)) $ and $C_{\ref{FIRST_LOC_LEMME}}:=C_{\ref{FIRST_LOC_LEMME}}(T,\A{A},R,\delta) $.
\end{LEMME}

We now specify how this Lemma needs to be used. As a direct corollary of Lemmas \ref{FIRST_LOC_LEMME} and \ref{REG_FIRST_LOC_COND} we derive:
\begin{LEMME}[Second Localization Lemma]
\label{LOC_LEM2}
Let $\tilde a(s,\x):=a(s,\x), 
$ in ${\mathcal C}_{\underline t,\overline t ,R}(s_0,\x_0) $ and $\tilde a(s,\x):=a(s_0,\x_0), 
$ elsewhere. Assume that the functions $\tilde a, \gF\in {\mathcal A}(r,T) $ for some $r\in (2,+\infty)$. Let $\P_{s_0,\x_1}$ solve the martingale problem for $L$ starting at $s_0,\x_1\in [0,T]\times \R^{nd}$. Then for each $0<\alpha<R $ and $r<\rho\le +\infty $. 
$$|\E^{\P_{s_0,\x_1}}[\int_{s_0}^T(t-s_0)^Nf(t,\X_t)dt]|\le C_{\ref{LOC_LEM2}}^1\|f\|_{L^\rho([s_0,T]\times \R^{nd})},  $$
for all $f\in C_0({\mathcal C}_{\underline t,\overline t ,\alpha}(s_0,\x_0)) $, where $N=\lceil (n^2 d+2)/2r\rceil $ and $C_{\ref{LOC_LEM2}}^1:=C_{\ref{LOC_LEM2}}^1(T,\A{A},r,\rho,R-\alpha)$. If additionally, $|\x_1-\x_0|>\alpha $, then 
$$|\E^{\P_{s_0,\x_1}}[\int_{s_0}^Tf(t,\X_t)dt]|\le C_{\ref{LOC_LEM2}}^2\|f\|_{L^\rho([s_0,T]\times \R^{nd})}, $$
where $C_{\ref{LOC_LEM2}}^2:=C_{\ref{LOC_LEM2}}^2(T,\A{A},r,\rho,R-\alpha,|\x_1-\x_0|-\alpha)$.

\end{LEMME}

\subsubsection{Proof of Theorem \ref{THM_LOC}}
From the previous localization Lemmas, the idea is now to specifically \textit{partition} the space in order to have crowns of the previous type,
${\mathcal C}_{\underline t,\overline t,R}(s_0,\x_0) $, introduced in \eqref{DEF_CROWN}, on which the local condition \eqref{LOC_COND} holds.
Let $q'$ denote the conjugate of $q\in (1,2)$ and choose $r\in (2,q')$  s.t. $N:=\lceil \frac{n^2 d+2}{2}\frac 1r \rceil=\lceil \frac{n^2 d+2}{2}\frac 1{q'} \rceil $.
Choose $T\le (4(\sup_{\rho\in [r,(\frac{n^2d+2}{2}\vee r)+1]}C(\rho,\A{A})))^{-2}$ in \eqref{CTR_R_OPINV} and set $\varepsilon^{-1}:= C_{\ref{EQ_EST_CZ}}(T,r,\A{A})\vee C_{\ref{EQ_EST_CZ}}(T,(\frac{n^2d+2}{2}\vee r)+1,\A{A})\vee \varepsilon_a^{-1}$ 
. 
Let us introduce for a fixed starting point $(s,\x) $ of the martingale problem, the spatial balls
$$Q_\k:=\{\x+\y: |\y_j-\k_j\gamma/(nd)^{1/2}|\le \gamma/(nd)^{1/2},\ j\in \leftB 1, nd\rightB  \},\ \k\in \Z^{nd},$$
where $\gamma:=\frac{\delta_T(\varepsilon)}{C_1}$, and recalling that $\delta_T $ stands for the modulus of continuity of $a$, the constant $C_1$ is then chosen large enough so that for all $\k\in \Z^{nd},\ \y_0,\y_1 \in Q_\k,\ t\in [s,T]$,
\begin{eqnarray} 
\label{loc_coeff}
|a(t,\btheta_{t,s}(\y_0))-a(t,\btheta_{t,s}(\y_1))| 
\le \varepsilon/2.
\end{eqnarray}
 This means that the local condition is satisfied on the time section of the transport of $Q_\k $ by the flow. In order to apply the previous results, we also need to handle the time contribution. Define now $h^\k:=\frac{T-s}{\lceil C_2(1+|\k\gamma/(nd)^{1/2}+\x|)\rceil} $ where the constant $C_2:=C_2(\varepsilon)$ is chosen large enough, so that
 for all $i\in \leftB 0, \lceil C_2(1+|\k\gamma/(nd)^{1/2}+\x|)\rceil\rightB$, setting $t_i^\k:=s+ih^\k $, the coefficient $a$ restricted to ${\mathcal C}_{t_i^\k,t_{i+1}^\k,\beta\gamma}(s,\x+\k \gamma/(nd)^{1/2}),\ \beta>1 $, coincides with some $\tilde a$ s.t. $\tilde a,\gF $ belong to the class ${\mathcal A}(r,T) $.

This choice simply means that the length of the time intervals for which we partition the set $\btheta(T,s,Q_\k):=\{(t,\z)\in [s,T] \times \R^{nd}: \btheta_{s,t}(\z)\in Q_\k
\} $ (image of $Q_\k$ by the flow between times $s$ and $T$) highly depends on the norm of the starting point. This is specifically due to the unbounded drift. Precisely we can write:
\begin{equation}
\label{CHOP_FLOW}
\btheta({T,s},Q_\k):=\bigcup_{i=0}^{\lceil C_2(1+|\k\gamma/(nd)^{1/2}+\x|)\rceil-1}{\mathcal C}_{t_i^\k,t_{i+1}^\k,\gamma}(s,\x+\k \gamma/(nd)^{1/2}).
\end{equation}
Now from Lemma \ref{LOC_LEM2} we get that for $\rho=(r+q')/2 $, for all $i\in \leftB 0,\lceil C_2(1+|\k\gamma/(nd)^{1/2}+\x|)\rceil-1\rightB$
\begin{equation}
\label{NEW_DEC_CUBE}
\begin{split}
&|\int_s^T dt (t-s)^N\int \I_{(t,\y)\in {\mathcal C}_{t_i^\k,t_{i+1}^\k,\gamma}(s,\x+\k \gamma/(nd)^{1/2})}P(s,t,\x,d\y)f(t,\y)|\\
&=|\int_{t_i^\k}^{t_{i+1}^\k} dt (t-s)^N\int_{\btheta_{t,s}(Q_\k)} P(s,t,\x,d\y)f(t,\y)|
\le C_{\ref{NEW_DEC_CUBE}}\|f\|_{L^\rho([0,T]\times \R^{nd})},
\end{split}
\end{equation}
where $C_{\ref{NEW_DEC_CUBE}}:=C_{\ref{NEW_DEC_CUBE}}(T,\A{A},r,\rho) $. On the other hand, comparing deviations along the characteristics allows once again to use Bernstein inequalities, similarly to the proof of Lemma \ref{REG_FIRST_LOC_COND}. Namely,
\begin{equation*}
\begin{split}
&P(s,t,\x,\btheta_{t,s}(Q_\k))\\
&\le \P_{s,\x}[\exists u\in[s,t], |\X_u-\btheta_{u,s}(\x)|>|\btheta_{u,s}(\x)-\btheta_{u,s}(\x+\k\gamma/(nd)^{1/2})|/2]\\
&\le \P_{s,\x}[\sup_{u\in [s,T]}|\int_s^u\sigma(v,\X_v)dW_v|\ge C|\k\gamma|]
\le 2d\exp(-\bar C^{-1}\frac{|\k|^2\gamma^2}{T{\mathbf d}^2}),\ \bar C:=\bar C(\A{A})\ge 1,
\end{split}
\end{equation*}
using Gronwall's Lemma (see proof of Lemma \ref{REG_FIRST_LOC_COND}) and the bi-Lipschitz property of the flow for the last but one inequality. 
We therefore obtain:
\begin{equation}
\label{CTR_INF_NEW_CUBE}
\begin{split}
|\int_s^T dt(t-s)^N\int \I_{(t,\y)\in {\mathcal C}_{t_i^\k,t_{i+1}^\k,\gamma}(s,\x+\k \gamma/(nd)^{1/2})}P(s,t,\x,d\y)f(t,\y)|\\
\le 
 C_{\ref{CTR_INF_NEW_CUBE}} \exp(-\bar C^{-1}\frac{|\k|^2\gamma^2}{T{\mathbf d}^2})|f|_\infty,
 \end{split}
\end{equation}
where $C_{\ref{CTR_INF_NEW_CUBE}}:=C_{\ref{CTR_INF_NEW_CUBE}}(T,n,d,N)$.
 We thus get by interpolation that for $\vartheta=1-\frac{\rho}{q'} \in (0,1)$:
\begin{equation*}
\begin{split}
|\int_s^T dt(t-s)^N\int \I_{(t,\y)\in {\mathcal C}_{t_i^\k,t_{i+1}^\k,\gamma}(s,\x+\k \gamma/(nd)^{1/2})}P(s,t,\x,d\y)f(t,\y)|\\
\le C_{\ref{NEW_DEC_CUBE}}^{1-\vartheta}C_{\ref{CTR_INF_NEW_CUBE}}^{\vartheta}\exp(-\bar C^{-1}\frac{\vartheta|\k|^2\gamma^2}{T{\mathbf d}^2})\|f\|_{L^{q'}([0,T]\times \R^{nd})}.
\end{split}
\end{equation*}
Summing for a given $\k\in \Z^d$ first over $i\in \leftB 0, \lceil C_2(1+|\k\gamma/(nd)^{1/2}+\x|)\rceil-1\rightB $ (that is according to \eqref{CHOP_FLOW} on $\btheta({T,s},Q_\k) $) we obtain 
\begin{equation*}
\begin{split}
&|\int_s^T dt (t-s)^N\int \I_{(t,\y)\in \btheta({T,s},Q_\k) }P(s,t,\x,d\y)f(t,\y)|\\
=&|\int_s^T dt (t-s)^N\int_{\btheta_{t,s}(Q_\k) }P(s,t,\x,d\y)f(t,\y)|\\
\le& \tilde C(T^{1/2}+|\x|)C_{\ref{NEW_DEC_CUBE}}^{1-\vartheta}C_{\ref{CTR_INF_NEW_CUBE}}^{\vartheta}\exp(-\frac{\bar C^{-1}}2\frac{\vartheta|\k|^2\gamma^2}{T{\mathbf d}^2})\|f\|_{L^{q'}([0,T]\times \R^{nd})}, \tilde C:=\tilde C(\A{A},\vartheta).
\end{split}
\end{equation*} 
Summing now over $\k\in \Z^{nd}$ yields:
\begin{equation}
\label{PRELIM_ESTIM_DES}
\begin{split}
|\int_s^T dt(t-s)^N\int_{\R^{nd}}P(s,t,\x,d\y)f(t,\y)|\\
\le C_{\ref{PRELIM_ESTIM_DES}}(1+|\x|)\|f\|_{L^{q'}([0,T]\times \R^{nd})}, C_{\ref{PRELIM_ESTIM_DES}}:=C_{\ref{PRELIM_ESTIM_DES}}(T,q,\A{A},\gamma). 
\end{split}
\end{equation} 
This contribution  already emphasizes the main difference w.r.t. the non degenerate case: the estimate depends on the initial point.
The proof of Theorem \ref{THM_LOC} can then be completed similarly to the one of Theorem 9.1.9 in \cite{stro:vara:79}. This achieves the proof of the existence of the density and the associated estimates when $T$ is small enough.

The existence of the density in Theorem \ref{THM_M} then follows from a chaining argument.

\subsubsection{Derivation of equation \eqref{krylov} in Theorem \ref{THM_M}}

To prove \eqref{krylov} we will inductively apply the results of the previous section along a time grid whose time-steps are lower than $T_0$ in Theorem \ref{THM_LOC}. We can assume w.l.o.g. that $T:=T_0 N,\ N\in \N  $. Setting now $t_i:=iT_0 ,\ i\in \leftB 0,N\rightB$, write from the strong Markov property:
\begin{equation}
\label{EQ_P_KRYL}
\begin{split}
&\E^{\P_{0,\x}}[\int_0^T f(t,\X_t)dt]=\sum_{i=0}^{N-1}\E^{\P_{0,\x}}[ \E^{\P_{t_i,\X_{t_i}}}[\int_{t_i}^{t_{i+1}}f(t,\X_t)dt]]\\
&=\Bigl\{\sum_{i=0}^{N-1}\E^{\P_{0,\x}}[ \E^{\P_{t_i,\X_{t_i}}}[\int_{t_i}^{t_{i+1}}f(t,\X_t)\I_{\X_t\in B^E(\btheta_{t,t_i}(\X_{t_i}),\delta)}dt]]\Bigr\}\\
&+
\Bigl\{\sum_{i=0}^{N-1}\E^{\P_{0,\x}}[ \E^{\P_{t_i,\X_{t_i}}}[\int_{t_i}^{t_{i+1}}f(t,\X_t)\I_{\X_t\not\in B^E(\btheta_{t,t_i}(\X_{t_i}),\delta)}dt]]\Bigr\}:=T_C+T_F,
\end{split}
\end{equation}
for some $\delta>0 $ to be specified later on.
From the second part of Theorem \ref{THM_LOC} we obtain that for $p>n^2d/2+1 $,
\begin{equation}
\label{CTR_F}
|T_F|\le C_{\ref{THM_LOC}}^2\sum_{i=0}^{N-1}\E^{\P_{0,\x}}[(1+|\X_{t_i}|)]\|f\|_{L^p([0,T]\times \R^{nd})}\le N C_{\ref{THM_LOC}}^2 C_F(1+|\x|)\|f\|_{L^p([0,T]\times \R^{nd})},
\end{equation}
where $C_F:=C_F(T_0,\A{A})$. On the other hand, we can follow the localization procedure of the previous proof (see equation \eqref{loc_coeff}), and find  $\delta>0$ s.t. setting $h_i:=\frac{t_{i+1}-t_i}{\lceil C_2(1+|\X_{t_i}|)\rceil} $ and $t_i^j:=t_i+h_ij,\ j\in \leftB 0,\lceil C_2(1+|\X_{t_i}|)\rceil$ then the coefficient $\tilde a$ is equal to $a$ on  ${\mathcal C}_{t_i^j,t_i^{j+1},2\delta}(t_i,\X_{t_i}) $ and to $a(t_i,\X_{t_i}))$ elsewhere, belongs to the class ${\mathcal A}(r,T) $ for some $r\in (2,n^2d/2+1)$. We then derive from Lemma \ref{FIRST_LOC_LEMME}:
\begin{equation*}
\begin{split}
|T_C|\le \sum_{i=0}^{N-1}\E^{\P_{0,\x}}\biggl[\sum_{j=0}^{\lceil C_2(1+|\X_{t_i}|)\rceil-1}\E^{\P_{t_i,\X_{t_i}}}\bigl[ \E^{\tilde \P_{t_i^j,\X_{t_i^j}}}[\int_{t_i^j}^{t_i^{j+1}}|f(t,\X_t)|\I_{\X_t\in B^E(\btheta_{t,t_i}(\X_{t_i}),\delta)}dt]\\
+\sup_{(s,\y)\in \partial {\mathcal C}_{t_i^j,t_i^{j+1},\delta}(t_i,\X_{t_i})} \E^{\tilde \P_{s,\y}}[\int_{s}^{t_i^{j+1}}|f(t,\X_t)|\I_{\X_t\in B^E(\btheta_{t,t_i}(\X_{t_i}),\delta)}dt]\bigr]\biggr].
\end{split}
\end{equation*}
Now, from equation \eqref{POINT_W_CTR}, we derive
\begin{equation}
\label{CTR_C}
|T_C|\le C\|f\|_{L^p([0,T]\times \R^{nd})}\sum_{i=0}^{N-1}\E^{\P_{0,\x}}[\lceil C_2(1+|\X_{t_i}|)\rceil]\le CN(1+|\x|)\|f\|_{L^p([0,T]\times \R^{nd})},
\end{equation}
up to a modification of $C$. The result follows from  \eqref{CTR_C}, \eqref{CTR_F}, \eqref{EQ_P_KRYL}.

\section{Proofs of the technical results}

\label{SEC_TEC}

\subsection{Proofs concerning the quasi-metric structure (Proposition \ref{PROP_QM})}
\label{APP_IS}
Let us first observe from the definition of the ${\mathbf d}$-balls, see equations \eqref{metric}, \eqref{distance}, that there exists $C_1:=C_1(\A{A})>0$ s.t. for all $\delta>0, (s,\x)\in S $,
$$|B((s,\x),\delta)|\le  C_1 \delta^{2+d\sum_{i=1}^n (2i-1)}=C_1 \delta^{2+n^2d}.$$
On the other hand, introducing 
\begin{eqnarray*}
\bar B((s,\x),\delta)&:=&\bigl\{(t,\y)\in S: |t-s|\le \frac{\delta^2}{4},\  \rho_{{\rm Sp}}(\btheta_{t,s}(\x)-\y)\le \frac\delta 2 \bigr\},\\
\forall \z\in \R^{nd},\ \rho_{\rm {Sp}}(\z)&:=&\sum_{i=1}^n|\z_i|^{1/(2i-1)},
\end{eqnarray*}
i.e. $\rho_{{\rm Sp}} $ corresponds to the spatial contribution in the metric \eqref{metric}, we have that $\bar B((s,\x),\delta)\subset B((s,\x),\delta) $. Indeed, for all $(t,\y)\in \bar B((s,\x),\delta),\ {\mathbf d}((s,\x),(t,\y)):=\rho(|t-s|,\btheta_{t,s}(\x)-\y)\le |t-s|^{1/2}+\rho_{{\rm Sp}}(\btheta_{t,s}(\x)-\y)\le \delta $.
Since we also have, up to a modification of $C_1$ that for all $(s,\x)\in S,\ \delta>0,\  |\bar B((s,\x),\delta)|\ge C_1^{-1}\delta^{2+n^2d} $ we therefore derive:
$$C_1^{-1}\delta^{2+n^2 d}\le |B((s,\x),\delta)|\le C_1 \delta^{2+n^2d},$$
which gives b) and c). 
To derive a), we need to exploit the specific structure of the dynamics. Let us first recall how to relate the forward and backward dynamics. Precisely, one has for all $v\in I(t,s):=([s,t]\I_{s<t})\cup ([t,s]\I_{s\ge t})$,
\begin{equation}
\label{COMP_FB}
\btheta_{v,s}(\x)-\btheta_{v,t}(\y)=\btheta_{t,s}(\x)-\y-\int_{v}^t (\gF(u,\btheta_{u,s}(\x))-\gF(u,\btheta_{u,t}(\y))) du,
\end{equation}
which for $v=s$ yields:
\begin{eqnarray*}
\x-\btheta_{s,t}(\y)=\btheta_{t,s}(\x)-\y-\int_{s}^t (\gF(u,\btheta_{u,s}(\x))-\gF(u,\btheta_{u,t}(\y))) du.
\end{eqnarray*}
Starting from the last components, and assuming w.l.o.g. that $ t>s$, we have:
\begin{eqnarray*}
|(\x-\btheta_{s,t}(\y))_n|\le |(\btheta_{t,s}(\x)-\y)_n|\\
+C_2 \int_{s}^t \left( |(\btheta_{v,s}(\x)-\btheta_{v,t}(\y))_{n-1}|+|(\btheta_{v,s}(\x)-\btheta_{v,t}(\y))_{n}| \right)dv \\
\le \exp(C_2(t-s))\left(|(\btheta_{t,s}(\x)-\y)_n|+C_2 \int_{s}^t|(\btheta_{v,s}(\x)-\btheta_{v,t}(\y))_{n-1}|dv \right),
\end{eqnarray*}
where $C_2:=C_2(\A{A})$ and using Gronwall's Lemma for the last inequality. Using iteratively \eqref{COMP_FB} and Gronwall's Lemma we derive that there exists $C_3:=C_3(T,\A{A})$ s.t.
\begin{equation}
\label{CTR_N}
|(\x-\btheta_{s,t}(\y))_n|\le C_3\sum_{j=1}^n |(\btheta_{t,s}(\x)-\y)_j||t-s|^{n-j}.
\end{equation}
Using Young's inequality with $p_j =\frac{2n-1}{2j-1},\ q_j=\frac{2n-1}{2(n-j)}$ in order to make the homogeneous exponent of coordinate $j\in \leftB 1,n-1\rightB$ appear, we get: 
\begin{eqnarray}
|(\x-\btheta_{s,t}(\y))_n|^{1/(2n-1)}\le C_3^{1/(2n-1)}\biggl[|(\btheta_{t,s}(\x)-\y)_n|^{1/(2n-1)}+\nonumber \\
\sum_{j=1}^{n-1} \left( \frac{|(\btheta_{t,s}(\x)-\y)_j)|^{1/(2j-1)}}{p_j}+\frac{|t-s|^{1/2}}{q_j} \right)\biggr]\le C_4{\mathbf d}((s,\x),(t,\y))\label{CTR_N_D},
\end{eqnarray}
for $C_4:=C_4(T,\A{A})$.

The above estimate does not exploit the fact that ${\mathbf d}((s,\x),(t,\y))\le \Lambda\le 1 $. This last assumption is actually needed for the components $i\in \leftB 1, n-1\rightB $ whose differential dynamics potentially involve coordinates $j>i$ with higher characteristic time-scales in small times but that are not negligible in the ``homogeneous" norm we consider. Namely, similarly to \eqref{CTR_N} we derive for all $i\in \leftB 1,n-1\rightB $ up to a modification of $C_3$:
\begin{equation}
\label{CTR_I}
|(\x-\btheta_{s,t}(\y))_i|\le C_3\left( \sum_{j=1}^i |(\btheta_{t,s}(\x)-\y)_j||t-s|^{i-j}+\sum_{j=i+1}^n |(\btheta_{t,s}(\x)-\y)_j| |t-s|\right).
\end{equation}
Thus, 
\begin{eqnarray*}
|(\x-\btheta_{s,t}(\y))_i|^{1/(2i-1)}\le C_3^{1/(2i-1)}\left( \sum_{j=1}^i \left(|(\btheta_{t,s}(\x)-\y)_j||t-s|^{i-j}\right)^{1/(2i-1)}\right.\\
\left.+\sum_{j=i+1}^n \left(|(\btheta_{t,s}(\x)-\y)_j| |t-s|\right)^{1/(2i-1)}\right).
\end{eqnarray*}
For the first contribution of the r.h.s. we can use again Young's inequality with $p_j=\frac{2i-1}{2j-1} ,\ q_j=\frac{2i-1}{2(i-j)} $. For the second contribution we exploit that since ${\mathbf d}((s,\x),(t,\y))\le \Lambda\le 1$ then for all $j\in \leftB 1,n\rightB$, $ |(\btheta_{t,s}(\x)-\y)_j|\le 1$ which for $j\ge i+1 $ yields 
$|(\btheta_{t,s}(\x)-\y)_j|^{1/(2i-1)}\le |(\btheta_{t,s}(\x)-\y)_j|^{1/(2j-1)} $. We therefore get up to a modification of $C_4$ that:
$$|(\x-\btheta_{s,t}(\y))_i|^{1/(2i-1)}\le C_4 {\mathbf d}((s,\x),(t,\y)),$$
which together with \eqref{CTR_N_D} indeed gives that there exists $C_{\ref{PROP_QM}}:=C_{\ref{PROP_QM}}(\A{A},T))$ s.t. ${\mathbf d}((t,\y),(s,\x)) \le C_{\ref{PROP_QM}}{\mathbf d}((s,\x),(t,\y))$ for $(s,\x),(t,\y)\in S,\  {\mathbf d}((s,\x),(t,\y))\le \Lambda\le 1$ which is the first part of a). It remains to prove the quasi-triangle inequality.
 Recalling that $\rho $ defined in \eqref{metric} satisfies the quasi-triangle inequality,  let us write:
\begin{eqnarray}
{\mathbf d}((s,\x),(t,\y))&=&\rho(t-\sigma+\sigma-s,\btheta_{t,s}(\x)- \btheta_{t,\sigma}(\bxi)+\btheta_{t,\sigma}(\bxi)   -\y)\nonumber\\
&\le & K (\rho(\sigma-s,\btheta_{t,s}(\x)- \btheta_{t,\sigma}(\bxi))+\rho(t-\sigma,\btheta_{t,\sigma}(\bxi)   -\y))\nonumber\\
&:=&K(\rho(\sigma-s,\btheta_{t,s}(\x)- \btheta_{t,\sigma}(\bxi))+{\mathbf d}((\sigma,\bxi),(t,\y))).\nonumber\\
\label{QUASI_TR}
\end{eqnarray}
On the other hand, using the specific form of $\gF$ in the dynamics of  $\btheta $, we can derive similarly to \eqref{CTR_N}, \eqref{CTR_I} using the direct forward dynamics that for all $i\in \leftB 1,n\rightB $:
\begin{eqnarray*}
|( \btheta_{t,\sigma}(\btheta_{\sigma,s}(\x)) -\btheta_{t,\sigma}(\bxi)  )_i|\\
\le C_3\left(\bsum{j=1}^{i}|(\btheta_{\sigma,s}(\x)-\bxi)_j||t-\sigma|^{i-j}+\bsum{j=i+1}^{n}|(\btheta_{\sigma,s}(\x)-\bxi)_j| |t-\sigma| \right).
\end{eqnarray*}
Thus, using as above Young inequalities and the fact that ${\mathbf d}((s,\x),(\sigma,\bxi))\le \Lambda\le 1 $ we get for all $ i\in \leftB 1,n\rightB$,
\begin{eqnarray*}
 |(\btheta_{t,s}(\x)-\btheta_{t,\sigma}(\bxi))_i|^{1/(2i-1)}\\
 \le C_3^{1/(2i-1)}\biggl( \bsum{j=1}^i \left[\frac{|(\btheta_{\sigma,s}(\x)-\bxi)_j|)^{1/(2j-1)} 
}{p_j}+ \frac{|t-\sigma|^{1/2}+|\sigma-s|^{1/2}}{q_j}\right]\\
+\sum_{j=i+1}^{n} |(\btheta_{\sigma,s}(\x)-\bxi)_j||^{1/(2j-1)}|t-\sigma|^{1/(2i-1)}\biggr)\le C_4({\mathbf d}((s,\x),(\sigma,\bxi))+{\mathbf d}((\sigma,\bxi),(t,\y)),
\end{eqnarray*}
with $p_j=\frac{2i-1}{2j-1} ,\ q_j=\frac{2i-1}{2(i-j)} $ in the last but one equality. Hence $\rho(\sigma-s,\btheta_{t,s}(\x)- \btheta_{t,\sigma}(\bxi))\le C_5 ({\mathbf d}((s,\x),(\sigma,\bxi))+{\mathbf d}((\sigma,\bxi),(t,\y)),\ C_5:=C_5(\A{A},T)$, 
which plugged into \eqref{QUASI_TR} concludes the proof up to a modification of $C_{\ref{PROP_QM}}$.\hfill$\square$
 \begin{REM}[Subdiagonal structure]
 \label{REM_SUBD}
Observe from the previous proof that when the function $\gF $ in the dynamics of $\btheta $ has the following structure, $\gF_1(t,\x)=\gF_1(t,\x_1)$, and for all $ i\in \leftB 2,n\rightB,\  \gF_i(t,\x^{i-1,n})=\gF_i(t,\x_{i-1},\x_i)$, then the 
terms in $j\in \leftB i+1,n\rightB $ do not appear in equation \eqref{CTR_I}. 
Hence, the distances respectively associated with  the forward and backward transport are actually equivalent. 
\end{REM}

\subsection{Controls on the flows and the frozen kernel}
\label{SEC_FLOW}
We first state a Lemma that gives some controls and equivalences for the scaled forward, backward and linearized flows, specifying the controls given in \eqref{EQUIV_FL}.
\begin{LEMME}[Controls and Equivalences of the scaled flows]
\label{BIG_LEMME_FLOW} 
There exists  a constant $C:=C(T,\A{A})$ s.t. for all 
$0\le s\le u<v\le t\le T $, $(\x,\y)\in (\R^{nd})^2 $, $w\ge v-u $,
\begin{eqnarray*}
C^{-1}|\T_{w}^{-1}(\x-\btheta_{u,v}(\y))|\le|\T_{w}^{-1}(\btheta_{v,u}(\x)-\y)|\le C|\T_{w}^{-1}(\x-\btheta_{u,v}(\y))|.
\end{eqnarray*}
We also have:
\begin{eqnarray*}
C^{-1} |\T_{t-s}^{-1}(\x-\btheta_{s,t}(\y))|\le |\T_{t-s}^{-1}(\tilde \btheta_{t,s}^{t,\y}(\x)-\y)| \le C |\T_{t-s}^{-1}(\x-\btheta_{s,t}(\y))|.\nonumber\\
\label{EQUIV_FL_BIS}
\end{eqnarray*}
\end{LEMME}

\textit{Proof.} The first control can be derived from the structure of the drift in \eqref{SYST} using Gronwall's Lemma. Indeed, similarly to \eqref{COMP_FB}:
\begin{eqnarray*}
|\T_{w}^{-1}(\btheta_{v,u}(\x)-\y)|=|\T_{w}^{-1}(\x-\btheta_{u,v}(\y))+\int_u^v \T_{w}^{-1}(\gF(r,\btheta_{r,u}(\x))-\gF(r,\btheta_{r,v}(\y)) )dr|\\
\le |\T_{w}^{-1}(\x-\btheta_{u,v}(\y))|+C w^{-1} \int_u^v|\T_{w}^{-1}(\btheta_{r,u}(\x)-\btheta_{r,v}(\y)) |dr \\
\le \exp(Cw^{-1}\int_u^v dr)|\T_{w}^{-1}(\x-\btheta_{u,v}(\y))|,
\end{eqnarray*}
where $C:=C(T,\A{A})$. Since $v-u\le w $, this gives the r.h.s. The l.h.s. is proved similarly.

To prove the second control, we need the following auxiliary yet important Scaling Lemma.
\begin{LEMME}[Scaling Lemma]
\label{scaling_lemma}
Let $0\le s<t\le T$ and $(s_0,\y_0)\in [0,T]\times \R^{nd} $ be given. The resolvent matrices in
\eqref{DYN_RES} can be written in the following way: for all $u \in [s,t] $,
\begin{eqnarray*}
 \tilde\gR^{s_0,\y_0}(u,s)&=&\T_{t-s} \bar \gR_{\frac{u-s}{t-s}}^{s,t,(s_0,\y_0)}\T_{t-s}^{-1},\\
\tilde\gR^{s_0,\y_0}(s,u)&=&\T_{t-s} \bar \gR_{\frac{s-u}{t-s}}^{s,t,(s_0,\y_0)}\T_{t-s}^{-1},
\end{eqnarray*}
where there exists $C_{\ref{scaling_lemma}}:=C_{\ref{scaling_lemma}}(T,\A{A})>1$ s.t. for all $(u,v)\in [s,t]^2,\ |\bar \gR_{\frac{u-v}{t-s}}^{s,t,(s_0,\y_0)}|\le C_{\ref{scaling_lemma}} $.
\end{LEMME}
\begin{proof}
Let us define for $w\in [0,1],\ \bar \gR_w^{s,t,(s_0,\y_0)}:=\T_{t-s}^{-1} \tilde \gR^{s_0,\y_0}(s+w(t-s),s)\T_{t-s}$. Observe from the differential dynamics in \eqref{DYN_RES} that:
\begin{eqnarray*}
\partial_w \bar \gR_w^{s,t,(s_0,\y_0)}&=&\T_{t-s}^{-1} (t-s) D\gF(s+w(t-s),\btheta_{s+w(t-s),s_0}(\y_0)) \T_{t-s}\\
&&\times \bigg[\T_{t-s}^{-1}\tilde \gR^{s_0,\y_0}(s+w(t-s),s)\T_{t-s} \bigg]\\
&=&\left\{\T_{t-s}^{-1} (t-s) D\gF(s+w(t-s),\btheta_{s+w(t-s),s_0}(\y_0)) \T_{t-s}\right\} \bar \gR_w^{s,t,(s_0,\y_0)}. 
\end{eqnarray*} 
Setting $\alpha_w^{s,t,(s_0,\y_0)}:=\left\{\T_{t-s}^{-1} (t-s) D\gF(s+w(t-s),\btheta_{s+w(t-s),s_0}(\y_0)) \T_{t-s}\right\} $, we derive from the subdiagonal structure of the partial gradient $D\gF$ that $ |\alpha_w^{s,t,(s_0,\y_0)}|\le C:=C(\A{A})$.
This gives the first statement taking $w=\frac{u-s}{t-s}$. The second follows by symmetry setting,
for $w\in [-1,0] $, $\bar \gR_w^{s,t,(s_0,\y_0)}:=\T_{t-s}^{-1} \tilde \gR^{s_0,\y_0}(s,s-w(t-s))\T_{t-s}$, differentiating in $w$ as above and taking eventually $w=-\frac{u-s}{t-s} $.
\end{proof}

Observe now from equation \eqref{AFFINE} that  
\begin{equation}
\label{PULL_BACK}
\tilde \gR^{t,\y}(s,t)(\y-\tilde \m^{t,\y}(s,t))=\btheta_{s,t}(\y),
\end{equation}
 i.e. we get the pull-back by the deterministic system of the final point $\y$ from $t$ to $s$. Hence:
\begin{eqnarray*}
|\T_{t-s}^{-1}(\tilde \btheta_{t,s}^{t,\y}(\x)-\y)|=|\T_{t-s}^{-1}\tilde \gR^{t,\y}(t,s)(\x-\btheta_{s,t}(\y))|=|\bar \gR_{1}^{s,t,(t,\y)}\T_{t-s}^{-1}(\x-\btheta_{s,t}(\y))|\\
\le C_{\ref{scaling_lemma}}|\T_{t-s}^{-1}(\x-\btheta_{s,t}(\y))|,
\end{eqnarray*}
giving the r.h.s. Once again, the  l.h.s. can be proved similarly. 
\hfill $\square $

As a consequence of Lemma \ref{scaling_lemma}, we derive the following controls for the derivatives of the \textit{frozen} density \eqref{DEF_KERN}  (see also the arguments in Section 5 of \cite{dela:meno:10}).
\begin{PROP}
\label{CTR_DENSITY}
There exist constants $C_{\ref{CTR_DENSITY}}:=C_{\ref{CTR_DENSITY}}(T,\A{A}),\ c_{\ref{CTR_DENSITY}}:=c_{\ref{CTR_DENSITY}}(T,\A{A})$ s.t. for all multi index $\alpha=(\alpha_1,\cdots, \alpha_n)\in \N^n, \ |\alpha|:=\sum_{i=1}^n \alpha_i\le 3 $ 
we have $\forall 0\le s<t\le T, \forall (\x,\y)\in (\R^{nd})^2$,
\begin{eqnarray*}
|
\partial_\x^\alpha \tilde  q(s,t,\x,\y)| \le \frac{C_{\ref{CTR_DENSITY}}}{(t-s)^{
\sum_{i=1}^n\frac{(2i-1)\alpha_i}2}}q_{c_{\ref{CTR_DENSITY}}}(s,t,\x,\y),\\
\forall c>0,\  q_{c}(s,t,\x,\y):=\frac{ c^{nd/2}}{(2\pi)^{nd/2}(t-s)^{n^2d/2}} \exp\left(-\frac c2 (t-s)|\T_{t-s}^{-1}(\tilde \btheta_{t,s}^{t,\y}(\x)-\y)|^2 \right).
\end{eqnarray*}
\end{PROP}

\subsection{Proof of  Proposition \ref{CZ_KER_SING} 
}
\label{SEC_GROS_CALCULS}
This section is devoted to the proof of Proposition
\ref{CZ_KER_SING} 
which provides the key estimates to derive Theorem \ref{EST_CZ} under assumption \A{A}.

\subsubsection{Some Preliminary Notations and Control of the Linearization Error}
Introduce first $\Sigma_1:=\{ (u,\z)\in \R \times \R^{nd}: \rho(u,\z)=1 \} $ with $\rho $ defined in \eqref{metric}, i.e. $\Sigma_1 $ is the level curve at $1$ of the parabolic metric. 
With this definition we can introduce, for given points $(s,\x),(t,\y),(\sigma,\bxi) \in [-T,T]\times \R^{nd} $, the mappings:
\begin{eqnarray}
\label{DEF_J1}
J_1: (t,\y)\in [-T,T]\times\R^{nd} \mapsto \left( s+\rho^2  \tilde s,\btheta_{t,s}(\x) +\rho^{-1}\T_{\rho^2 } \tilde \x \right),\ \rho:=\rho(t-s,\by-\btheta_{t,s}(\x))\nonumber\\
(\tilde s,\tilde \x):= \left( (t-s)\rho^{-2},\rho \T_{\rho^{-2}} \left( \y-\btheta_{t,s}(\x)  \right)  \right) \in \Sigma_1.
\end{eqnarray}
\begin{eqnarray}
\label{DEF_J2}
J_2: (\sigma,\bxi)\in [-T,T]\times\R^{nd} \mapsto \left( s+ \alpha^2  \bar s,\btheta_{\sigma,s}(\x) + \alpha^{-1}\T_{  \alpha^2 } \bar \x \right),\   \alpha:=\rho(\sigma-s,\bxi-\btheta_{\sigma,s}(\x))\nonumber\\
(\bar s,\bar \x):= \left( (\sigma-s) \alpha^{-2}, \alpha \T_{ \alpha^{-2}} \left(\bxi-\btheta_{\sigma,s}(\x)  \right)  \right) \in \Sigma_1.
\end{eqnarray}

Define now
\begin{equation}
\label{DEF_RESTES}
\cR_{t,s}^\rho(\x,\y):= \rho \T_{\rho^{-2}}(\btheta_{t,s}(\x)-\tilde \btheta_{t,s}^{t,\y}(\x)). 
\end{equation}

From a stability analysis similar to the one of equations (A-8), (A-10) in \cite{meno:10}, we derive the following lemma, which allows to control the linearization error. 
\begin{LEMME}
\label{LEMME_STAB}
Assume that \A{A} holds and that $0<\rho\le  \Lambda $ for some $\Lambda\le 1 $. Then, there exists $C_{\ref{LEMME_STAB}}:=C_{\ref{LEMME_STAB}}(\A{A})$ s.t. with the notation of \eqref{DEF_RESTES}:
\begin{eqnarray}
\label{LS1}
|\cR_{t,s}^\rho(\x,\y)|\le C_{\ref{LEMME_STAB}}(\rho^\eta 
+(t-s))|\tilde \x|. 
\end{eqnarray}
Also, if $\rho=\rho(t-s,\btheta_{t,s}(\x)-\y)\ge c_\infty \alpha=c_\infty \rho(\sigma-s,\bxi-\btheta_{\sigma,s}(\x)) $, then 
\begin{eqnarray}
\label{LS2}
|\cR_{t,\sigma}^\rho(\btheta_{\sigma,s}(\x),\y)|\le C_{\ref{LEMME_STAB}}( \rho^\eta 
+(t-\sigma))|\tilde \x|.
\end{eqnarray}
This implies that taking $j_0\in \leftB 1,n\rightB$ s.t. $|\tilde \x_{j_0}|:=\sup_{i\in \leftB 1,n \rightB}|\tilde \x_i|$ one has for all $j\in \leftB 1,n \rightB $:
\begin{eqnarray}
\label{CTR_R}
|\{\cR_{t,s}^\rho(\x,\y) \}_{j}|\le C_{\ref{LEMME_STAB}}n (
{\Lambda}^\eta+(t-s))|\tilde \x_{j_0}|,\nonumber\\
|\{\cR_{t,\sigma}^\rho(\btheta_{\sigma,s}(\x),\y)\}_{j}|\le C_{\ref{LEMME_STAB}}n (
{\Lambda}^\eta+(t-\sigma))|\tilde \x_{j_0}|.
\end{eqnarray}
Eventually, if $|\sigma-s|\le K|t-\sigma| $, we also have
\begin{equation}
\label{LE_CTR_QUI_SAUVE}
(t-\sigma)^{1/2}|\T_{t-\sigma}^{-1}(\tilde \btheta_{\sigma,s}^{t,\y}(\x) -\btheta_{\sigma,s}(\x))|\le C\{\rho^\eta+|\sigma-s| \}(t-\sigma)^{1/2}|\T_{t-\sigma}^{-1}(\btheta_{t,s}(\x)-\y)|.
\end{equation} 
\end{LEMME}
\textit{Proof.} Let us prove \eqref{LS1}.
Recalling equations \eqref{DET_SYST}, \eqref{eq:F:240409:3} and \eqref{DEF_N}, \eqref{DEF_R}, we write: 
\begin{eqnarray} 
\cR_{t,s}^\rho(\x,\y):=\rho \T_{\rho^{-2}}\left\{\btheta_{t,s}(\x)-\tilde \btheta_{t,s}^{t,\y}(\x) \right\}=\nonumber\\ 
\rho \T_{\rho^{-2}}\left\{ \int_{s}^{t} du \biggl[\biggl( \gF(u,\btheta_{u,s}(\x))-\gF^{t,\y}(u,\btheta_{u,s}(\x)) \biggr)\right.\nonumber\\ 
\left.+\biggl(D\gF(u,\btheta_{u,t}(\y))(\btheta_{u,s}(\x)-\tilde \btheta_{u,s}^{t,\y}(\x) )\biggr) \right. \nonumber \\ 
\left. + \biggl(\bint{0}^{1} d\delta \left(  D\gF^{t,\y}(u,\btheta_{u,t}(\y)+\delta (\btheta_{u,s}(\x)-\btheta_{u,t}(\y) ) ) \right. \right.\nonumber\\ 
-\left. \left. D\gF^{t,\y}(u,\btheta_{u,t}(\y)) \right) (\btheta_{u,s}(\x)-\btheta_{u,t}(\y) )\biggr) \biggr] \right\}\nonumber \\ 
:=(\cR_{t,s}^{\rho,1}+\cR_{t,s}^{\rho,2}+\cR_{t,s}^{\rho,3})(\x,\y),
\label{D_EPS} 
\end{eqnarray} 
where, accordingly with the notations of \eqref{DEF_N}, for $(u,\z)\in [s,t]\times \R^{nd}, D\gF^{t,\y}(u,\z) $ is the $(nd)\times (nd)$ matrix with only non zero $d\times d $ matrix entries $(D\gF^{t,\y}(u,\z))_{j,j-1}:=D_{\x_{j-1}}\gF_j(u,\z_{j-1},\btheta_{u,t}(\y)^{j,n}) ,\  j\in \leftB 2,n\rightB $, so that in particular $D\gF^{t,\y}(u,\btheta_{u,t}(\y))=D\gF(u,\btheta_{u,t}(\y)) $. 
 
The structure of the ``partial gradient" $D\gF^{t,\y} $ and its H\"older property yield that there exists $C_3:=C_3(T,\A{A})$ s.t. for all $j\in \leftB 2,d \rightB $: 
\begin{eqnarray} 
|(\cR_{t,s}^{\rho,3}(\x,\y))_j| &\le & C_3 \rho^{1-2j}\int_{s}^{t} du|(\btheta_{u,s}(\x)- \btheta_{u,t}(\y))_{j-1}||\btheta_{u,s}(\x)-\btheta_{u,t}(\y)|^\eta
\nonumber \\ 
&\le & C_3|\btheta_{t,s}(\x)-\y|^\eta  \rho^{-2}\int_s^{t}du (\sum_{k=2}^n  \rho^{1-2(k-1)}|(\btheta_{u,s}(\x)- \btheta_{u,t}(\y))_{k-1}|)\nonumber. 
\end{eqnarray}
Since $\rho=\rho(t-s,\btheta_{t,s}(\x)-\y)=d\big((s,\x),(t,\y) \big)  \le \Lambda \le 1 $ we derive:
$$|\btheta_{t,s}(\x)-\y|\le C \sum_{i=1}^n |(\btheta_{t,s}(\x)-\y)_i|^{1/(2i-1)}\le C\rho,\ C:=C(n).$$
Hence, up to modifications of $C_3$,
\begin{eqnarray}
|(\cR_{t,s}^{\rho,3}(\x,\y))_j|&\le & C_3 \rho^\eta \times  \rho^{-2}\int_{s}^{t} du (\rho|\T_{\rho^{-2}}(\btheta_{u,s}(\x)- \btheta_{u,t} (\y))|)\nonumber \\ 
&\le & C_3  \rho^{-2+\eta}\int_{s}^{t} du (\rho|\T_{\rho^{-2}}(\btheta_{t,s}(\x)-\y)|)\nonumber \\ 
&\le & C_3 \rho^{1+\eta}|\T_{\rho^{-2}}(\btheta_{t,s}(\x)-\y)|, \label{CTR_D3} 
\end{eqnarray} 
using Lemma \ref{BIG_LEMME_FLOW} 
for the last but one inequality 
and recalling from \eqref{DEF_J1} that $|t-s|/\rho^2\le 1 $ for the last one.

On the other hand, the term $\cR_{t,s}^{\rho,1}(\x,\y)$ can be seen as a remainder w.r.t. the characteristic time scales. Precisely, there exists $C_1:=C_1(T,\A{A})$ (possibly changing from line to line) s.t. for all $j\in \leftB 1,n \rightB $: 
\begin{eqnarray} 
|(\cR_{t,s}^{\rho,1}(\x,\y))_j|&\le&C_1\rho^{1-2j}\int_s^{t}du \sum_{k=j}^n|(\btheta_{u,s}(\x)-\btheta_{u,t}(\y))_k|\nonumber \\ 
                                 &\le & C_1\int_{s}^{t} du \rho|\T_{\rho^{-2}}(\btheta_{u,s}(\x)-\btheta_{u,t}(\y))| \nonumber \\ 
                                 &\le & C_1 (t-s) \rho|\T_{\rho^{-2}}(\btheta_{t,s}(\x)-\y)|\label{CTR_D1} 
\end{eqnarray} 
using once again Lemma \ref{BIG_LEMME_FLOW} 
for the last inequality. 

Recall now that $\cR_{t,s}^{\rho,2}(\x,\y)$ is the linear part of equation \eqref{D_EPS}. Setting 
$$\forall u\in [s,t],\  \alpha_{t,u}^\rho(\y):=\left\{\rho\T_{\rho^{-2}} D\gF(u,\btheta_{u,t}(\y)) \rho^{-1}\T_{\rho^2}\right\},$$
  it can be rewritten  
\begin{eqnarray*} 
\cR_{t,s}^{\rho,2}(\x,\y)&=&\int_s^{t} du \left\{\rho\T_{\rho^{-2}} D\gF(u,\btheta_{u,t}(\y)) \rho^{-1}\T_{\rho^2}\right\}
\left(\rho\T_{\rho^{-2}}(\btheta_{u,s}(\x)-\tilde \btheta_{u,s}^{t,\y}(\x)) \right)\\  
&=&\int_s^{t} du \alpha_{t,u}^\rho(\y)\left(\rho\T_{\rho^{-2}} (\btheta_{u,s}(\x)-\tilde \btheta_{u,s}^{t,\y}(\x))\right) 
=\int_s^{t} du \alpha_{t,u}^\rho(\y)\cR_{u,s}^{\rho}(\x,\y),
\end{eqnarray*} 
where there exists a constant $ C_2:=C_2(T,\A{A})$ s.t. $\int_{s}^{t} du|\alpha_{t,u}^\rho(\y)|\le C_2$. 
From \eqref{CTR_D1}, \eqref{CTR_D3}, \eqref{D_EPS} and Gronwall's Lemma  we derive  
\begin{eqnarray*}
\exists C_4:=C_4(T,\A{A}),\ |\cR_{t,s}^{\rho}(\y)|&\le &C_4 (\rho^\eta+(t-s))\rho|\T_{\rho^{-2}}(\btheta_{t,s}(\x)-\y)|\\
&\le & C_4(\rho^\eta+(t-s)) |\tilde \x|
\end{eqnarray*} 
recalling \eqref{DEF_J1} for the last inequality. This gives equation \eqref{LS1} of the Lemma.
For \eqref{LS2}, the previous proof can be adapted with obvious modifications using thoroughly that $\btheta_{u,\sigma}(\btheta_{\sigma,s}(\x)):=\btheta_{u,s}(\x) $ and Lemma \ref{BIG_LEMME_FLOW}. The main differences are that the time integrals are taken between $\sigma $ and $t$. Following the computations leading to \eqref{CTR_D3}, the contribution $\cR_{t,\sigma}^{\rho,3}(\btheta_{\sigma,s}(\x),\y) $ would be bounded by 
$C_3\rho^{-2+\eta}|t-\sigma| \{\rho|\T_{\rho^{-2}}(\btheta_{t,s}(\x)-\y)|\}\le C_3\rho^{-2+\eta}(|t-s|+|s-\sigma|) \{\rho|\T_{\rho^{-2}}(\btheta_{t,s}(\x)-\y)|\}  $. Recalling also that $|t-s|+|\sigma-s|\le \rho^2+\alpha^2\le (1+c_\infty)\rho^2 $ on the considered set, we get that
\eqref{CTR_D3} still holds in that 
case. We would similarly have $|\cR_{t,\sigma}^{\rho,1}(\btheta_{\sigma,s}(\x),\y)|\le C_1 (t-\sigma)\rho|\T_{\rho^{-2}}(\btheta_{t,s}(\x)-\y)| $ giving \eqref{CTR_D1} in that case. Eventually, the same previous triangle inequality would give that, on the considered set $ \int_\sigma^t du|\alpha_{t,y}^\rho(\y)|\le C$, so that $\cR_{t,\sigma}^{\rho,2}(\btheta_{\sigma,s}(\x),\y) $ can still be viewed as the well controlled linear part of the inequality. The  proof then again follows from Gronwall's lemma. 

Eventually, \eqref{LE_CTR_QUI_SAUVE} is established similarly exploiting again Lemma \ref{BIG_LEMME_FLOW} and the condition $|\sigma-s|\le K|t-\sigma| $.
\qed
\subsubsection{Main Proof of Proposition
\ref{CZ_KER_SING}.}
Point \textit{i)} can be derived for both kernels $k_{ij}^{d},\ k_{ij}^{d,*} $ recalling from  \eqref{CTR_SING}  and Lemma \ref{BIG_LEMME_FLOW} that there exists $C:=C(\A{A},T)\ge 1$ s.t. $\forall ((s,\x),(t,\y))\in S^2 $, \label{PREUVE_PI}
\begin{eqnarray}
|k_{i,j}^{d}(s,t,\x,\y)|+|k_{i,j}^{d,*}(s,t,\x,\y)|\nonumber \\
\le \frac{C}{(t-s)^{1+n^2d/2}}\exp\left(-C^{-1}(t-s)|\T_{t-s}^{-1}(\btheta_{t,s}(\x)-\y)^2| \right).\label{kerns-eq-fl}
\end{eqnarray}
Now for a given $c_1>2 $, if $ c_1|t-s|^{1/2}> {\mathbf d}((s,\x),(t,\y))$ then the r.h.s of \eqref{kerns-eq-fl} can directly be upper bounded by  $Cc_1^{2+n^2d}/{\mathbf d}((s,\x),(t,\y))^{2+n^2d} $. On the other hand, if $c_1|t-s|^{1/2}\le {\mathbf d}((s,\x),(t,\y)) $ then, by definition of  $d $ in \eqref{distance} we derive that $\exists i\in \leftB 1,n\rightB $ s.t.
$$|(\btheta_{t,s}(\x)-\y)_i|^{1/(2i-1)}\ge \frac{1}{n}(1-\frac1{c_1}){\mathbf d}((s,\x),(t,\y)).$$
This property yields:
\begin{eqnarray*}
|k_{i,j}^{d}(s,t,\x,\y)|+|k_{i,j}^{d,*}(s,t,\x,\y)|
\le \frac{C}{|(\btheta_{t,s}(\x)-\y)_i|^{\frac{n^2d+2}{2i-1}}}\left( \frac{|(\btheta_{t,s}(\x)-\y)_i|}{|t-s|^{1/2(2i-1)}}\right)^{\frac{n^2 d+2}{2i-1}}\\
\times \exp\left(-C^{-1}(t-s)|\T_{t-s}^{-1}(\btheta_{t,s}(\x)-\y)^2| \right)\\
\le  \frac{\tilde C}{{\mathbf d}((s,\x),(t,\y))^{n^2 d+2}}\exp\left(-\bar C^{-1}(t-s)|\T_{t-s}^{-1}(\btheta_{t,s}(\x)-\y)^2| \right), 
\end{eqnarray*}
where $ \tilde C:=\tilde C(\A{A},T,c_1),\ \bar C:=\bar C(\A{A},T)$. This gives the first claim.\\

Let us now establish point  \textit{ii)} for the kernel $k_{i,j}^{d} $, recalling that  below 
\begin{equation}
\label{COND_CZ_LOIN}
c_\infty {\mathbf d}((s,\x),(\sigma,\bxi)) \le {\mathbf d}((s,\x),(t,\y)) \le \Lambda\le 1,
\end{equation}
for $ c_\infty:= c_\infty(\A{A})$ large enough and $\Lambda$ small enough to be specified later on. We can w.l.o.g. assume that $|\sigma-s|\le K|t-\sigma|$, for some $K:=K(\A{A},T) >1$ and write:
\begin{eqnarray} 
\label{DECOUP_INTEGRAND} 
|k_{ij}(s,t, \x,\y)-k_{ij}(\sigma,t, \bxi,\y)|\le |k_{ij}(s,t, \x,\y)-k_{ij}(\sigma,t, \tilde \btheta_{\sigma,s}^{t,\y}(\x),\y)|\nonumber \\
+|k_{ij}(\sigma,t, \tilde \btheta_{\sigma,s}^{t,\y}(\x),\y)-k_{ij}(\sigma,t, \bxi,\y)|\nonumber\\
 :=I_1(s,\sigma,t,\x,\y)+I_2(\sigma,t,\x,\bxi,\y).
\end{eqnarray}
\begin{REM}
\label{DECOUP_ALTERNATIVE}
The previous splitting of $|k_{ij}(s,t, \x,\y)-k_{ij}(\sigma,t, \bxi,\y)| $ has been done to separate the time and space sensitivities.  In $I_1 $ the space variable is frozen and from \eqref{reg_kernel} and the flow property of $\tilde \btheta^{t,\y} $ its value is equal to $\tilde \btheta_{t,s}^{t,\y}(\x)-\y $. In $I_2 $ the time variables are equal to $t-\sigma $. Also, the intermediate spatial point $\tilde \btheta_{\sigma,s}^{t,\y}(\x)$ yields from $I_2$ a difference of the form $|\tilde \btheta_{\sigma,s}^{t,\y}(\x)-\bxi| $ which up to a linearization error has the same order as $|\btheta_{\sigma,s}(\x)-\bxi| $, norm of the spatial point appearing in \eqref{COND_CZ_LOIN}. The condition $|\sigma-s|\le K|t-\sigma| $ is here needed to use properties on the rescaled flows for the spatial sensitivity (see Lemma \ref{BIG_LEMME_FLOW} and equation \eqref{CTR_SPATIAL_FLOW}) which allow to control the linearization error.
We emphasize that if $|\sigma-s|\ge K|t-\sigma| $ (and therefore $|t-s|\ge (1-1/K) |\sigma-s|$) then the integrand has to be split differently, writing
\begin{eqnarray*}
|k_{ij}(s,t, \x,\y)-k_{ij}(\sigma,t, \bxi,\y)|\le |k_{ij}(s,t, \x,\y)-k_{ij}(s,t, \tilde \btheta_{s,\sigma}^{t,\y}(\bxi),\y)|\nonumber \\
+|k_{ij}(s,t, \tilde \btheta_{s,\sigma}^{t,\y}(\bxi),\y)-k_{ij}(\sigma,t, \bxi,\y)|.
\end{eqnarray*}
The above terms could be analyzed similarly to those appearing in \eqref{DECOUP_INTEGRAND} following the procedure below. 
\end{REM}


 Setting for all $-T\le s<t\le T,\ (\z,\y)\in (\R^{nd})^2 $:
\begin{eqnarray*}
\bar k_{ij}(s,t,\z,\y)&:=& \I_{-T\le s<t\le T} \left(- [\tilde \gR^{t,\y}(t,s)^*\tilde \K^\y(s,t)^{-1}\tilde \gR^{t,\y}(t,s)]_{11}\right. \nonumber\\
&&\left.+[\tilde \gR^{t,\y}(t,s)^*\tilde \K^\y(s,t)^{-1}(\z-\y)  ]_1^{\otimes 2}\right) \times \biggl(\frac{1}{(2\pi)^{nd}\det(\tilde \K^\y(s,t))^{1/2}}\\
&&\times \exp(-\frac12 \langle \tilde \K^\y(s,t)^{-1}(\z-\y),\z-\y\rangle) \biggr), 
\end{eqnarray*}
we can rewrite $I_1(s,\sigma,t,\x,\y)=|\bar k_{ij}(s,t,\tilde \btheta_{t,s}^{t,\y}(\x),\y)-\bar k_{ij}(\sigma,t,\tilde \btheta_{t,s}^{t,\y}(\x),\y)| $. Thus,
from \eqref{reg_kernel}, we derive similarly to  \eqref{CTR_SING} (see also the proof of Proposition 3.7 in \cite{dela:meno:10} for a thorough discussion on the time sensitivities of the covariance matrix) that $\exists (c,C):=(c,C)(T,\A{A})>0 $ s.t.
\begin{eqnarray}
\label{SING_I1}
I_1(s,\sigma,t,\x,\y) &\le& |s-\sigma|\sup_{\tau\in [s\wedge \sigma,(s\vee \sigma)\wedge t[}| \partial_\tau \bar k_{ij}(\tau,t, \z,\y)|_{\z=\tilde \btheta_{t,s}^{t,\y}(\x)}\nonumber\\
&\le &C |s-\sigma |\times \sup_{\tau\in [s\wedge  \sigma ,(s\vee \sigma) \wedge t[}\frac{\bar q_c(\tau,t,\z,\y)}{(t-\tau)^2}|_{\z=\tilde \btheta_{t,s}^{t,\y}(\x)},
\end{eqnarray}
where $\bar q_c(\tau,t,\z,\y):=\frac{c^{nd/2}}{(2\pi)^{nd/2}(t-\tau)^{n^2d/2}}\exp(-\frac c2 (t-\tau)|\T_{t-\tau}^{-1}(\z-\y)|^2) $.

We have to consider the terms $I_1,I_2$ under the condition \eqref{COND_CZ_LOIN} that rewrites $\{(t,\y)\in [0,T]\times \R^{nd}: 
\rho=\rho(t-s,\by- \btheta_{t,s}(\x)) \ge \my_c \alpha=\my_c\rho(\sigma-s,\bxi-\btheta_{\sigma,s}(\x)) \} $. 

From 
\eqref{DEF_J1} we get that for all $\tau\in [s\wedge  \sigma ,(s\vee \sigma) \wedge t[$:
\begin{eqnarray}
\label{THE_SING_TIME}
\frac{\bar q_c(\tau,t, \tilde \btheta_{t,s}^{t,\y}(\x),\y)}{(t-\tau)^2}&\le& \frac{C}{\rho^{4+n^2d}}\frac{1}{\left(\tilde s-\frac{\tau-s}{\rho^2}\right)^{2+n^2d/2}}\nonumber\\
&&\times\exp\left(-c(t-\tau)|\T_{t-\tau}^{-1}(\rho^{-1}\T_{\rho^2}\{\tilde \x+\cR_{t,s}^\rho(\x,\y) \})|^2 \right), 
\end{eqnarray}
using the notation introduced in \eqref{DEF_RESTES},
i.e. the term $ \cR_{t,s}^\rho(\x,\y)$ measures the difference associated with the approximation of the non-linear flow by the linear one.

Now, if $|\tilde s-\frac{\tau-s}{\rho^2}|^{1/2}\ge \frac{1}{\my_c}=:\tilde c$, for $\my_c\ge 1 $ to be specified later on, we have from 
\eqref{THE_SING_TIME} and \eqref{SING_I1} that $I_1(s,\sigma,t,\x,\y)\le \frac{C|s-\sigma|}{\rho^{4+n^2d}}\le \frac{C\alpha^2}{\rho^{4+n^2d}}$ using \eqref{DEF_J2} for the last inequality. 
On the other hand,  from  \eqref{COND_CZ_LOIN} $\rho\ge \my_c\alpha $.  Hence, $|\tau-s|\le |\sigma-s|\le \alpha^2\le \frac{\rho^2}{\my_c^2} $. Since $(\tilde s,\tilde \x)\in \Sigma_1 $,  we thus 
 derive:
\begin{eqnarray*}
 \sum_{j=1}^{n}|\tilde \x_j|^{1/(2j-1)}+\left|\tilde s-\frac{\tau-s}{\rho^2}\right|^{1/2}=  1- \left| \frac{(\tau-s)}{\rho^2} \right|^{1/2}
 \ge 1-\frac{1}{\my_c}=1-\tilde c.
 \end{eqnarray*}
Hence, for $|\tilde s-\frac{\tau-s}{\rho^2}|^{1/2}\le  \tilde c $, we obtain 
\begin{eqnarray}
\label{CTR_GRD}
 \sum_{j=1}^{n}|\tilde \x_j|^{1/(2j-1)}\ge 1-2\tilde c\  {\rm and  }\  \exists j_0\in \leftB 1,n \rightB,\ |\tilde \x_{j_0}|^{1/(2j_0-1)}\ge \frac{1-2\tilde c}{n}>0,\ {\rm for \ } \my_c> 2 .\nonumber\\ 
 \end{eqnarray}
Write now:
\begin{eqnarray*}
(t-\tau)|\T_{t-\tau}^{-1}(\rho^{-1}\T_{\rho^2}\{ \tilde \x+ \cR_{t,s}^\rho(\x,\y)\})|^2 
=\bsum{j=1}^{n} \left(\tilde s-\frac{\tau-s}{\rho^2} \right)^{-(2j-1)}|\{\tilde\x+ \cR_{t,s}^\rho(\x,\y) \}_j|^2. 
\end{eqnarray*}

 Thus, we get from Lemma \ref{LEMME_STAB}, equation \eqref{CTR_R}, that for $T$ and $\Lambda $ s.t. $ C_{\ref{LEMME_STAB}}n ({\Lambda}^\eta+(t-s)) \le 1/2$:
\begin{eqnarray*}
&&(t-\tau)|\T_{t-\tau}^{-1}(\rho^{-1}\T_{\rho^2}\{ \tilde \x+ \cR_{t,s}^\rho(\x,\y)\})|^2\\
&&\ge   \left(\tilde s-\frac{\tau-s}{\rho^2} \right)^{-(2j_0-1)}|\tilde \x_{j_0}|^2\left(\frac 12-(C_{\ref{LEMME_STAB}} n(\Lambda^\eta+(t-s)))^2   \right)\\
&&\ge \frac14 \left(\tilde s-\frac{\tau-s}{\rho^2} \right)^{-(2j_0-1)}\left(\frac{1-2\tilde c}{n} \right)^{2(2j_0-1)},
\end{eqnarray*}
using \eqref{CTR_GRD} for the last inequality.
Plugging the above control into \eqref{THE_SING_TIME} yields:
\begin{eqnarray}
\frac{\bar q_c(\tau,t, \tilde \btheta_{t,s}^{t,\y}(\x),\y)}{(t-\tau)^2}\le \frac{C}{\rho^{4+n^2d}}\frac{1}{\left(\tilde s-\frac{\tau-s}{\rho^2}\right)^{2+n^2d/2}}\nonumber\\
\times \exp\left(-\frac{c}4\left(\tilde s-\frac{\tau-s}{\rho^2} \right)^{-(2j_0-1)}\left(\frac{1-2\tilde c}{n} \right)^{2(2j_0-1)}\right)\le \frac{C}{\rho^{4+n^2d}}.\label{CTR_PART_I1}
\end{eqnarray}
From \eqref{CTR_PART_I1} and \eqref{SING_I1} we finally get the global bound:
\begin{equation}
\label{CTR_I1}
\exists C_1:=C_1(T,\A{A})>0, \ I_1(s,\sigma,t,\x,\y) \le \frac{C_1\alpha^2}{\rho^{4+n^2d}}.
\end{equation}

\label{P25}
Let us now turn to $I_2(\sigma,t,\x,\bxi,\y)
$. 
From Proposition \ref{CTR_DENSITY} and  \eqref{reg_kernel}, we get similarly to \eqref{SING_I1} that $\exists (c,C):=(c,C)(T,\A{A})>0 $ s.t.: 
\begin{eqnarray}
I_2(\sigma,t,\x,\bxi,\y) &\le &C (t-\sigma)^{1/2}|\T_{t-\sigma}^{-1}(\tilde \btheta_{t,s}^{t,\y}(\x)-\tilde \btheta_{t,\sigma}^{t,\y}(\bxi))|\nonumber\\
&& \times\frac{1}{(t-\sigma)}\sup_{\gamma\in [0,1] }\bar q_c(\sigma,t,\gamma\tilde \btheta_{t,s}^{t,\y}(\x)+(1-\gamma)\tilde \btheta_{t,\sigma}^{t,\y}(\bxi),\y)\nonumber\\
&\le&C (t-\sigma)^{1/2}|\T_{t-\sigma}^{-1}( \tilde \btheta_{\sigma,s}^{t,\y}(\x)-\bxi)|\nonumber\\
&&\times \frac{1}{(t-\sigma)}\sup_{\gamma\in [0,1] }\bar q_c(\sigma,t,\gamma\tilde \btheta_{t,s}^{t,\y}(\x)+(1-\gamma)\tilde \btheta_{t,\sigma}^{t,\y}(\bxi),\y)\nonumber\\
&\le&C (t-\sigma)^{1/2}\{|\T_{t-\sigma}^{-1}( \btheta_{\sigma,s}(\x)-\bxi)|+|\T_{t-\sigma}^{-1}( \tilde \btheta_{\sigma,s}^{t,\y}(\x) -\btheta_{\sigma,s}(\x))|\}\nonumber\\
&&\times \frac{1}{(t-\sigma)}\sup_{\gamma\in [0,1] }\bar q_c(\sigma,t,\gamma\tilde \btheta_{t,s}^{t,\y}(\x)+(1-\gamma)\tilde \btheta_{t,\sigma}^{t,\y}(\bxi),\y),
\label{CTR_SPATIAL_FLOW}
\end{eqnarray}
where the last but one inequality is derived similarly to the first statement in Lemma \ref{BIG_LEMME_FLOW} using the dynamics \eqref{eq:F:240409:3} associated with \eqref{AFFINE}. 
Plugging  equation \eqref{LE_CTR_QUI_SAUVE} from Lemma \ref{LEMME_STAB} into \eqref{CTR_SPATIAL_FLOW} now yields:
\begin{eqnarray}
I_2(\sigma,t,\x,\bxi,\y)\nonumber\\
\le  C\biggl\{ \bsum{k=1}^{n}\frac{|(\btheta_{\sigma,s}(\x)-\bxi)_k|+(\rho^\eta+|s-\sigma|)|(\btheta_{t,s}(\x)-\y)_k|}{\rho^{2+(2k-1)+n^2d}}\frac{1}{\left(\tilde s-\frac{\sigma-s}{\rho^2}\right)^{\frac{2+(2k-1)+n^2d}{2}}  } \biggr\}\nonumber\\
\times \sup_{\gamma\in [0,1]}\exp\left(-c(t-\sigma)\left|\T_{t-\sigma}^{-1}\biggl( \gamma \tilde  \btheta_{t,s}^{t,\y}(\x)+(1-\gamma)\tilde \btheta_{t,\sigma}^{t,\y}(\bxi)  -\y\biggr)\right|^2 \right).\label{BIG_BD_I2}
\end{eqnarray}
Thus, if $|\tilde s-\frac{\sigma-s}{\rho^2}|^{1/2}\ge \tilde c$ , we get recalling \eqref{COND_CZ_LOIN}, \eqref{DEF_J1}, \eqref{DEF_J2}:
\begin{equation*}
\begin{split}
I_2(\sigma,t,\x,\bxi,\y)\le C\Bigl\{\sum_{k=1}^{n} \frac{|(\btheta_{\sigma,s}(\x)-\bxi)_k|}{\rho^{2+(2k-1)+n^2d}}+\frac{1}{\rho^{2-\eta+n^2d}}\Bigr\}\\
\le C\Bigl\{\sum_{k=1}^{n}\frac{\alpha^{2k-1}}{\rho^{2+(2k-1)+n^2d}}+\frac{1}{\rho^{2-\eta+n^2d}} \Bigr\},
\end{split}
\end{equation*}
using \eqref{DEF_J2} for the last inequality.
Recall now from \eqref{DEF_J1}, \eqref{DEF_J2} and \eqref{DEF_RESTES} that 
\begin{eqnarray}\bigl(\gamma \tilde \btheta_{t,s}^{t,\y}(\x) +(1-\gamma)\tilde \btheta_{t,\sigma}^{t,\y}(\bxi)\bigr)-\y\nonumber\\
= 
\gamma 
\tilde \btheta_{t,s}^{t,\y}(\x)
+
(1-\gamma)\tilde \btheta_{t,\sigma}^{t,\y}(\btheta_{\sigma,s}(\x)+\alpha^{-1}\T_{\alpha^2}\bar \x)-\y\nonumber\\
\overset{\eqref{AFFINE}}{=} 
\gamma 
\tilde \btheta_{t,s}^{t,\y}(\x)
+(1-\gamma)\biggl\{\tilde \btheta_{t,\sigma}^{t,\y}(\btheta_{\sigma,s}(\x))+\tilde \gR^{t,\y}(t,\sigma)\alpha^{-1}\T_{\alpha^2}\bar \x\biggr\}-\y \nonumber\\
=\btheta_{t,s}(\x)-\y+\gamma (\tilde \btheta_{t,s}^{t,\y}(\x)-\btheta_{t,s}(\x)) +(1-\gamma)(\tilde \btheta_{t,\sigma}^{t,\y}(\btheta_{\sigma,s}(\x))-\btheta_{t,s}(\x))\nonumber\\
+(1-\gamma)\tilde \gR^{t,\y}(t,\sigma)\alpha^{-1}\T_{\alpha^2}\bar \x\nonumber\\
:=-\rho^{-1}\T_{\rho^2}\{\tilde \x+\gamma \cR_{t,s}^\rho(\x,\y)+(1-\gamma)\cR_{t,\sigma}^\rho(\btheta_{\sigma,s}(\x),\y)\}+(1-\gamma)\tilde \gR^{t,\y}(t,\sigma)\alpha^{-1}\T_{\alpha^2}\bar \x,\nonumber\\
\label{I}
\end{eqnarray} 
where $(\bar s, \bar \x)\in \Sigma_1 $. Observe that, from \eqref{CTR_R} in Lemma \ref{LEMME_STAB} we have for all $j\in \leftB 1,n\rightB $,
\begin{eqnarray}
\label{II}
| (\cR_{t,s}^\rho(\x,\y))_j|+|(\cR_{t,\sigma}^\rho(\btheta_{\sigma,s}(\x),\y))_j|\le C_{\ref{LEMME_STAB}}n(2
{\Lambda}^\eta+(t-\sigma)+(t-s))|\tilde \x_{j_0}|.\nonumber\\
\end{eqnarray}
On the other hand, from the scaling Lemma \ref{scaling_lemma} we obtain that 
\begin{equation*}
(t-\sigma)^{1/2}\T_{t-\sigma}^{-1}\tilde \gR^{t,\y}(t,\sigma) \alpha^{-1}\T_{\alpha^2}\bar \x={\bar  \gR}^{\sigma,t,(t,\y)}_1 (t-\sigma)^{1/2}\T_{t-\sigma}^{-1}\alpha^{-1}\T_{\alpha^2}\bar \x
\end{equation*}
with $| {\bar  \gR}^{\sigma,t,(t,\y)}_1|\le \hat C:=\hat C(T,\A{A}) $. Thus, recalling that from the structure of the linearized system the resolvent is subdiagonal (see \eqref{eq:F:240409:3}, \eqref{AFFINE}), we derive for all $j\in \leftB1,n\rightB$:
\begin{eqnarray}
(t-\sigma)^{1/2}|(\T_{t-\sigma}^{-1}\tilde \gR^{t,\y}(t,\sigma) \alpha^{-1}\T_{\alpha^2}\bar \x)_j|=(t-\sigma)^{1/2}|({\bar  \gR}^{\sigma,t,(t,\y)}_1 \T_{t-\sigma}^{-1}\alpha^{-1}\T_{\alpha^2}\bar \x)_j|\nonumber\\
\le  \hat C \bsum{i=1}^{j}\left( \frac{\alpha^2}{t-\sigma}\right)^{i-1/2}|\bar \x_i|
\le \hat C \sum_{i=1}^j\left(\frac{\alpha^2}{\rho^2\{\tilde s-\frac{\sigma-s}{\rho^2}\}}\right)^{i-1/2}|\bar \x|\nonumber\\
\le \hat C \left(\tilde s-\frac{\sigma-s}{\rho^2} \right)^{-(j-1/2)} \sum_{i=1}^j\left(\frac{\alpha^2}{\rho^2}\right)^{i-1/2}|\bar \x|\label{III}
\end{eqnarray}
as soon as $\tilde c\le 1 $ and $|\tilde s-\frac{\sigma-s}{\rho^2}|^{1/2}\le \tilde c $. In that case, using \eqref{CTR_GRD}, \eqref{I}, \eqref{II}, \eqref{III} we then derive that: 
\begin{eqnarray*}
(t-\sigma)|\T_{t-\sigma}^{-1}(\gamma \tilde \btheta_{t,s}^{t,\y}(\x)+(1-\gamma)\tilde \btheta_{t,\sigma}^{t,\y}(\bxi)-\y)|^2
\\
\ge   \left(\tilde s-\frac{\sigma-s}{\rho^2} \right)^{-(2j_0-1)}\biggl(\frac 12 |\tilde \x_{j_0}|^2- 2\biggl\{(|(\cR_{t,s}^\rho(\x,\y))_{j_0}|+|(\cR_{t,\sigma}^\rho(\btheta_{\sigma,s}(\x),\y))_{j_0}|)^2\\
+ \hat C^2\left\{\sum_{i=1}^{j_0}\left(\frac{\alpha^2}{\rho^2}\right)^{i-1/2}\right\}^2 |\bar \x|^2 \biggr\} \biggr)\\ 
\ge  \left(\tilde s-\frac{\sigma-s}{\rho^2} \right)^{-(2j_0-1)}\left( \frac 12|\tilde \x_{j_0}|^2-2n^2[C_{\ref{LEMME_STAB}}^2(2 
{\Lambda^\eta}+2(t-\sigma\wedge s))^2 +\hat C^2\my_c^{-2}]  \right) 
\\
\ge  \left(\tilde s-\frac{\sigma-s}{\rho^2} \right)^{-(2j_0-1)}\bar c, 
\end{eqnarray*}
where $\bar c>0$ for $T,\Lambda $ small enough and a sufficiently large $\my_c$. 
Plugging this last inequality in \eqref{BIG_BD_I2}, we thus obtain the global bound:
\begin{equation}
\label{CTR_I2}
\exists C_2:=C_2(T,\A{A})>0, \ I_2(\sigma,t,\x,\bxi,\y) \le C_2\left( \sum_{k=1}^{n}\frac{\alpha^{2k-1}}{\rho^{2+(2k-1)+n^2d}}+\frac{1}{\rho^{2-\eta+n^2d}}\right).
\end{equation}
Plugging \eqref{CTR_I1}, \eqref{CTR_I2} into \eqref{DECOUP_INTEGRAND}, and recalling as well that \eqref{COND_CZ_LOIN} holds, gives the point.

Let us now turn to the estimates concerning the \textit{adjoint} kernel $k_{i,j}^{d,*} $. Some additional contributions need to be taken into account. Namely, when investigating the difference
\begin{eqnarray*}
D_{i,j}^{d}((s,\x),(\sigma,\bxi),(t,\y))&:=&k_{i,j}^{d,*}(s,t,\x,\y)-k_{i,j}^{d,*}(\sigma,t,\bxi,\y)\\
&=&k_{i,j}^{d}(t,s,\y,\x)-k_{i,j}^{d}(t,\sigma,\y,\bxi),
\end{eqnarray*}
 we are led to consider the linearized systems $\tilde \btheta_{s,t}^{s,\x}(\y), \tilde \btheta_{\sigma,t}^{\sigma,\bxi}(\y)$. 
Define now  for all $(s,t,u,\x)\in [-T,T]^3\times \R^{nd} $ : 
 \begin{eqnarray}
\label{DEF_H}
\tilde \H^{s,\x}(t,u)& :=\tilde \gR^{s,\x}(t,s) \tilde  \K^{s,\x}(t,u) \tilde \gR^{s,\x}(t,s)^*,\nonumber\\
\tilde  \K^{s,\x}(t,u) &:=\int_t^u \tilde \gR^{s,\x}(s,v) B \varsigma(v)B^*[\tilde \gR^{s,\x}(s,v)]^* dv.
\end{eqnarray}
\begin{REM}
\label{RQ_SUR_DET_HK}
Let us note that for $u=s$ we have $\tilde \K^{s,\x}(t,s)=\tilde \K^\x(t,s) $ introduced after \eqref{DEF_KERN}.
Observe also, from the above definition and the specific structure of the resolvent (see equations \eqref{eq:F:240409:3}-\eqref{AFFINE}) that we actually have ${\rm det}(\tilde \gR^{s,\x}(t,s))=1 $ and therefore ${\rm det} (\tilde \H^{s,\x}(t,u))={\rm det} (\tilde  \K^{s,\x}(t,u))$.
\end{REM}
From the definition in \eqref{DEF_H}  and rewriting  \eqref{PULL_BACK} in the current variables, the exponential bounds write:
 \begin{eqnarray*}
\langle \tilde \K^{\x}(t,s)^{-1} (\tilde \btheta_{s,t}^{s,\x}(\y)-\x),\tilde \btheta_{s,t}^{s,\x}(\y)-\x \rangle\\
=\langle \tilde \gR^{s,\x}(s,t)^*\tilde \K^{\x}(t,s)^{-1} \tilde \gR^{s,\x}(s,t)(\y-\btheta_{t,s}(\x)),\y-\btheta_{t,s}(\x) \rangle\\
=\langle \tilde \H^{s,\x}(t,s)^{-1}(\y-\btheta_{t,s}(\x)),\y-\btheta_{t,s}(\x)\rangle,\\
\langle \tilde \K^{\bxi}(t,\sigma)^{-1} (\tilde \btheta_{\sigma,t}^{\sigma,\bxi}(\y)-\bxi),\tilde \btheta_{\sigma,t}^{\sigma,\bxi}(\y)-\bxi \rangle
=\langle \tilde \H^{\sigma,\bxi}(t,\sigma)^{-1}(\y-\btheta_{t,\sigma}(\bxi)),\y-\btheta_{t,\sigma}(\bxi)\rangle.
\end{eqnarray*}
Introducing, for all $(s,t,u,\x)\in [-T,T]^3\times \R^{nd},\ \z\in \R^{nd} $,
\begin{eqnarray}
\label{DEF_CHECK_K}
\check k^{s,\x,u,t}(\z):=\I_{s-t>0}\left\{ -[\tilde \H^{s,\x}(t,u)^{-1}]_{1,1}+[\tilde \H^{s,\x}(t,u)^{-1} \z]_1^{\otimes 2}\right\}\nonumber\\
\frac{1}{(2\pi)^{nd/2}\det(\tilde \H^{s,\x}(t,u))^{1/2}}\times \exp\left(-\frac 12 \langle \tilde \H^{s,\x}(t,u)^{-1}\z,\z\rangle \right),
\end{eqnarray}
we can rewrite:
\begin{eqnarray*}
|D_{i,j}^{d}((s,\x),(\sigma,\bxi),(t,\y))|=|\check k_{i,j}^{s,\x,s,t}(\y-\btheta_{t,s}(\x))-\check k_{i,j}^{\sigma,\bxi,\sigma,t}(\y-\btheta_{t,\sigma}(\bxi))|\\
\le |\check k_{i,j}^{s,\x,s,t}(\y-\btheta_{t,s}(\x))-\check k_{i,j}^{s,\x,\sigma,t}(\y-\btheta_{t,s}(\x))|\\
+ |\check k_{i,j}^{s,\x,\sigma,t}(\y-\btheta_{t,s}(\x))-\check k_{i,j}^{\sigma,\bxi,\sigma,t}(\y-\btheta_{t,s}(\x))|\\
+|\check k_{i,j}^{\sigma,\bxi,\sigma,t}(\y-\btheta_{t,s}(\x))-\check k_{i,j}^{\sigma,\bxi,\sigma,t}(\y-\btheta_{t,\sigma}(\bxi))|=:\sum_{l=1}^3 |\{D_{i,j}^{d}((s,\x),(\sigma,\bxi),(t,\y))\}_l|.
\end{eqnarray*}
Now the terms $|\{D_{i,j}^{d}((s,\x),(\sigma,\bxi),(t,\y))\}_1|$ and $|\{D_{i,j}^{d}((s,\x),(\sigma,\bxi),(t,\y))\}_3| $ respectively involve time and space sensitivities when the freezing parameters in the covariance matrix are fixed. 
Those contributions can therefore be investigated as terms $I_1 $ and $I_2 $ in \eqref{CTR_I1}, \eqref{CTR_I2}. Once again, the previous splitting is associated w.l.o.g. to the case $|\sigma-s|\le K|t-\sigma|$, see also Remark \ref{DECOUP_ALTERNATIVE}. The term $ |\{D_{i,j}^{d}((s,\x),(\sigma,\bxi),(t,\y))\}_2|$ involves two different covariance matrices observed at the same time but that are respectively associated with the freezing points $(s,\x) $ and $(\sigma,\bxi) $ in the linearization of \eqref{DET_SYST}. To analyze this difference we proceed as in the proof of Lemma 2.4 in \cite{meno:10}. Namely, using \eqref{DEF_H} and the Scaling Lemma \ref{scaling_lemma}, we rewrite:
\begin{eqnarray*}
 \tilde \H^{s,\x}(t,\sigma) &=&\int_{t}^\sigma \tilde \gR^{s,\x}(t,u) B\varsigma (u)B^*[\tilde \gR^{s,\x}(t,u)]^* du\\
                                       &=&\T_{\sigma-t}\int_t^\sigma \bar \gR^{t,\sigma, (s,\x)}_{\frac{t-u}{\sigma-t}} \T_{\sigma-t}^{-1}B\varsigma(u)B^*\T_{\sigma-t}^{-1}[\bar \gR^{t,\sigma,(s,\x)}_{\frac{t-u}{\sigma-t}}]^* du\T_{\sigma-t}\\
                                                                              &=&(\sigma-t)^{-1}\T_{\sigma-t}\left[\frac{1}{\sigma-t}\int_t^\sigma \bar \gR^{t,\sigma,(s,\x)}_{\frac{t-u}{\sigma-t}} B\varsigma(u)B^*[\bar \gR^{t,\sigma,(s,\x)}_{\frac{t-u}{\sigma-t}}]^*du\right]\T_{\sigma-t}.
\end{eqnarray*}
Defining,
\begin{eqnarray}
\bar {\H}^{t,\sigma,(s,\x)}_1:=\left[\frac{1}{\sigma-t}\int_t^\sigma \bar \gR^{t,\sigma, (s,\x)}_{\frac{t-u}{\sigma-t}} B\varsigma(u)B^*[\bar \gR^{t,\sigma,(s,\x)}_{\frac{t-u}{\sigma-t}}]^*du\right],\label{DEF_BARH}
\end{eqnarray} 
yields:                                                                    
\begin{eqnarray*}
 \tilde \H^{s,\x}(t,\sigma) &=& (\sigma-t)^{-1}\T_{\sigma-t}\bar {\H}^{t,\sigma,(s,\x)}_1\T_{\sigma-t}.
\end{eqnarray*}                                                                             
Observe now from Lemma \ref{scaling_lemma} and the non degeneracy assumption on $c$ in \A{A}, that $\bar \H^{t,\sigma,(s,\x)} $ is a bounded uniformly elliptic matrix of $\R^{nd}\otimes \R^{nd} $. 
Similarly,
 \begin{eqnarray}                                                                             
 \tilde \H^{\sigma,\bxi}(t,\sigma)&=&(\sigma-t)^{-1}\T_{\sigma-t}\left[\frac{1}{\sigma-t}\int_t^\sigma \bar \gR^{t,\sigma, (\sigma,\bxi)}_{\frac{t-u}{\sigma-t}} B\varsigma(u)B^*[\bar \gR^{t,\sigma,(\sigma,\bxi)}_{\frac{t-u}{\sigma-t}}]^*du\right]\T_{\sigma-t}\nonumber \\
&=:&(\sigma-t)^{-1}\T_{\sigma-t} {\bar \H}^{t,\sigma, (\sigma,\bxi)}_1\T_{\sigma-t},\label{DEF_BARH_2}
 \end{eqnarray}
where $ {\bar \H}^{t,\sigma,(\sigma,\bxi)}_1$ is again a uniformly elliptic bounded matrix on $\R^{nd}\otimes \R^{nd}$. Thus,
\begin{eqnarray}
\label{AVANT_SCALE}
\langle(\tilde \H^{s,\x}(t,\sigma)-\tilde \H^{\sigma,\bxi}(t,\sigma))(\y-\btheta_{t,s}(\x)),\y-\btheta_{t,s}(\x) \rangle\nonumber\\
=\langle ( {\bar \H}^{t,\sigma,(s,\x)}_1- {\bar \H}^{t,\sigma,(\sigma,\bxi)}_1)( (\sigma-t)^{-1/2}\T_{\sigma-t}(\y-\btheta_{t,s}(\x)) ),(\sigma-t)^{-1/2}\T_{\sigma-t}(\y-\btheta_{t,s}(\x)) \rangle.\nonumber\\
\end{eqnarray}
We now want to control the difference $ ( {\bar \H}^{t,\sigma,(s,\x)}_1- {\bar \H}^{t,\sigma,(\sigma,\bxi)}_1)$ in \eqref{AVANT_SCALE}. 
From the definitions in \eqref{DEF_BARH} and \eqref{DEF_BARH_2}:
\begin{eqnarray}
 {\bar \H}^{t,\sigma,(s,\x)}_1- {\tilde \H}^{t,\sigma,(\sigma,\bxi)}_1
&=&(\sigma-t)^{-1}\int_t^\sigma \left\{ \bar \gR_{\frac{t-u}{\sigma-t}}^{t,\sigma,(s,\x)}B\varsigma(u)B^*  [\bar \gR_{\frac{t-u}{\sigma-t}}^{t,\sigma,(s,\x)}]^*\right.\nonumber\\
&&-\left.\bar \gR_{\frac{t-u}{\sigma-t}}^{t,\sigma,(\sigma,\bxi)}B\varsigma(u)B^*[ \bar \gR_{\frac{t-u}{\sigma-t}}^{t,\sigma,(\sigma,\bxi)}]^* \right\}du,\nonumber\\
| {\bar \H}^{t,\sigma,(s,\x)}_1- {\bar \H}^{t,\sigma,(\sigma,\bxi)}_1|&\le & C(\sigma-t)^{-1}\int_t^\sigma \left|\bar \gR_{\frac{t-u}{\sigma-t}}^{t,\sigma,(s,\x)}-\bar \gR_{\frac{t-u}{\sigma-t}}^{t,\sigma,(\sigma,\bxi)}\right| du .\nonumber\\ \label{DELTA_BAR_H}
\end{eqnarray}
Let us now write, still from Lemma \ref{scaling_lemma} and \eqref{DYN_RES}:
\begin{eqnarray*}
|\bar \gR_{\frac{t-u}{\sigma-t}}^{t,\sigma,(s,\x)}-\bar \gR_{\frac{t-u}{\sigma-t}}^{t,\sigma,(\sigma,\bxi)}|=|\T_{\sigma-t}^{-1}(\tilde \gR^{s,\x}(t,u)-\tilde \gR^{\sigma,\bxi}(t,u))\T_{\sigma-t}|=\\
\biggl|\T_{\sigma-t}^{-1}\int_t^u \biggl\{\tilde \gR^{s,\x}(t,v)D\gF(v,\btheta_{v,s}(\x))-\tilde \gR^{\sigma,\bxi}(t,v)D\gF(v,\btheta_{v,\sigma}(\bxi))\biggr\} dv\T_{\sigma-t}\biggr|\\
\le \biggl|\int_t^u (\bar \gR_{\frac{t-v}{\sigma-t}}^{t,\sigma,(s,\x)}-\bar \gR_{\frac{t-v}{\sigma-t}}^{t,\sigma,(\sigma,\bxi)})\left\{ \T_{\sigma-t}^{-1}D\gF(v,\btheta_{v,s}(\x))  \T_{\sigma-t}\right\}dv\\+
\int_t^u \bar \gR_{\frac{t-v}{\sigma-t}}^{t,\sigma,(\sigma,\bxi)} \T_{\sigma-t}^{-1}(D\gF(v,\btheta_{v,s}(\x))-D\gF(v,\btheta_{v,\sigma}(\bxi))) \T_{\sigma-t}  dv\biggr|\le C |\btheta_{\sigma,s}(\x)-\bxi|^\eta,
\end{eqnarray*}
where $C:=C(\A{A}) $, using the smoothness conditions assumed in \A{S}, the subdiagonal structure of $D\gF$ (see eq. \eqref{eq:F:240409:3}), the Liscphitz property of the flow and Gronwall's Lemma for the last inequality.
From the above equation and \eqref{DELTA_BAR_H}, \eqref{AVANT_SCALE}, we thus derive:
\begin{equation}
\label{CTR_COV_SCALEE}
\begin{split}
&\langle(\tilde \H^{s,\x}(t,\sigma)-\tilde \H^{\sigma,\bxi}(t,\sigma))(\y-\btheta_{t,s}(\x)),\y-\btheta_{t,s}(\x) \rangle\\
&\le C|\btheta_{\sigma,s}(\x)-\bxi|^\eta |(\sigma-t)^{-1/2}\T_{\sigma-t}(\y-\btheta_{t,s}(\x))|^2,\ C:=C(\A{A}). 
\end{split}
\end{equation}
Because of the non-degeneracy of $c$, the inverse matrices $ ( {\bar \H}^{t,\sigma,(s,\x)}_1)^{-1},({\bar \H}^{t,\sigma,(\sigma,\bxi)}_1)^{-1}$ 
have the same spatial H\"older regularity. Indeed, up to a change of coordinates one can assume that one of the two matrices is diagonal at the considered point and that the other has dominant diagonal if $|\btheta_{\sigma,s}(\x)-\bxi| $ is small enough (depending on the ellipticity bounds in \A{A} and the dimension). This reduces to the scalar case. Hence,
\begin{eqnarray}
\label{CTR_INV_COV_SCALEE}
\langle ((\tilde \H^{s,\x}(t,\sigma))^{-1}-(\tilde \H^{\sigma,\bxi}(t,\sigma))^{-1})(\y-\btheta_{t,s}(\x)),\y-\btheta_{t,s}(\x) \rangle\notag\\
=\langle ( ( {\bar \H}^{t,\sigma,(s,\x)}_1)^{-1}-( {\bar \H}^{t,\sigma,(\sigma,\bxi)}_1)^{-1})( (\sigma-t)^{1/2}\T_{\sigma-t}^{-1}(\y-\btheta_{t,s}(\x)) ),(\sigma-t)^{1/2}\T_{\sigma-t}^{-1}(\y-\btheta_{t,s}(\x)) \rangle\notag\\
\le C |\btheta_{\sigma,s}(\x)-\bxi|^\eta |(\sigma-t)^{1/2}\T_{\sigma-t}^{-1}(\y-\btheta_{t,s}(\x))|^2.
\end{eqnarray}
The difference of the determinants can be investigated similarly.
 We therefore derive:
 \begin{eqnarray*}
|\{D_{i,j}^{d}((s,\x),(\sigma,\bxi),(t,\y))\}_2|\le C \frac{{\mathbf d}((s,\x),(\sigma,\bxi))^\eta}{|t-\sigma|}\bar q_c(\sigma,t,\btheta_{t,s}(\x)-\y).
 \end{eqnarray*}
which can be analyzed similarly to $I_2$ (see equation \eqref{DECOUP_INTEGRAND} and page \pageref{I}) in the previous proof and yields:
 \begin{equation}
 \label{CT_RESOLV}
|\{D_{i,j}^{d}((s,\x),(\sigma,\bxi),(t,\y))\}_2|\le C \frac{{\mathbf d}((s,\x),(\sigma,\bxi))^\eta}{{\mathbf d}((s,\x),(t,\y))^{2+n^2 d}}\le C \frac{{\mathbf d}((s,\x),(\sigma,\bxi))^\eta}{{\mathbf d}((s,\x),(t,\y))^{2+\eta+n^2d}},
 \end{equation}
 recalling that ${\mathbf d}((s,\x),(t,\y)\le \Lambda\le 1$. This gives points \textit{ii)} and \textit{iii)}.
The cancellation property \textit{iv)} is the more subtle to derive. Let us first prove:
\begin{eqnarray*}
\sup_{\epsilon>0}|\int_{{\mathbf d}((t,\y),(s,\x))>\epsilon} k_{i,j}^{d}(s,t,\x,\y)dtd\y|
<+\infty.
\end{eqnarray*}
Write with the notation of \eqref{DEF_CHECK_K}:
\begin{eqnarray*}
\int_{{\mathbf d}((t,\y),(s,\x))>\epsilon} k_{i,j}^{d}(s,t,\x,\y)dtd\y
=\int_{  \rho(t-s,\x-\btheta_{s,t}(\y))\in ( \epsilon,\delta)}\check k_{i,j}^{t,\y,t,s}(\x-\btheta_{s,t}(\y))dtd\y\\
+\int_{ \rho(t-s,\x-\btheta_{s,t}(\y))\in ( \delta,2\delta)}\check k_{i,j}^{t,\y,t,s}(\x-\btheta_{s,t}(\y))dtd\y=:O_1^\epsilon+O_2^\epsilon.
\end{eqnarray*}
Recall now from Section \ref{SUBS_EST_CZ} that the kernel involves a cut-off that localizes the singularities. Hence, it is easily seen from \eqref{CTR_HD} and the computations following that equation that $|O_2^\epsilon|\le C:=C(T,\A{A},\delta)  $. Let us now focus on $O_1^\epsilon $.
Set $\z:=\x-\btheta_{s,t}(\y) $ that yields $d\z=\det({\rm Jac}_{\btheta_{s,t}(\y)}) d\y $ where for $T$ small enough $\det({\rm Jac}_{\btheta_{s,t}(\y)})=1+O(|t-s|) $. 
We thus derive:
\begin{eqnarray}
O_1^\epsilon=\int_{ \rho(t-s,\z)\in (\epsilon,\delta)} \exp(-\frac12\langle (\tilde \H^{t,\btheta_{t,s}(\x-\z)}(s,t))^{-1}\z,\z\rangle )\label{DEF_O1_EPS}\\
\times  P_{i,j}((\tilde \H^{t,\btheta_{t,s}(\x-\z)}(s,t))^{-1},\z)\frac{1}{(2\pi)^{nd/2}\det(\tilde \H^{t,\btheta_{t,s}(\x-\z)}(s,t))^{1/2}} dt d\z+O(1),\nonumber
\end{eqnarray} 
denoting for $({\mathbf A},\z)\in \R^{nd}\otimes \R^{nd}\times \R^{nd},\ P_{i,j}({\mathbf A},\z):=\left\{[{\mathbf A} ]_{1,1}+[{\mathbf A}\z]_{1}^{\otimes 2}\right\}_{ij} $.

To conclude the analysis we need an additional regularization of the drift. To this end, we now introduce for given $((s,\x),(t,\z))\in S :$
\begin{equation}
\begin{split}
q^{s,t,\x,(t-s)*}(\z):=\left\{ -[\tilde \H^{t,\btheta_{t,s}(\x-\z),(t-s)*}(s,t)^{-1}]_{1,1}+[\tilde \H^{t,\btheta_{t,s}(\x-\z),{(t-s)*}}(s,t)^{-1} \z]_1^{\otimes 2}\right\}\\
\frac{1}{(2\pi)^{nd/2}\det(\tilde \H^{t,\btheta_{t,s}(\x-\z),{(t-s)*}}(s,t))^{1/2}}\times \exp\left(-\frac 12 \langle \tilde \H^{t,\btheta_{t,s}(\x-\z),(t-s)*}(s,t)^{-1}\z,\z\rangle \right), 
\end{split}
\label{DEF_NOY_REG}
\end{equation}
where 
\begin{equation*}
\begin{split}
 \tilde \H^{t,\btheta_{t,s}(\x-\z),{(t-s)*}}(s,t)& :=
 \int_s^t  \tilde \gR^{t,\btheta_{t,s}(\x-\z),(t-s)*}_{s,u}B\varsigma(u)B^*[\tilde \gR^{t,\btheta_{t,s}(\x-\z),(t-s)*}_{s,u}]^* du,\\
 \partial_u \tilde \gR^{t,\btheta_{t,s}(\x-\z),(t-s)*}_{s,u}
 &=
 -\tilde \gR^{t,\btheta_{t,s}(\x-\z),(t-s)*}_{s,u} D\gF^{(t-s)*}(u,\btheta_{u,s}(\x-\z)),\ u\in [s,t],\\
\tilde  \gR^{t,\btheta_{t,s}(\x-\z),(t-s)*}_{s,s}&=I_{nd\times nd},\\ 
 D\gF^{(t-s)*}(u,\btheta_{u,s}(\x-\z))
& :=
 D\gF(u,\btheta_{u,s}(\cdot))* \zeta_{t-s}(\x-\z),\\
 \end{split}
\end{equation*}
where the last $*$ stands for the spatial convolution and $\zeta_{t-s}:\R^{nd}\rightarrow [0,1] $ is a smooth mollifyer s.t. for $\y\in \R^{nd} $, $\zeta_{t-s}(\y)=1$ if $|\y|\le |t-s|^{1/8} $ and $0$ if $|\y|\ge 2|t-s|^{1/8} $. Hence, there exists  $C>0$ s.t. $|{\mathbf D}_{\y_i} \zeta(\y)|\le C|t-s|^{-1/8},\forall i\in \leftB 1,d\rightB,\ |{\mathbf D}_{\y_i\y_j}^2 \zeta(\y)|\le C|t-s|^{-1/4},\ \forall (i,j)\in \leftB 1, d\rightB^2$.

Under \A{A}, one easily gets that the \textit{mollified} matrix $\tilde \H^{t,\btheta_{t,s}(\x-\z),{(t-s)*}}(s,t) $ satisfies the good scaling property \eqref{GSP}. Computations similar to those leading from \eqref{AVANT_SCALE} to \eqref{CTR_COV_SCALEE}, \eqref{CTR_INV_COV_SCALEE} also yield for all $\z\in \R^{nd} $:
\begin{eqnarray}
|\langle [ (\tilde \H^{t,\btheta_{t,s}(\x-\z),{(t-s)*}}(s,t))^{-1}- (\tilde \H^{t,\btheta_{t,s}(\x-\z)}(s,t))^{-1}] \z,\z\rangle|
\nonumber \\
\le C |t-s|^{\eta/8}|(t-s)^{1/2}\T_{t-s}^{-1}\z|^2,\label{CTR_DIFF_REG}\\
\langle [{\mathbf D}_{\z_1^i}(\tilde \H^{t,\btheta_{t,s}(\x-\z),{(t-s)*}}(s,t))^{-1}+{\mathbf D}_{\z_1^i\z_1^j}^2(\tilde \H^{t,\btheta_{t,s}(\x-\z),{(t-s)*}}(s,t))^{-1}]\z ,\z\rangle \nonumber \\
\le C |t-s|^{-1/4}|(t-s)^{1/2}\T_{t-s}^{-1}\z|^2, \ \forall (i,j)\in \leftB 1,d\rightB^2 \label{CTR_DIFF_SENSI}.
\end{eqnarray}
Similar controls would hold for the difference and sensitivities of the determinants.
Precisely, we observe that the sensitivities of the \textit{mollified} covariance matrices w.r.t. the freezing parameters induce additional integrable singularities.
Write then from \eqref{DEF_O1_EPS}, \eqref{DEF_NOY_REG}:
\begin{eqnarray*}
O_1^\epsilon=\int_{  \rho(t-s,\x-\btheta_{s,t}(\y))\in ( \epsilon,\delta)} [q^{s,t,\x,0}(\z)-q^{s,t,\x,(t-s)*}(\z)]dt d\z\\
+\int_{  \rho(t-s,\x-\btheta_{s,t}(\y))\in ( \epsilon,\delta)}q^{s,t,\x,(t-s)*}(\z)+O(1)
=:O_{11}^\epsilon+O_{12}^\epsilon+O(1),
\end{eqnarray*}
denoting with a slight abuse of notation by $q^{s,t,\x,0}(\z) $ the term in \eqref{DEF_NOY_REG} when there is no convolution.
From the definition of $q^{s,t,\x,\cdot}$ in \eqref{DEF_NOY_REG}, the good scaling property property \eqref{GSP} satisfied by both $\tilde \H^{t,\btheta_{t,s}(\x-\z)}, \tilde \H^{t,\btheta_{t,s}(\x-\z),*(t-s)}$ and equation \eqref{CTR_DIFF_REG},  we get that the contribution $O_{11}^\epsilon=O(1)$. Indeed, the singularities are integrable for that term. On the other hand
\begin{eqnarray*}
q^{s,t,\x,(t-s)*}(\z)\\
=D_{\z_1^i\z_1^j}^2\biggl\{\frac{\exp(-\frac12\langle (\tilde \H^{t,\btheta_{t,s}(\x-\z),*(t-s)}(s,t))^{-1}\z,\z\rangle )}{(2\pi)^{nd}\det(\tilde \H^{t,\btheta_{t,s}(\x-\z),*(t-s)}(s,t))^{1/2}}\biggr\}+{\mathcal R}_{i,j}(s,t,\x,\z).
\end{eqnarray*}
The remainder term ${\mathcal R}_{i,j} $ gathers the contributions deriving from the sensitivities of $(\tilde \H^{t,\btheta_{t,s}(\x-\z),(t-s)*}(s,t))^{-1} $ w.r.t. $\z_1^i,\z_1^j $. Actually, the regularization of the coefficient $D\gF$ is  just required here to differentiate the dynamics of the resolvent.
From \eqref{CTR_DIFF_SENSI}  it can be checked that:
\begin{eqnarray*}
|{\mathcal R}_{i,j}(s,t,\x,\z)|\le \frac{C}{(t-s)^{n^2d/2+3/4}}\exp(-(t-s)|\T_{t-s}^{-1}\z|^2),\ C:=C(\A{A},T).
\end{eqnarray*} 
We thus write
\begin{eqnarray*}
O_{12}^\epsilon&=&\int_{ \rho(t-s,\z)\in (\epsilon,\delta)}D_{\z_1^i\z_1^j}^2\biggl\{ q_{\H,*}(s,t,\z)\biggr\}dt d\z+O(1),\\
q_{\H,*}(s,t,\z)&:=&\frac{\exp(-\frac12\langle (\tilde \H^{t,\btheta_{t,s}(\x-\z),(t-s)*}(s,t))^{-1}\z,\z\rangle )}{(2\pi)^{nd/2}\det(\tilde \H^{t,\btheta_{t,s}(\x-\z),(t-s)*}(s,t))^{1/2}}.
\end{eqnarray*}
By the divergence theorem:
\begin{eqnarray*}
|O_{12}^\epsilon|\le \sum_{\beta\in\{\epsilon,\delta \}}\biggl | \int_{\rho(t-s, \z)= \beta}\partial_{\z_1^i}q_{\H,*}(s,t,\z) n_jd\nu
((t-s), \z)\biggr|.
\end{eqnarray*}
Now, from the metric homogeneity (see Remark \ref{Homogeneity}), the good scaling property \eqref{GSP} that is valid for $\tilde \H^{t,\btheta_{t,s}(\x-\z),(t-s)*} $, it can be shown, changing variables similarly to \eqref{DEF_J1}, that for $\epsilon $ small enough:
\begin{equation}
|O_{12}^\epsilon|\le C\int_{\rho(|\tilde s|,\bar \z)=1} \frac{d\nu(\tilde s,\bar \z)}{|\tilde s|^{n^2 d/2+1/2}}\exp\left(-C|\tilde s||\T_{|\tilde s|}^{-1}\bar \z |^2\right)<+\infty, \label{CTR_M_C}
\end{equation}
and that $O_{12}^\epsilon$ admits a limit when $\epsilon\rightarrow 0 $. 
From \eqref{CTR_M_C} 
we thus get point \textit{iv)} for the kernel $ k_{i,j}^{d}$.
The proof of the cancellation property for the adjoint kernel can be proved similarly exploiting the equivalence of the ``forward" and ``backward" distance on compact sets (see Proposition \ref{PROP_QM}).  \hfill $\square$

\bibliographystyle{alpha}
\bibliography{bibli}

\end{document}

%% file: macro.tex
\def\I{{\rm{I\!\!I}}}

\def\Z{{{\mathbb Z}}}

\newcommand \A[1]{{\bf (#1)}}

\def\F{{\mathcal F}}

\def\R{{\mathbb{R}} }
\def\N{{\mathbb{N}} }
\def\E{{\mathbb{E}}  }

\def\P{{\mathbb{P}}  }
\def\Q{{\mathbb{Q}}  }

\def\I{{\mathbb{I}}}
\def\det{{\rm{det}}}
\def\tr{{\rm{tr}}}
\def\bint#1^#2{\displaystyle{\int_{#1}^{#2}}}
\def\bsum#1^#2{\displaystyle{\sum_{#1}^{#2}}}

\def\xdt_#1{X_#1(\Delta t)}

\newtheorem{THM}{Theorem}[section]

\newtheorem{PROP}{Proposition}[section]

\newtheorem{LEMME}{Lemma}[section]
\newtheorem{REM}{Remark}[section]
\newcommand{\mysection}{\setcounter{equation}{0} \section}